\documentclass[submission]{dmtcs-episciences}

\usepackage[T1]{fontenc}
\usepackage{amsfonts,amsmath,amssymb,amscd}
\usepackage{floatflt}
\usepackage{pgf}
\usepackage{tikz}
\usetikzlibrary{calc,shapes,positioning,backgrounds,petri,trees,arrows,automata,decorations.pathreplacing}
\tikzstyle{every node}=[ellipse, minimum width=1em, minimum height=1em,inner sep=0pt]
\tikzstyle{accepting}=[tokens=3,accepting by double]
\usepackage{graphicx, subfigure}
\DeclareGraphicsExtensions{.pdf}
\usepackage[round]{natbib}
\usepackage{placeins}

\newcommand{\A}{\mathcal{A}}
\newcommand{\M}{\mathcal{M}}

\DeclareMathOperator{\Imm}{Im}

\newtheorem{theorem}{Theorem}
\newtheorem{lemma}[theorem]{Lemma}
\newtheorem{proposition}[theorem]{Proposition}
\newtheorem{corollary}[theorem]{Corollary}
\newtheorem{definition}[theorem]{Definition}

\begin{document}

\author{Alessandra Cherubini\affiliationmark{1}\begin{NoHyper}\thanks{Email: \texttt{alessandra.cherubini@polimi.it}}\end{NoHyper}
  \and Achille Frigeri\affiliationmark{1}\begin{NoHyper}\thanks{Email: \texttt{achille.frigeri@polimi.it}}\end{NoHyper}
  \and Zuhua Liu\affiliationmark{2}\begin{NoHyper}\thanks{Email: \texttt{liuzuhua@hotmail.com}}\end{NoHyper}}
\title[Composing short 3-compressing words]{Composing short 3-compressing words on a 2 letter alphabet\begin{NoHyper}\thanks{The first two authors are partially supported by PRIN: \lq\lq Automi e linguaggi formali: aspetti matematici e applicativi\rq\rq. The work was completed during the third-named author's visit to Politecnico di Milano which was supported by the CSC (China Scholarship Council), and acknowledges support from National Natural Science Foundation of China, grant No. 11261066.}\end{NoHyper}}
\affiliation{
Dipartimento di Matematica ``F. Brioschi'', Politecnico di Milano\\
Department of Mathematics, Kunming University}
\received{1998-10-14}
\revised{\today}
\accepted{tomorrow}

\maketitle

\begin{abstract}
A finite deterministic (semi)automaton $\A =(Q,\Sigma,\delta)$ is $k$-compressible if there is some word $w\in \Sigma^+$ such that the image of its state set $Q$ under the natural action of $w$ is reduced by at least $k$ states.
Such word, if it exists, is called a $k$-compressing word for $\A$.
A word is $k$-collapsing if it is $k$-compressing for each $k$-compressible automaton.
We compute a set $W$ of short words such that each $3$-compressible automata on a two letter alphabet is $3$-compressed at least by a word in $W$.
Then we construct a shortest common superstring of the words in $W$ and, with a further refinement, we obtain a $3$-collapsing word of length $53$.
Moreover, as previously announced, we show that the shortest $3$-synchronizing word is not $3$-collapsing, illustrating the new bounds $34\leq c(2,3)\leq 53$ for the length $c(2,3)$ of the shortest $3$-collapsing word on a two letter alphabet.
\end{abstract}

\section{Introduction}

Let $\A =(Q,\Sigma,\delta)$ be a finite deterministic complete (semi)automaton with state set $Q$, input alphabet $\Sigma$, and transition function $\delta: \; Q \times \Sigma \to Q$.
For any word $w\in \Sigma^+$, the deficiency of $w$ is the difference between the cardinality of $Q$ and the cardinality of the image of $Q$ under the natural action of $w$.
For a fixed $k\geq 1$, the word $w$ is called $k$-\emph{compressing} for $\A$ if its deficiency is greater or equal to $k$.
An automaton $\A$ is $k$-\emph{compressible}, if there exists a $k$-compressing word for $\A$.
A word $w\in \Sigma^+$ is $k$-\emph{collapsing}, if it is $k$-compressing for every $k$-compressible automaton with input alphabet $\Sigma$.
A word $w\in \Sigma^+$ is called \emph{$k$-synchronizing} if it is $k$-compressing for all $k$-compressible automata with $k+1$ states and input alphabet $\Sigma$.
Obviously each $k$-collapsing word is also $k$-synchronizing.

The concept of $k$-collapsing words arose (under a different name) in the beginning of the 1990s with original motivations coming from combinatorics (\cite{SS}) and from abstract algebra (\cite{PSSSV}).
In \cite{SS} it has been proved that $k$-collapsing words always exist, for any $\Sigma$ and any $k \geq 1$, by means of a recursive construction which gives a $k$-collapsing word whose length is $O(2^{2^k})$.
Better bounds for $c(k,t)$ and $s(k,t)$, the length  of the shortest $k$-collapsing and $k$-synchronizing words respectively, on an alphabet of cardinality $t$ were given in \cite{MVP}.
The bounds for $c(2,t)$ were slightly improved in \cite{Pri} and \cite{CKP}, but the gaps from lower and upper bounds are quite large even for small values of $k$ and $t$.
Exact values of $s(k,t)$ and $c(k,t)$ are known for $k=2$ and $t=2,3$ and are quite far from the theoretical upper bounds (\cite{APV}).
Moreover it is known that $s(3,2)=33$ and that the words $s_{3,2}=ab^2aba^3b^2a^2babab^2a^2b^3aba^2ba^2b^2a$ and its dual $\bar{s}_{3,2}$ are the unique shortest synchronizing words on $\{a,b\}$ (\cite{AP}).
Observing that $s(k,t)\leq c(k,t)$, and applying the construction in \cite{MVP}, one gets $33\leq c(3,2)\leq 154$.

The reader is referred to \cite{APV, C, CKP, MVP} for  references and connections to Theoretical Computer Science and Language Theory.
The paper is organized as follows: in Section \ref{background} we introduce some general concepts about $3$-compressible automata and the main tool we use to study them, \emph{i.e.}, 3-Missing State Automata.
In Section \ref{permutation} we give a complete characterization of proper $3$-compressible automata on a two letter alphabet with a letter acting as a permutation, while in Section \ref{nopermutation} we characterize all proper $3$-compressible automata without permutations.
In Section \ref{zot} we describe how to use the previous characterization to obtain a short $3$-collapsing word, improving the known upper bound for $c(3,2)$, as already announced in \cite{dlt2011}.
Section \ref{conclusion} ends the paper with some considerations about the quest for short $3$-collapsing words in general and the relationship between $3$-synchronizing and $3$-collapsing words, and how our analysis can be exploited to obtain more general results, as already done in \cite{kisi2,kisi}.

\section{Background}\label{background}

Let $\A =(Q,\Sigma,\delta)$ be a finite deterministic complete (semi)automaton with state set $Q$, input alphabet $\Sigma=\{a,b\}$, and transition function $\delta: \; Q \times \Sigma \to Q$.
The action of $\Sigma$ on $Q$ given by $\delta$ extends naturally, by composition, to the action of any word $w\in \Sigma^+$ on $q\in Q$; we denote it by $qw=\delta(q,w)$, while the action of $w$ on the entire state set $Q$ is denoted by $Qw=\{qw|q\in Q\}$.

\begin{definition}
The difference $|Q|-|Qw|$ is called the \emph{deficiency} of the word $w$ with respect to $\A$ and denoted by df$_\A(w)$.
For a fixed $k\geq 1$, a word $w \in \Sigma^+$ is called $k$-\emph{compressing} for $\A$, if df$_\A(w) \geq k$.
An automaton $\A$ is $k$-\emph{compressible}, if there exists a $k$-compressing word for $\A$.
A word $w\in \Sigma^+$ is $k$-\emph{collapsing}, if it is $k$-compressing for every $k$-compressible automaton with input alphabet $\Sigma$.
A word $w$ is called \emph{$k$-synchronizing} if it is $k$-compressing for all $k$-compressible automata with $k+1$ states and input alphabet $\Sigma$.
Obviously each $k$-collapsing word is also $k$-synchronizing.
\end{definition}

Actually, we view the automaton $\A$ as a set of \emph{transformations} on $Q$ induced via $\delta$ and labeled by letters of $\Sigma$, rather than as a standard triple.
Indeed, in order to define an automaton, it is enough just to assign to every letter $a\in\Sigma$ the corresponding transformation $\tau_a:q\rightarrow \delta(q,a)$ on $Q$.
Now, for $a\in\Sigma$, we get df$_\A(a)=|Q|-|\Imm(\tau_a)|$, hence df$_\A(a)= 0$ if and only if $\tau_a$ is a permutation on $Q$.
If df$_\A(a)=m \geq 1$, then there are exactly $m$ different states $y_1,y_2,\ldots,y_m \notin \Imm(\alpha)$, and there are some elements of $Q$ whose images under $\tau_a$ are equal.

\begin{definition}
Let $\mathcal{P}=\{\{x_{1_1},\ldots,x_{j_1}\},\ldots,\{x_{1_r},\ldots,x_{j_r}\}\}$ be a partition of $Q$ (where singleton sets are omitted), and $y_1,\ldots,y_m\in Q$, we say that $\tau_a$ is a \emph{transformation of type} \[
[x_{1_1},\ldots,x_{j_1}]\ldots[x_{1_r},\ldots,x_{j_r}]\backslash y_1,\ldots,y_m
\]
if $\mathcal{P}$ is induced by the kernel of $\tau_a$ and the states $y_1,y_2,\ldots,y_m$ do not belong to $\Imm(\tau_a)$.
\end{definition}

For instance, if $\A$ has at least three states denoted by $1$, $2$ and $3$, the transformation $\tau$ is of type $[1,2]\backslash 3$, if and only if $\tau(1)=\tau(2)$, the preimage of $3$ is empty, and for any $q,q'\notin \{1,2\}$, $\tau(q)=\tau(q')$ if and only if $q=q'$.
So, with an abuse of language, we will write $\tau=[1,2]\backslash 3$ (actually $[1,2]\backslash 3$ is a family of transformations).
Then, in the sequel \emph{we will identify each letter of the input alphabet with its corresponding transformation.}

\begin{definition}
Let $a\in\Sigma$, we say that $a$ is a \emph{permutation letter} if it induces a permutation on the set of states, \emph{i.e.}, it has deficiency $0$.
We assume that permutations on $Q$, viewed as elements of the symmetric group $S_n$ with $|Q|=n$, are written in the factorization in disjoint cycles where sometimes also cycles of length 1 are explicitly written.
So we will write $a=(1)(23)\pi$ to denote that (the permutation induces by) $a$ fixes the state 1, swaps states 2 and 3, and $\pi$ is a permutation that does not act over $\{1,2,3\}$ ($\pi$ is not necessarily a cycle).
\end{definition}

The notion of transformation induced by a letter naturally extends to words, and then the semigroup generated by the transformations of $\A$ consists precisely of the transformations corresponding to words in $\Sigma^+$.
If $\mathcal{A}$ is $k$-compressible, at least one letter of its input alphabet has deficiency greater than $0$.
It is well known that each $k$-collapsing word over a fixed alphabet $\Sigma$ is $k$-\emph{full} (\cite{SS}), \emph{i.e.}, contains each word of length $k$ on the alphabet $\Sigma$ among its factors.
Hence, to characterize $k$-collapsing words it is enough
to consider $k$-full words compressing all \emph{proper} $k$-\emph{compressible automata}, \emph{i.e.}, $k$-compressible automata which are not compressed by any word of length $k$.

\begin{proposition}
Let $\mathcal{A}$ be a finite complete automaton on the alphabet $\{a,b\}$: it is $3$-compressible and not proper if at least one letter, say $a$, fulfills one of the following conditions:
\begin{enumerate}
  \item it has deficiency greater than 2;
  \item it has deficiency 2 and is of type $[x,y,z]\backslash u,v$, with $\{u,v\}\nsubseteq\{x,y,z\}$;
  \item it has deficiency 2 and is of type $[x,y][z,v]\backslash u,w$, with either $\{u,w\}=\{x,y\}$, or $\{u,w\}=\{z,v\}$, or $\{u,w\}\nsubseteq\{x,y,z,v\}$;
  \item it has deficiency 1 and is of type $[x,y]\backslash z$, with $z \notin \{x,y\}$ and $za \neq x$.
\end{enumerate}
\end{proposition}

The proof of the previous proposition is trivial, indeed if the letter $a$ fulfills one of the above conditions, then either $a$ or $a^2$ or $a^3$ has deficiency 3.
Since we are looking for a $3$-compressible proper automaton we may assume that each letter of the alphabet either is a permutation or is of one of the following types (we assume different letters represent different states):
\begin{description}
  \item[\bf{1.}] $[x,y,z]\backslash x,y$;
  \item[\bf{2.}] $[x,y][z,v]\backslash x,z$;
  \item[\bf{3.}] $[x,y]\backslash x$;
  \item[\bf{4.}] $[x,y]\backslash z$ with $za=x$.
\end{description}

In the sequel we view the set $Q$ of the states of $\A$ as a set of natural numbers: $Q=\{1,2,\ldots,n\}$, so that, when no confusion arises, a letter $a$ of types $\bf{1},\bf{2},\bf{3},\bf{4}$ is denoted respectively by $[1,2,3]\backslash 1,2$, $[1,2][3,4]\backslash 1,3$, $[1,2]\backslash 1$,
$[1,2]\backslash 3$ with $3a=1$.

\begin{definition}
Let $w\in\Sigma^+$, we call $\M(w)=Q\setminus Qw$ the \emph{missing set} of $w$.
Let $Q_1\subseteq Q$, we denote by $\M(Q_1,w)$ the set $\M(w)\cup \{qw\mid \ q\in Q_1 \text{ and } \forall q'\in Q\setminus Q_1,\ qw\neq q'w\}$, i.e., the \emph{missing set of $w$ having already missed $Q_1$}.
\end{definition}

Observe that $\M(\varnothing,w)=\M(w)$, and if $a\in\Sigma$ is a permutation, $\M(Q_1,a)=Q_1a$.
Moreover, $|\M(w)|\geq |\M(w_1)|$, whenever $w_1$ is a factor of $w$.

\begin{definition}
With abuse of language, for a letter $a$ and $Q_1 \subseteq Q$, we call the \emph{orbit of $a$ over $Q_1$} the set $Orb_a(Q_1)=\bigcup^{+\infty}_{n=0}Q_1a^n=\bigcup_{q\in Q_1}Orb_a(q)$.
\end{definition}

We say that $\A$ is a $(\bf{i},\bf{j})$-\emph{automaton}, $1\leq i,j\leq 4$, if it is an automaton on a two letter alphabet $\{a,b\}$ and the letter $a$ is of type $\bf{i}$ and $b$ is of type $\bf{j}$.
We say that $\A$ is a $(\bf{i},\bf{p})$-\emph{automaton}, with $1\leq i\leq 4$, to denote that the letter $a$ is of type $\bf{i}$ and $b$ is a permutation.
In the sequel, without loss of generality, we will always suppose that in a  $(\bf{i},\bf{j})$-automaton (resp. $(\bf{i},\bf{p})$-automaton) the letter $a$ is of type $\bf{i}$ and $b$ is of type $\bf{j}$ (resp. a permutation).

Although the notion of missing set is sufficient to describe the compressibility of an automaton, it is in general quite intricate to use, especially when long words are involved.
So, to easily calculate the set $\M(Q_1,w)$, we introduce a graphical device, the \emph{missing state automaton} of $\A$.

\begin{definition}
Let $\A=(Q,\Sigma,\delta)$ be a deterministic (semi)automaton, with $|Q|=n$ and $m<n$.
The (complete) $m$-Missing State Automaton ($m$MSA for short) of $\A$ is the automaton $\M=(\wp^{m-1}(Q)\cup\{\mathbf{m}\},\Sigma,\emptyset,\tau,\mathbf{m})$, where $\wp^{m-1}(Q)$ is the set of subsets of $Q$ of cardinality less or equal to $m-1$, $\mathbf{m}$ is a special state not belonging to $\wp^{m-1}(Q)$ graphically denoted by a circle with $m$ token inside, and $\tau: \wp^{m-1}(Q)\times \Sigma\rightarrow \wp^{m-1}(Q)\cup\{\mathbf{m}\}$ is the transition relation defined by
\[
\tau(Q_1,a)=\left\{
              \begin{array}{ll}
                \M(Q_1,a), & \hbox{if $|\M(Q_1,a)|<m$;} \\
                \mathbf{m}, & \hbox{else.}
              \end{array}
            \right.
\]
notice that $\tau$ is not defined over $\mathbf{m}$.
\end{definition}

For example, in Fig. \ref{es} we draw the 2MSA of a simple semiautomaton, proving that it is synchronizable.

\begin{figure}[ht]%
\centering
\subfigure[][The Cern\'{y} automaton $\A$.]{%
{\footnotesize
\begin{tikzpicture}[shorten >=1pt,node distance=5em,auto]

    \node[state] (q_0) [draw=none] {};
    \node[state] (q_1) [right of= q_0] {$1$};
    \node[state] (q_2) [below of= q_0] {$0$};
    \node[state] (q_3) [right of= q_2,draw=none] {};
    \node[state] (q_4) [right of= q_3] {$2$};

    \path[->] (q_1) edge [loop above] node {$a$} (q_1)
              (q_1) edge node {$b$} (q_4)
              (q_4) edge [loop right] node {$a$} (q_4)
              (q_4) edge node {$b$} (q_2)
              (q_2) edge node {$a,b$} (q_1);

\end{tikzpicture}}%
}
\hspace{20pt}%
\subfigure[][The 2MSA of $\A$.]{%
{\footnotesize
\tikzstyle{accepting}=[tokens=2,accepting by double]
\begin{tikzpicture}[shorten >=1pt,node distance=5em,auto]

    \node[state] (q_0) {};
    \node[state] (q_1) [right of= q_0] {$0$};
    \node[state] (q_2) [right of= q_1] {$1$};
    \node[state] (q_3) [right of= q_2,draw=none] {};
    \node[state] (q_4) [right of= q_3] {$2$};
    \node[state,accepting] (q_5) [below of= q_3] {};

    \path[->] (q_0) edge [loop above] node {$b$} (q_0)
              (q_0) edge node {$a$} (q_1)
              (q_1) edge [loop above] node {$a$} (q_1)
              (q_1) edge node {$b$} (q_2)
              (q_2) edge node {$b$} (q_4)
              (q_4) edge [bend right=25] node [swap] {$b$} (q_1)
              (q_2) edge node {$a$} (q_5)
              (q_4) edge node {$a$} (q_5);

    \draw [->] (-.5,.5) -- (-.27,.27);

\end{tikzpicture}}}%
\caption[]{The Cern\'{y} semiautomaton with 3 states and its 2MSA: the set of synchronizing words for $\A$ is the regular language $b^*a^+(b^3)^*(b+b^2)a(a+b)^*$.}%
\label{es}%
\end{figure}
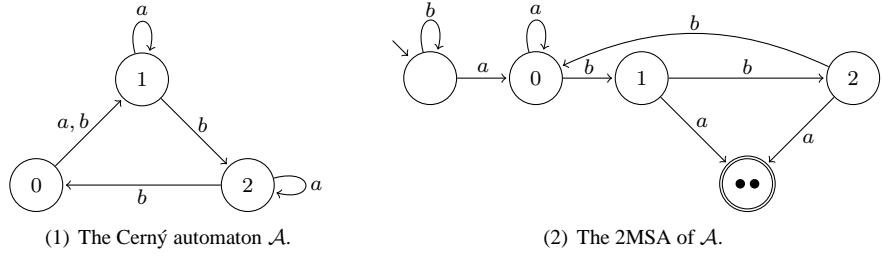

The notion of missing state automaton is similar to that of power state automaton, which is a standard tool in computing synchronizing words, see \cite{Trahtman2006,Kudlacik,Volkov2008}, the difference is that the names of states are replaced by their complements and all superstates made by more than $m$ states are identified.
Although, power set automata are only used to design algorithm to find possibly shortest synchronizing words of a fixed automaton, while we need to consider a whole class of automata.
Moreover, as we are only interested in knowing if an automaton is 3-compressible, often we will draw only a \emph{Partial $3$-Missing State Automaton} (P3MSA), \emph{i.e.}, a path (possibly the shortest) from the initial to a final state of the whole 3MSA.

Lastly, we observe that when considering a family of automata, dozens of subcases arise when we try to find some (short) 3-collapsing word for such family.
So, to capture a greater number of cases and improve the readability, we gather several subcases using a \emph{conditional} 3MSA.
In such case, a label can be of the form $a|qw\in Q'$. So, $\tau(q_1,a|qw\in Q')=q_2$ means that $\M(q_1,a)=q_2$ under the hypothesis (condition) that $qw\in Q'$.
Observe that the condition $qw\in Q'$ spreads over all the states reached by $q_2$, so two different states can share the same name, when belonging to different branches.
For example, in the conditional 3MSA in Fig. \ref{fig:caso:3-p:d:2:1}, the two states named by $\{1,3\}$ have different behavior as the one in the first row inherits the condition $3a=3$ and then $\M({1,3},a)=\{1,3\}$, while the one in the second row inherits the condition $3a=2$ and then $\M({1,3},a)=\{1,2\}$. 
\section{\texorpdfstring{3-compressible $(\bf{i},\bf{p})$-automata}{3-compressible i,p-automata}}\label{permutation}

In this section we characterize all proper $3$-compressible automata over the alphabet $\{a,b\}$ in which the letter $b$ acts as a permutation on the set $Q$ of states.
In particular in the following propositions we give a small set of short $3$-collapsing words when letter $a$ is of type $\bf{i}$, $1\leq i\leq 4$.


\begin{proposition}\label{proposition:1-p}
Let $\mathcal{A}$ be a $(\bf{1},\bf{p})$-automaton with $a=[1,2,3]\backslash 1,2$. Then $\A$ is $3$-compressible and proper if, and only if, the following conditions hold:
\begin{enumerate}
  \item $Orb_{b}(1,2)\nsubseteq\{1,2,3\}$, and
  \item $\{1,2\}b \subset\{1,2,3\}$.
\end{enumerate}
Moreover, if $\A$ is proper and $3$-compressible, then the word $ab^2a$ 3-compresses $\A$.
\end{proposition}

\begin{proof}
Let $\A$ be a $(\bf{1},\bf{p})$-automaton that does not satisfy one of the conditions $1.$ and $2.$ If $Orb_{b}(1,2)\subseteq\{1,2,3\}$ for all word $w\in\{a,b\}^*$ it is $\M(wa)=\{1,2\}$, then $\A$ is not $3$-compressible. Else, if $\{1,2\}b\nsubseteq\{1,2,3\}$, then $\M(a)=\{1,2\}$, $\M(ab)=\{1b,2b\}\nsubseteq\{1,2,3\}$, and $|\M(aba)|=3$, so $\A$ is not proper.

Conversely, let $\A$ be an automaton satisfying conditions $1.$ and $2.$ A $3$-compressing word for $\A$ must have at least two non-consecutive occurrences of letter $a$, and the word $aba$ is not $3$-compressing.
Moreover, $\{1,2\}b^2 \nsubseteq\{1,2,3\}$, else $Orb_{b}(1,2)\subseteq\{1,2,3\}$, against the hypothesis, and then the word $ab^2a$ $3$-compresses $\A$.
\end{proof}


\begin{proposition}\label{proposition:2-p}
Let $\mathcal{A}$ be a $(\bf{2},\bf{p})$-automaton with $a=[1,2][3,4]\backslash 1,3$.
Then $\mathcal{A}$ is 3-compressible and proper if, and only if,
\begin{enumerate}
  \item $Orb_{b}(1,3)\nsubseteq\{1,2,3,4\}$ and $\{1,3\}b\subseteq\{1,2,3,4\}$, or
  \item $\{1,3\}b\in\{\{1,4\},\{2,3\}\}$ and either $|Orb_b(1)|= 3$ or $|Orb_b(3)|= 3$.
\end{enumerate}
Moreover, if $\mathcal{A}$ is proper and $3$-compressible, then one of the words $ab^2a$ or $ab^3a$ $3$-compresses $\A$.
\end{proposition}

\begin{proof}
Let $\A$ a $(\bf{2},\bf{p})$-automaton that does not satisfy both conditions $1.$ and $2.$ If $\{1,3\}b\nsubseteq\{1,2,3,4\}$, then the word $aba$ $3$-compresses $\A$, which is not proper.
So, let $Orb_{b}(1,3)\subseteq\{1,2,3,4\}$, then we have to consider only the following cases:
\begin{enumerate}
    \item $\{1,3\}b\in\{\{1,2\},\{3,4\}\}$, then again the word $aba$ 3-compresses $\A$;
    \item $\{1,3\}b=\{1,3\}$, then for all $w\in b^*$, $\M(w)=\emptyset$, while for all $w\in \{a,b\}^*\setminus b^*$, $\M(w)=\{1,3\}$, then $\A$ is not $3$-compressible;
    \item $\{1,3\}b\in\{\{1,4\},\{2,3\}\}$ with $|Orb_b(1)|\neq 3$ or $|Orb_b(3)|\neq 3$, then $b=(1 4 2 3)\pi$ or $b=(1 3 2 4)\pi$ or $b=(1)(2)(3 4)\pi$ or $b=(1 2)(3)(4)\pi$. The 3MSA in figures \ref{fig:caso:2-p:2} and \ref{fig:caso:2-p:3} prove that in any case $\A$ is not $3$-compressible;
    \item $\{1,3\}b=\{2,4\}$, in this case we have four subcases: $b=(1 2)(3 4)\pi$ or $b=(1 4)(3 2)\pi$ or $b=(1 2 3 4)\pi$ or $b=(1 4 3 2)\pi$, and the 3MSA in Fig. \ref{fig:caso:2-p:3} proves that $\A$ is not $3$-compressible.
\end{enumerate}

\begin{figure}[ht]%
\centering
\subfigure[][3MSA for the case $b=(1 4 2 3)\pi$ or $b=(1 3 2 4)\pi$.]{%
\label{fig:caso:2-p:2}%
{\footnotesize
\begin{tikzpicture}[shorten >=1pt,node distance=4.6em,auto]
    \node[state] (q_0) {};
    \node[state] (q_1) [right of= q_0] {$1,3$};
    \node[state] (q_2) [right of= q_1] {$1b,3b$};
    \node[state] (q_3) [right of= q_2] {$2,4$};
    \node[state] (q_4) [right of= q_3] {$2b,4b$};

    \path[->] (q_0) edge node {$a$} (q_1) edge [loop above] node {$b$} ()
              (q_1) edge [bend left=10] node {$b$} (q_2) edge [loop above] node {$a$} ()
              (q_2) edge [bend left=10] node {$a$} (q_1) edge node {$b$} (q_3)
              (q_3) edge [bend left=30] node {$a$} (q_1) edge node {$b$} (q_4)
              (q_4) edge [bend right=30] node [swap] {$a,b$} (q_1);

    \draw [->] (-.5,.5) -- (-.27,.27);

\end{tikzpicture}}}%
\hspace{30pt}%
\subfigure[][3MSA for the case $b=(1)(2)(3 4)\pi$, $b=(1 2)(3)(4)\pi$, or $\{1,3\}b=\{2,4\}$.]{%
\label{fig:caso:2-p:3}%
{\footnotesize
\begin{tikzpicture}[shorten >=1pt,node distance=4.6em,auto]

    \node[state] (q_0) {};
    \node[state] (q_1) [right of= q_0] {$1,3$};
    \node[state] (q_2) [right of= q_1] {$1b,3b$};
    \node[state] (q_3) [draw=none,right of= q_2] {};

    \path[->] (q_0) edge node {$a$} (q_1) edge [loop above] node {$b$} ()
              (q_1) edge [bend left=10] node {$b$} (q_2) edge [loop above] node {$a$} ()
              (q_2) edge [bend left=10] node {$a,b$} (q_1);

    \draw [->] (-.5,.5) -- (-.27,.27);

    \phantom{\path[->] (q_2) edge [bend left=30] node {$a$} (q_0);}
    \phantom{\path[->] (q_2) edge [loop right] node {} (q_2);}

\end{tikzpicture}}}%
\caption[]{3MSA for automata that do not satisfy conditions $1.$ and $2.$ of Proposition \ref{proposition:2-p}.}%
\label{fig:caso:2-p}%
\end{figure}

Conversely,
  \begin{enumerate}
    \item if $Orb_{b}(1,3)\nsubseteq\{1,2,3,4\}$ and $\{1,3\}b\subseteq\{1,2,3,4\}$.
        Then $\A$ is proper and either $1b\neq 1$ or $3b\neq 3$, and so $|\{1,3,1b,3b\}|\geq 3$.
        If $|\{1,3,1b,3b\}|=4$, then $|\{1,3,1b,3b,1b^2,3b^2\}|>4$ and the word $ab^2a$ 3-compresses $\A$.
        Otherwise, let $|\{1,3,1b,3b\}|=3$.
        Without loss of generality, we can assume $1b=1$ and $3b\neq 3$.
        As $3b\in\{2,4\}$, $\A$ is proper.
        If $3b^2\not\in\{1,2,3,4\}$, then the word $ab^2a$ 3-compresses $\A$.
        If $3b^2\in\{1,2,3,4\}$, actually $3b^2\in\{2,4\}$.
        As $3b\neq 3b^2$, then $|\{1,3,3b,3b^2\}|=4$, so $3b^3\not\in\{1,2,3,4\}$ and the word $ab^3a$ 3-compresses $\A$.
    \item Suppose $\{1,3\}b=\{1,4\}$.
        If $|Orb_b(1)|= 3$, then $1b=4$ and $3b=1$, i.e., $b=(1 4 3)\pi$,
        if $|Orb_b(3)|= 3$, then either $b=(1 4 3)\pi$, or $b$ fixes 1 and is of the form $(3 4 x)\pi$ for some $x\not\in\{1,3,4\}$.
        Then $\{1,3\}b^2\in\{\{1,x\},\{3,4\}\}$, so the word $ab^2a$ $3$-compresses $\A$.
        The case $\{1,3\}b=\{2,3\}$ is identical.
  \end{enumerate}
\end{proof}

Observe that each $3$-compressible $(\bf{3},\bf{p})$- or $(\bf{4},\bf{p})$-automaton $\A$ is proper.
Indeed each $3$-compressing word for $\A$ contains at least three occurrences of the letter $a$ which are not all consecutive.
So in the next two propositions, we only look for $3$-compressible automata.


\begin{proposition}\label{proposition:3-p}
Let $\A$ be a $(\bf{3},\bf{p})$-automaton with $a=[1,2]\backslash 1$.
Then $\A$ is $3$-compressible (and proper) if, and only if, the following conditions hold:
 \begin{enumerate}
   \item $|Orb_b(1)|\geq 2$ and $\{1,2\}b\neq\{1,2\}$, and
   \item if $b=(13)\pi$, then
   \begin{enumerate}
     \item $3a\neq3$, and
     \item if $3a=2$, then $2b\neq 2$ or $2a\neq 3$, and
     \item if $3a=2b$ and $2b\not\in\{2,3\}$, then $2b^2\neq 2$ and $2ba\neq 3$,
   \end{enumerate}
   \item if $b=(123)\pi$, or $b=(132)\pi$, then $\{2,3\}a\neq \{2,3\}$, and
   \item if $b=(1324)\pi$, then $\{3,4\}a\neq \{3,4\}$.
 \end{enumerate}
Moreover, if $\mathcal{A}$ is $3$-compressible (and proper), then one of the words $ababa$ or $aba^2ba$ or or $ab^2ab^2a$ or $ab^2a^2b^2a$ or $ab^2abab^2a$ or $abab^2aba$ or $ab^3aba$ or $abab^3a$ or $ab^3ab^3a$ $3$-compresses $\A$.
\end{proposition}

\begin{proof}
Let $\A$ be a $(\bf{3},\bf{p})$-automaton that does not satisfy one of the conditions 1.-4., we prove that it is not $3$-compressible.

\begin{enumerate}
  \item Let $|Orb_b(1)|=1$ or $\{1,2\}b=\{1,2\}$, then $1b=1$ or $b=(12)\pi$.
      In the former case for all $w\in (a+b)^*$, $\M(w)\in\{\emptyset,\{1\}\}$, in the latter $\M(w)\in\{\emptyset,\{1\},\{2\}\}$, so $\A$ is not $3$-compressible.
  \item Let $b=(13)\pi$ and either $3a=3$, or $3a=2$, $2b= 2$ and $2a= 3$, or $3a=2b=4$, $2b^2= 2$ and $2ba= 3$ (in the last two cases $3a^2=3$ and $3ab=2$). Then the 3MSA in Fig. \ref{fig:caso:3-p:d:1} proves that in any case $\A$ is not 3-compressible.
      \begin{figure}[ht]%
      \centering
      {\footnotesize
      \begin{tikzpicture}[shorten >=1pt,node distance=5em,auto]

        \node[state] (q_0) {};
        \node[state] (q_1) [right of= q_0] {$1$};
        \node[state] (q_2) [right of= q_1] {$3$};
        \node[state] (q_7) [right of= q_2] {$2,3$};
        \node[state] (q_5) [right of= q_7] {$1,3a$};
        \node[state] (q_3) [right of= q_5] {$1,3$};
        \node[state] (q_8) [right of= q_3] {$1,3$};

        \path[->] (q_0) edge node {$a$} (q_1)
                  (q_0) edge [loop above] node {$b$} (q_0)
                  (q_1) edge [loop above] node {$a$} (q_1)
                  (q_1) edge [bend left=10] node {$b$} (q_2)
                  (q_2) edge [bend left=40] node {$a|3a\in\{2,4\}$} (q_5)
                  (q_2) edge [bend right=20] node [swap] {$a|3a=3$} (q_8)
                  (q_2) edge [bend left=10] node {$b$} (q_1)
                  (q_3) edge [bend left=10] node {$a$} (q_5)
                  (q_3) edge [loop above] node {$b$} (q_3)
                  (q_5) edge [bend left=10] node {$a$} (q_3)
                  (q_5) edge [bend right=10] node [swap] {$b$} (q_7)
                  (q_7) edge [bend right=10] node [swap] {$a,b$} (q_5)
                  (q_8) edge [loop above] node {$a,b$} (q_8);

        \draw [->] (-.5,.5) -- (-.27,.27);

      \end{tikzpicture}}\\
      \caption[]{3MSA for automata that do not satisfy condition $2.$ of Proposition \ref{proposition:3-p}.}%
      \label{fig:caso:3-p:d:1}%
      \end{figure}
  \item Let $b=(123)\pi$, or $b=(132)\pi$, and $\{2,3\}a=\{2,3\}$.
      The 3MSA in figures \ref{fig:caso:3-p:d:2:1} and \ref{fig:caso:3-p:d:2:2} prove that $\A$ is not 3-compressible.
      \begin{figure}[ht]%
        \centering
        \subfigure[][$b=(123)\pi$.]{%
        \label{fig:caso:3-p:d:2:1}%
        {\footnotesize
        \begin{tikzpicture}[shorten >=1pt,node distance=5em,auto]

            \node[state] (q_0) {};
            \node[state] (q_1) [right of= q_0] {$3$};
            \node[state] (q_2) [right of= q_1] {$2,3$};
            \node[state] (q_3) [right of= q_2] {$1,3$};
            \node[state] (q_4) [right of= q_3] {$1,2$};
            \node[state] (q_5) [below of= q_0] {$1$};
            \node[state] (q_6) [right of= q_5] {$2$};
            \node[state] (q_7) [right of= q_6] {$1,3$};
            \node[state] (q_8) [right of= q_7] {$1,2$};
            \node[state] (q_9) [right of= q_8] {$2,3$};

            \path[->] (q_0) edge node {$a$} (q_5)
                      (q_0) edge [loop above] node {$b$} (q_0)
                      (q_1) edge [bend left=40] node {$a|3a=3$} (q_3)
                      (q_1) edge [bend left=12] node [swap] {$a|3a=2$} (q_8)
                      (q_1) edge node {$b$} (q_5)
                      (q_2) edge [bend left=10] node {$a,b$} (q_3)
                      (q_3) edge [loop above] node {$a$} (q_3)
                      (q_3) edge [bend left=10] node {$b$} (q_4)
                      (q_4) edge [loop above] node {$a$} (q_4)
                      (q_4) edge [bend left=30] node {$b$} (q_2)
                      (q_5) edge [loop below] node {$a$} (q_5)
                      (q_5) edge [bend left=10] node {$b$} (q_6)
                      (q_6) edge [bend left=10] node {$a$} (q_5)
                      (q_6) edge node {$b$} (q_1)
                      (q_7) edge [bend left=10] node {$a,b$} (q_8)
                      (q_8) edge [bend left=10] node {$a$} (q_7)
                      (q_8) edge [bend left=10] node {$b$} (q_9)
                      (q_9) edge [bend left=10] node {$a$} (q_8)
                      (q_9) edge [bend left=30] node {$b$} (q_7);

            \draw[->] (-.5,.5) -- (-.27,.27);

        \end{tikzpicture}}}%
        \hspace{20pt}%
        \subfigure[][$b=(132)\pi$.]{%
        \label{fig:caso:3-p:d:2:2}%
        {\footnotesize
        \begin{tikzpicture}[shorten >=1pt,node distance=5em,auto]

            \node[state] (q_0) {};
            \node[state] (q_1) [right of= q_0] {$3$};
            \node[state] (q_2) [right of= q_1] {$1,2$};
            \node[state] (q_3) [right of= q_2] {$1,3$};
            \node[state] (q_4) [right of= q_3] {$2,3$};
            \node[state] (q_5) [below of= q_0] {$1$};
            \node[state] (q_6) [right of= q_5] {$2$};
            \node[state] (q_7) [right of= q_6] {$2,3$};
            \node[state] (q_8) [right of= q_7] {$1,2$};
            \node[state] (q_9) [right of= q_8] {$1,3$};

            \path[->] (q_0) edge node {$a$} (q_5)
                      (q_0) edge [loop above] node {$b$} (q_0)
                      (q_1) edge [bend left=30] node {$a|3a=3$} (q_3)
                      (q_1) edge [bend left=10] node [swap] {$a|3a=2$} (q_8)
                      (q_1) edge node {$b$} (q_6)
                      (q_2) edge [loop left] node {$a$} (q_2)
                      (q_2) edge [bend left=10] node {$b$} (q_3)
                      (q_3) edge [loop above] node {$a$} (q_3)
                      (q_3) edge [bend left=10] node {$b$} (q_4)
                      (q_4) edge [bend left=10] node {$a$} (q_3)
                      (q_4) edge [bend left=30] node {$b$} (q_2)
                      (q_5) edge [loop below] node {$a$} (q_5)
                      (q_5) edge node {$b$} (q_1)
                      (q_6) edge node {$a,b$} (q_5)
                      (q_7) edge node {$a,b$} (q_8)
                      (q_8) edge [bend left=10] node {$a,b$} (q_9)
                      (q_9) edge [bend left=10] node {$a$} (q_8)
                      (q_9) edge [bend left=30] node {$b$} (q_7);

            \draw[->] (-.5,.5) -- (-.27,.27);

        \end{tikzpicture}}}%
        \caption[]{3MSA for automata that do not satisfy condition $3.$ of proposition \ref{proposition:3-p}.}%
        \label{fig:caso:3-p:d:2}%
      \end{figure}
  \item Let $b=(1324)\pi$, and $\{3,4\}a=\{3,4\}$.
      The 3MSA in Fig. \ref{fig:caso:3-p:d:4} proves that $\A$ is not 3-compressible.
      \begin{figure}[ht]%
        \centering
        {\footnotesize
        \begin{tikzpicture}[shorten >=1pt,node distance=5em,auto]

            \node[state] (q_00) {};
            \node[state] (q_0) [right of= q_00] {$1$};
            \node[state] (q_7) [right of= q_0] {$2$};
            \node[state] (q_1) [above of= q_7] {$3$};
            \node[state] (q_2) [right of= q_1,draw=none] {};
            \node[state] (q_3) [right of= q_2,draw=none] {};
            \node[state] (q_4) [right of= q_3,draw=none] {};
            \node[state] (q_5) [right of= q_4,draw=none] {};
            \node[state] (q_6) [right of= q_5] {$1,3$};
            \node[state] (q_8) [right of= q_7] {$1,4$};
            \node[state] (q_9) [right of= q_8] {$2,4$};
            \node[state] (q_10) [right of= q_9] {$2,3$};
            \node[state] (q_11) [right of= q_10] {$1,3$};
            \node[state] (q_12) [right of= q_6] {$2,3$};
            \node[state] (q_14) [below of= q_7] {$4$};
            \node[state] (q_15) [right of= q_14,draw=none] {};
            \node[state] (q_16) [right of= q_15,draw=none] {};
            \node[state] (q_17) [right of= q_16,draw=none] {};
            \node[state] (q_18) [right of= q_17,draw=none] {};
            \node[state] (q_19) [right of= q_18] {$1,4$};
            \node[state] (q_13) [right of= q_19] {$2,4$};

            \path[->] (q_00) edge node {$a$} (q_0)
                      (q_00) edge [loop above] node {$b$} (q_00)
                      (q_0) edge [loop above] node {$a$} (q_0)
                      (q_0) edge node {$b$} (q_1)
                      (q_1) edge node {$b$} (q_7)
                      (q_1) edge node {$a|3a=3$} (q_6)
                      (q_1) edge node {$a|3a=4$} (q_8)
                      (q_7) edge node {$b$} (q_14)
                      (q_7) edge node {$a$} (q_0)
                      (q_14) edge node {$b$} (q_0)
                      (q_14) edge node [swap] {$a|3a=3$} (q_19)
                      (q_14) edge [bend right=10,in=230,out=15] node [swap] {$a|3a=4$} (q_11)
                      (q_8) edge [bend right=35] node [swap] {$a,b$} (q_11)
                      (q_9) edge node [swap] {$b$} (q_8)
                      (q_9) edge [bend right=25] node [swap] {$a$} (q_11)
                      (q_10) edge node [swap] {$b$} (q_9)
                      (q_10) edge [bend right=25] node [swap] {$a$} (q_8)
                      (q_11) edge node [swap] {$b$} (q_10)
                      (q_11) edge [bend right=35] node [swap] {$a$} (q_8)
                      (q_12) edge node {$b$} (q_13)
                      (q_6) edge [loop above] node {$a$} (q_13)
                      (q_6) edge [bend left=10] node {$b$} (q_12)
                      (q_12) edge [bend left=10] node {$a$} (q_6)
                      (q_12) edge node {$b$} (q_13)
                      (q_13) edge node {$a,b$} (q_19)
                      (q_19) edge [loop below] node {$a$} (q_19)
                      (q_19) edge node [swap] {$b$} (q_6);

            \draw [->] (-.5,.5) -- (-.27,.27);

        \end{tikzpicture}}\\
        \caption[]{3MSA for automata that do not satisfy condition $4.$ of proposition \ref{proposition:3-p}.}%
        \label{fig:caso:3-p:d:4}%
      \end{figure}
      \FloatBarrier
\end{enumerate}

\FloatBarrier

\noindent Conversely, let $\A$ be an automaton satisfying conditions $1.-4.$
Since $1b\neq 1$ and $\{1,2\}b\neq\{1,2\}$, then $|Orb_b(1)|\geq 2$ and $b$ is not of the form $(12)\pi$.
\begin{enumerate}
  \item Let $|Orb_b(1)|=2$, $b=(13)\pi$, and $3a\neq 3$, then there are two further subcases:
    \begin{enumerate}
      \item if $|Orb_b(2)|\leq 2$, \textit{i.e.}, $b=(13)(2)\pi$ or $b=(13)(24)\pi$, then $\{2b,3\}a\neq\{2b,3\}$ and in Fig. \ref{fig:caso:3-p:r:1:1} we draw a P3MSA for those cases, proving that either the word $ababa$ or $aba^2ba$ $3$-compresses $\A$.
      \item if $|Orb_b(2)|\geq 3$, \textit{i.e.}, $b=(13)(245\ldots)\pi$, then in Fig. \ref{fig:caso:3-p:r:1:3} we draw a P3MSA for this case, proving that either the word $ababa$ or $abab^3a$ $3$-compresses $\A$.
          \begin{figure}[ht]%
            \centering
            \subfigure[][P3MSA for the case $b=(13)(2)\pi$ or $b=(13)(24)\pi$, and $\{2b,3\}a\neq\{2,3\}$.]{%
            \label{fig:caso:3-p:r:1:1}%
            {\footnotesize
            \begin{tikzpicture}[shorten >=1pt,node distance=4.6em,auto]

                \node[state] (q_0) {};
                \node[state] (q_1) [right of= q_0] {$1$};
                \node[state] (q_2) [right of= q_1] {$3$};
                \node[state] (q_3) [right of= q_2] {$1,3a$};
                \node[state] (q_4) [below of= q_0] {$1,3a^2$};
                \node[state,inner sep=-4pt] (q_5) [right of= q_4] {{\scriptsize $\begin{array}{c}3, \\3a^2b\end{array}$}};
                \node[state,accepting] (q_6) [right of= q_5] {};
                \node[state] (q_7) [right of= q_6] {$3,3ab$};

                \path[->] (q_0) edge node {$a$} (q_1)
                          (q_1) edge node {$b$} (q_2)
                          (q_2) edge node {$a$} (q_3)
                          (q_3) edge node [swap] {$a|3a=2b$} (q_4)
                          (q_3) edge node [inner sep=-4pt] {$b|3a\neq 2b$} (q_7)
                          (q_4) edge node [swap] {$b$} (q_5)
                          (q_5) edge node [swap] {$a$} (q_6)
                          (q_7) edge node {$a$} (q_6);

                \draw [->] (-.5,.5) -- (-.27,.27);

            \end{tikzpicture}}%
            }
            \hspace{20pt}%
            \subfigure[][P3MSA for the case $b=(13)(245\ldots)\pi$.]{%
            \label{fig:caso:3-p:r:1:3}%
            {\footnotesize
            \begin{tikzpicture}[shorten >=1pt,node distance=4.6em,auto]

                \node[state] (q_0) {};
                \node[state] (q_1) [right of= q_0] {$1$};
                \node[state] (q_2) [right of= q_1] {$3$};
                \node[state] (q_3) [right of= q_2] {$1,3a$};
                \node[state] (q_4) [below of= q_0] {$2,3$};
                \node[state] (q_5) [right of= q_4] {$3,5$};
                \node[state,accepting] (q_6) [right of= q_5] {};
                \node[state] (q_7) [right of= q_6] {$3,3ab$};

                \path[->] (q_0) edge node {$a$} (q_1)
                          (q_1) edge node {$b$} (q_2)
                          (q_2) edge node {$a$} (q_3)
                          (q_3) edge node [swap] {$b|3ab=2$} (q_4)
                          (q_3) edge node [inner sep=-4pt] {$b|3ab\neq 2$} (q_7)
                          (q_4) edge node [swap] {$b^2$} (q_5)
                          (q_5) edge node [swap] {$a$} (q_6)
                          (q_7) edge node {$a$} (q_6);

                \draw [->] (-.5,.5) -- (-.27,.27);

            \end{tikzpicture}}}%
            \caption[]{P3MSA for $3$-compressible automata with $b=(13)\pi$ and $3a\neq 3$.}%
            \label{fig:caso:3-p:r:1}%
          \end{figure}
    \end{enumerate}
  \item Let $|Orb_b(1)|=3$, then there are three further subcases:
    \begin{enumerate}
      \item if $b=(123)\pi$, then $\{2,3\}a\neq\{2,3\}$ and in Fig. \ref{fig:caso:3-p:r:2:1} we draw a P3MSA for this case, proving that either the word $ab^2ab^2a$ or $ab^2a^2b^2a$ or $ab^2abab^2a$ $3$-compresses $\A$.
      \item if $b=(132)\pi$, then $\{2,3\}a\neq\{2,3\}$ and in Fig. \ref{fig:caso:3-p:r:2:2} we draw a P3MSA for this case, proving that either the word $aba^2ba$ or $ababa$ or $abab^2aba$ $3$-compresses $\A$.
      \item if $b=(134)\pi$, then in Fig. \ref{fig:caso:3-p:r:2:3} we draw a P3MSA for this case, proving that either the word $abab^2a$ or $ababa$ or $ab^2ab^2a$ or $ab^2aba$ $3$-compresses $\A$.
          \begin{figure}[ht]%
            \centering
            {\footnotesize
            \begin{tikzpicture}[shorten >=1pt,node distance=5em,auto]

                \node[state] (q_0) {};
                \node[state] (q_1) [right of= q_0] {$1$};
                \node[state] (q_2) [right of= q_1] {$2$};
                \node[state] (q_3) [right of= q_2] {$3$};
                \node[state] (q_4) [right of= q_3] {$1,3a$};
                \node[state,draw=none] (q_5) [right of= q_4] {};
                \node[state] (q_6) [right of= q_5] {$1,2a$};
                \node[state] (q_7) [right of= q_6] {$2,2ab$};
                \node[state] (q_8) [right of= q_7] {$3,2ab^2$};
                \node[state] (q_9) [below of= q_2] {$2,3ab$};
                \node[state] (q_10) [right of= q_9] {$3,3ab^2$};
                \node[state] (q_11) [right of= q_10] {$1,2$};
                \node[state] (q_12) [right of= q_11] {$1,2a$};
                \node[state] (q_13) [right of= q_12] {$2,2ab$};
                \node[state] (q_14) [right of= q_13] {$3,2ab^2$};
                \node[state,accepting] (q_15) [right of= q_14] {};

                \path[->] (q_0) edge node {$a$} (q_1)
                          (q_1) edge node {$b$} (q_2)
                          (q_2) edge node {$b$} (q_3)
                          (q_3) edge node {$a$} (q_4)
                          (q_4) edge node {$a|3a=2$} (q_6)
                          (q_4) edge node {$b|3a=3$} (q_11)
                          (q_4) edge node [swap] {$b|3a\not\in\{2,3\}$} (q_9)
                          (q_6) edge node {$b$} (q_7)
                          (q_7) edge node {$b$} (q_8)
                          (q_8) edge node {$a$} (q_15)
                          (q_9) edge node {$b$} (q_10)
                          (q_10) edge [bend right=22] node {$a$} (q_15)
                          (q_11) edge node {$a$} (q_12)
                          (q_12) edge node {$b$} (q_13)
                          (q_13) edge node {$b$} (q_14)
                          (q_14) edge node {$a$} (q_15);

                \draw [->] (-.5,.5) -- (-.27,.27);

            \end{tikzpicture}}\\
            \caption[]{P3MSA for 3-compressible automata with $b=(123)\pi$ and $\{2,3\}a\neq\{2,3\}$.}%
            \label{fig:caso:3-p:r:2:1}%
          \end{figure}

          \begin{figure}[ht]%
            \centering
            {\footnotesize
            \begin{tikzpicture}[shorten >=1pt,node distance=5em,auto]

                \node[state] (q_0) {};
                \node[state] (q_1) [right of= q_0] {$1$};
                \node[state] (q_2) [right of= q_1] {$3$};
                \node[state] (q_3) [right of= q_2] {$1,3a$};
                \node[state,draw=none] (q_4) [right of= q_3] {};
                \node[state] (q_5) [right of= q_4] {$1,2a$};
                \node[state] (q_6) [right of= q_5] {$3,2ab$};
                \node[state] (q_9) [below of= q_2] {$3,3ab$};
                \node[state] (q_10) [right of= q_9] {$2,3$};
                \node[state] (q_11) [right of= q_10] {$1,2$};
                \node[state] (q_12) [right of= q_11] {$1,2a$};
                \node[state] (q_13) [right of= q_12] {$3,2ab$};
                \node[state,accepting] (q_14) [right of= q_13] {};

                \path[->] (q_0) edge node {$a$} (q_1)
                          (q_1) edge node {$b$} (q_2)
                          (q_2) edge node {$a$} (q_3)
                          (q_3) edge node {$a|3a=2$} (q_5)
                          (q_3) edge node {$b|3a=3$} (q_10)
                          (q_3) edge node [swap] {$b|3a\not\in\{2,3\}$} (q_9)
                          (q_5) edge node {$b$} (q_6)
                          (q_6) edge node {$a$} (q_14)
                          (q_9) edge [bend right=22] node {$a$} (q_14)
                          (q_10) edge node {$b$} (q_11)
                          (q_11) edge node {$a$} (q_12)
                          (q_12) edge node {$b$} (q_13)
                          (q_13) edge node {$a$} (q_14);

                \draw [->] (-.5,.5) -- (-.27,.27);

            \end{tikzpicture}}\\
            \caption[]{P3MSA for 3-compressible automata with $b=(132)\pi$ and $\{2,3\}a\neq\{2,3\}$.}%
            \label{fig:caso:3-p:r:2:2}%
          \end{figure}

          \begin{figure}[ht]%
            \centering
            {\footnotesize
            \begin{tikzpicture}[shorten >=1pt,node distance=5em,auto]

                \node[state] (q_0) {};
                \node[state] (q_1) [right of= q_0] {$1$};
                \node[state] (q_2) [right of= q_1] {$3$};
                \node[state,draw=none] (q_3) [right of= q_2] {};
                \node[state] (q_4) [right of= q_3] {$1,3a$};
                \node[state] (q_5) [right of= q_4] {$3,3ab$};
                \node[state,draw=none] (q_55) [right of= q_5] {};
                \node[state] (q_6) [right of= q_55] {$3,4$};
                \node[state] (q_7) [below of= q_2] {$4$};
                \node[state] (q_8) [right of= q_7] {$1,4a$};
                \node[state] (q_9) [right of= q_8] {$3,4ab$};
                \node[state,draw=none] (q_10) [right of= q_9] {};
                \node[state] (q_11) [right of= q_10] {$3,4$};
                \node[state,accepting] (q_12) [right of= q_11] {};

                \path[->] (q_0) edge node {$a$} (q_1)
                          (q_1) edge node {$b$} (q_2)
                          (q_2) edge node {$a|3ab\neq 2$} (q_4)
                          (q_2) edge node [swap,inner sep=-4pt] {$b|3ab=2$} (q_7)
                          (q_4) edge node {$b$} (q_5)
                          (q_5) edge node {$b|3a=4$} (q_6)
                          (q_5) edge node {$a|3a\neq 4$} (q_12)
                          (q_6) edge node {$a$} (q_12)
                          (q_7) edge node {$a$} (q_8)
                          (q_8) edge node {$b$} (q_9)
                          (q_9) edge node {$b|4ab=1$} (q_11)
                          (q_9) edge [bend right=22] node [swap] {$a|4ab\neq 1$} (q_12)
                          (q_11) edge node {$a$} (q_12);

                \draw [->] (-.5,.5) -- (-.27,.27);

            \end{tikzpicture}}\\
            \caption[]{P3MSA for 3-compressible automata with $b=(134)\pi$.}%
            \label{fig:caso:3-p:r:2:3}%
          \end{figure}
    \end{enumerate}
  \item Let $|Orb_b(1)|=4$, then there are three further subcases:
    \begin{enumerate}
      \item if $b=(1234)\pi$, then in Fig. \ref{fig:caso:3-p:r:3:1} we draw a P3MSA for this case, proving that either the word $ab^2ab^2a$ or $ab^3ab^3a$ or $ab^2abab^2a$ or $ab^2a^2b^2a$ $3$-compresses $\A$.
          Indeed, observe that:
          \begin{enumerate}
            \item if $3a\not\in\{3,4\}$, then $3ab^2=1$ implies $3a=3$, and $3ab^2=2$ implies $3a=4$, and both of them are contradictions;
            \item if $3a=4$ and $4a=3$, then $2ab^2=1$ implies $2a=3$, and $2ab^2=2$ implies $2a=4$, both contradictions;
            \item if $3a=3$ and $4a\neq 4$ or if $3a=4$ and $4a\neq 3$, then $4ab^2=1$ implies $4a=3$, and $4ab^2=2$ implies $4a=4$, both contradictions.
          \end{enumerate}
          \begin{figure}[ht]%
            \centering
            {\footnotesize
            \begin{tikzpicture}[shorten >=1pt,node distance=5em,auto]

                \node[state] (q_0) {};
                \node[state] (q_1) [right of= q_0] {$1$};
                \node[state] (q_2) [right of= q_1] {$2$};
                \node[state] (q_3) [right of= q_2] {$3$};
                \node[state,draw=none] (q_4) [right of= q_3] {};
                \node[state] (q_5) [right of= q_4] {$1,3a$};
                \node[state] (q_6) [right of= q_5] {$2,3ab$};
                \node[state] (q_7) [right of= q_6] {$3,3ab^2$};
                \node[state] (q_8) [below of= q_4] {$4$};
                \node[state] (q_9) [right of= q_8] {$1,4$};
                \node[state] (q_10) [right of= q_9] {$1,2$};
                \node[state] (q_11) [right of= q_10] {$2,3$};
                \node[state] (q_12) [right of= q_11] {$3,4$};
                \node[state] (q_13) [below of= q_8] {$1,3$};
                \node[state] (q_14) [right of= q_13] {$2,4$};
                \node[state] (q_15) [right of= q_14] {$1,4a$};
                \node[state] (q_16) [right of= q_15] {$2,4ab$};
                \node[state] (q_17) [right of= q_16] {$3,4ab^2$};
                \node[state] (q_23) [above of= q_4] {$1,4$};
                \node[state] (q_24) [right of= q_23] {$1,4a$};
                \node[state] (q_25) [right of= q_24] {$2,4ab$};
                \node[state] (q_27) [right of= q_25] {$3,4ab^2$};
                \node[state] (q_18) [above of= q_23] {$1,4$};
                \node[state] (q_19) [right of= q_18] {$1,2$};
                \node[state] (q_20) [right of= q_19] {$1,2a$};
                \node[state] (q_21) [right of= q_20] {$2,2ab$};
                \node[state] (q_22) [right of= q_21] {$3,2ab^2$};
                \node[state,draw=none] (q_77) [right of= q_7] {};
                \node[state,accepting] (q_26) [right of= q_77] {};

                \path[->] (q_0) edge node {$a$} (q_1)
                          (q_1) edge node {$b$} (q_2)
                          (q_2) edge node {$b$} (q_3)
                          (q_3) edge node {$a|3a\not\in\{3,4\}$} (q_5)
                          (q_3) edge node {$b|3a=3,4a=4$} (q_8)
                          (q_3) edge node {$a|3a=4,4a=3$} (q_18)
                          (q_3) edge node [swap] {$a|3a=3,4a\neq 4$} (q_13)
                          (q_3) edge node [swap,inner sep=-2pt] {$a|3a=4,4a\neq 3$} (q_23)
                          (q_5) edge node {$b$} (q_6)
                          (q_6) edge node {$b$} (q_7)
                          (q_7) edge node {$a$} (q_26)
                          (q_8) edge node {$a$} (q_9)
                          (q_9) edge node {$b$} (q_10)
                          (q_10) edge node {$b$} (q_11)
                          (q_11) edge node {$b$} (q_12)
                          (q_12) edge node {$a$} (q_26)
                          (q_13) edge node {$b$} (q_14)
                          (q_14) edge node {$a$} (q_15)
                          (q_15) edge node {$b$} (q_16)
                          (q_16) edge node {$b$} (q_17)
                          (q_17) edge node {$a$} (q_26)
                          (q_18) edge node {$b$} (q_19)
                          (q_19) edge node {$a$} (q_20)
                          (q_20) edge node {$b$} (q_21)
                          (q_21) edge node {$b$} (q_22)
                          (q_22) edge node {$a$} (q_26)
                          (q_23) edge node {$a$} (q_24)
                          (q_24) edge node {$b$} (q_25)
                          (q_25) edge node {$b$} (q_27)
                          (q_27) edge node {$a$} (q_26);

                \draw [->] (-.5,.5) -- (-.27,.27);

            \end{tikzpicture}}\\
            \caption[]{P3MSA for the case $b=(1234)\pi$.}%
            \label{fig:caso:3-p:r:3:1}%
          \end{figure}
      \item if $b=(1324)\pi$, then $\{3,4\}a\neq\{3,4\}$ and in Fig. \ref{fig:caso:3-p:r:3:2} we draw a P3MSA for this case, proving that either the word $ababa$ or $ab^3aba$ $3$-compresses $\A$.
          Observe that if $3ab=1$, then $3a=4$, $4a\neq 3$, and so $4ab\not\in\{1,2\}$, else if $3ab=2$, then $3a=3$, $4a\not\in\{3,4\}$, and again $4ab\not\in\{1,2\}$.
          \begin{figure}[ht]%
            \centering
            {\footnotesize
            \begin{tikzpicture}[shorten >=1pt,node distance=5em,auto]

                \node[state] (q_0) {};
                \node[state] (q_1) [right=2em of q_0] {$1$};
                \node[state] (q_2) [right=2em of q_1] {$3$};
                \node[state,draw=none] (q_3) [right=2.2em of q_2] {};
                \node[state] (q_4) [right=2em of q_3] {$1,3a$};
                \node[state] (q_5) [right=2em of q_4] {$3,3ab$};
                \node[state,accepting] (q_6) [right=2em of q_5] {};
                \node[state] (q_10) [right=2em of q_6] {$3,4ab$};
                \node[state] (q_9) [right=2em of q_10] {$1,4a$};
                \node[state] (q_8) [right=2em of q_9] {$4$};
                \node[state] (q_7) [right=2em of q_8] {$2$};

                \path[->] (q_0) edge node {$a$} (q_1)
                          (q_1) edge node {$b$} (q_2)
                          (q_2) edge node {$a|3ab\not\in\{1,2\}$} (q_4)
                          (q_2) edge [bend right=16] node [swap] {$b|3ab\in\{1,2\}$} (q_7)
                          (q_4) edge node {$b$} (q_5)
                          (q_5) edge node {$a$} (q_6)
                          (q_7) edge node [swap] {$b$} (q_8)
                          (q_8) edge node [swap] {$a$} (q_9)
                          (q_9) edge node [swap] {$b$} (q_10)
                          (q_10) edge node [swap] {$a$} (q_6);

                \draw [->] (-.5,.5) -- (-.27,.27);

            \end{tikzpicture}}\\
            \caption[]{P3MSA for the case $b=(1324)\pi$ and $\{3,4\}a\neq\{3,4\}$.}%
            \label{fig:caso:3-p:r:3:2}%
          \end{figure}
    \end{enumerate}
  \item Let $|Orb_b(1)|\geq 5$, then there are three further subcases:
    \begin{enumerate}
      \item if $b=(12345\ldots)\pi$, then in Fig. \ref{fig:caso:3-p:r:4:1} we draw a P3MSA for this case, proving that either the word $ab^2ab^2a$ or $ab^2abab^2a$ or $ab^2a^2b^2a$ $3$-compresses $\A$.
          Note that if $3ab^2=1$ and $3a^2b^2=2$, then $3abab^2\not\in\{1,2\}$.
          In fact if $3abab^2=1$, then $3a=3aba$, hence $3a=2$ and $3ab^2=4$, a contradiction.
          If $3abab^2=2$, then $3aba=3a^2$, hence $3ab=3a$ and $3ab^2=3ab=1$, again a contradiction.
          Similarly, if $3ab^2=2$ and $4ab^3=2$, then $3abab^2\not\in\{1,2\}$; in fact if $3abab^2=1$ then $3ab^2ab^2=4ab^2$ and $3ab^2=4$, else if $3abab^2=2$ then $3ab^2=3abab^2$ hence $3ab=3$ and $3ab^2=4$, and in both cases this is a contradiction.
          \begin{figure}[ht]%
            \centering
            {\footnotesize
            \begin{tikzpicture}[shorten >=1pt,node distance=5em,auto]

                \node[state] (q_00) {};
                \node[state] (q_0) [right of= q_00] {$1$};
                \node[state] (q_1) [right of= q_0] {$2$};
                \node[state] (q_2) [right of= q_1] {$3$};
                \node[state] (q_3) [right of= q_2] {$1,3a$};
                \node[state,draw=none] (q_4) [right of= q_3] {};
                \node[state] (q_5) [right of= q_4] {$2,3ab$};
                \node[state] (q_6) [right of= q_5] {$3,3ab^2$};
                \node[state] (q_8) [below of= q_3] {$1,3a^2$};
                \node[state] (q_9) [right of= q_8] {$2,3a^2b$};
                \node[state,inner sep=-4pt] (q_10) [right of= q_9] {{\scriptsize $\begin{array}{c}3, \\3a^2b^2\end{array}$}};
                \node[state] (q_11) [below of= q_00] {$2,3ab$};
                \node[state] (q_12) [right of= q_11] {$1,3aba$};
                \node[state,inner sep=-4pt] (q_13) [right of= q_12] {{\scriptsize $\begin{array}{c}2, \\3abab\end{array}$}};
                \node[state,inner sep=-4pt] (q_14) [right of= q_13] {{\scriptsize $\begin{array}{c}3, \\3abab^2\end{array}$}};
                \node[state,accepting] (q_7) [right of= q_10] {};

                \path[->] (q_00) edge node {$a$} (q_0)
                          (q_0) edge node {$b$} (q_1)
                          (q_1) edge node {$b$} (q_2)
                          (q_2) edge node {$a$} (q_3)
                          (q_3) edge node {$b|3ab^2\not\in\{1,2\}$} (q_5)
                          (q_5) edge node {$b$} (q_6)
                          (q_6) edge node {$a$} (q_7)
                          (q_3) edge node [inner sep=-15pt] {$a|3ab^2\in\{1,2\},3a^2b^2\not\in\{1,2\}$}  (q_8)
                          (q_8) edge node {$b$} (q_9)
                          (q_9) edge node {$b$} (q_10)
                          (q_10) edge node {$a$} (q_7)
                          (q_3) edge [in=30,out=220,inner sep=5pt] node [swap] {$b|\{3ab^2,3a^2b^2\}=\{1,2\}$} (q_11)
                          (q_11) edge  node {$a$} (q_12)
                          (q_12) edge node {$b$} (q_13)
                          (q_13) edge node {$b$} (q_14)
                          (q_14) edge [bend right=20] node [swap] {$a$} (q_7);

                \draw [->] (-.5,.5) -- (-.27,.27);

            \end{tikzpicture}}\\
            \caption[]{P3MSA for the case $b=(12345\ldots)\pi$.}%
            \label{fig:caso:3-p:r:4:1}%
          \end{figure}
          \FloatBarrier
      \item if $b=(13245\ldots)\pi$, then in Fig. \ref{fig:caso:3-p:r:4:2} we draw a P3MSA for this case, proving that either the word $ab^3ab^3a$ or $ab^3aba$ or $abab^3a$ or $ababa$ $3$-compresses $\A$.
          \begin{figure}[ht]%
            \centering
            {\footnotesize
            \begin{tikzpicture}[shorten >=1pt,node distance=5em,auto]

                \node[state] (q_0) {};
                \node[state] (q_1) [right of= q_0] {$1$};
                \node[state] (q_2) [right of= q_1] {$3$};
                \node[state,draw=none] (q_3) [right of= q_2] {};
                \node[state] (q_4) [right of= q_3] {$2$};
                \node[state] (q_5) [right of= q_4] {$4$};
                \node[state] (q_6) [right of= q_5] {$1,4a$};
                \node[state] (q_7) [right of= q_6] {$3,4ab$};
                \node[state,draw=none] (q_8) [right of= q_7] {};
                \node[state] (q_9) [right of= q_8] {$2,4$};
                \node[state] (q_10) [below of= q_2] {$1,3a$};
                \node[state] (q_11) [right of= q_10] {$3,3ab$};
                \node[state] (q_12) [below of= q_5] {$2,4$};
                \node[state] (q_15) [right of= q_12] {$4,5$};
                \node[state,accepting] (q_13) [below of= q_7] {};
                \node[state] (q_14) [right of= q_13] {$4,5$};

                \path[->] (q_0) edge node {$a$} (q_1)
                          (q_1) edge node {$b$} (q_2)
                          (q_2) edge node {$b|3ab=1$} (q_4)
                          (q_2) edge node {$a|3ab\neq 1$} (q_10)
                          (q_4) edge node {$b$} (q_5)
                          (q_5) edge node {$a$} (q_6)
                          (q_6) edge node {$b$} (q_7)
                          (q_7) edge node {$b|4ab=2$} (q_9)
                          (q_7) edge node {$a|4ab\neq 2$} (q_13)
                          (q_9) edge node {$b$} (q_14)
                          (q_10) edge node {$b$} (q_11)
                          (q_11) edge node {$b|3ab=2$} (q_12)
                          (q_11) edge [bend right=30] node {$a|3ab\neq 2$} (q_13)
                          (q_12) edge node {$b$} (q_15)
                          (q_15) edge node {$a$} (q_13)
                          (q_14) edge node [swap] {$a$} (q_13);

                \draw [->] (-.5,.5) -- (-.27,.27);

            \end{tikzpicture}}\\
            \caption[]{3MSA for the case $b=(13245\ldots)\pi$.}%
            \label{fig:caso:3-p:r:4:2}%
          \end{figure}
          \FloatBarrier
      \item if $b=(13425\ldots)\pi$, then $\M(abab)=\{3,3ab\}$. If $3ab\not\in\{1,2\}$, then the word $ababa$ $3$-compresses $\A$, else if $3ab=1$, then $\M(abab^2)=\{3,4\}$, else if $3ab=2$, then $\M(abab^2)=\{4,5\}$ and in both cases the word $abab^2a$ $3$-compresses $\A$.
    \end{enumerate}
\end{enumerate}
\FloatBarrier
\end{proof}


\begin{proposition}\label{proposition:4-p}
Let $\A$ be a $(\bf{4},\bf{p})$-automaton with $a=[1,2]\backslash 3, 3a=1$. Then $\A$ is $3$-compressible (and proper) if, and only if, the following conditions hold:
 \begin{enumerate}
   \item $\{1,3\}b\neq\{1,3\}$ and
   \item $b\neq(12)(3)\pi$ and
   \item if $b=(1)(23)\pi$, or $b=(123)\pi$, or $b=(132)\pi$, then  $2a\neq 2$ and
   \item if $b=(1)(2)(34)\pi$, or $b=(12)(34)\pi$, or $b=(14)(23)\pi$, or $b=(1423)\pi$, or $b=(1324)\pi$, then $4a\neq 2$.
 \end{enumerate}
Moreover, if $\mathcal{A}$ is $3$-compressible (and proper), then one of the words $a^2ba^2$ or $a^2b^2a^2$ or $a^2b^3a$ or $a^2baba^2$ or $ab^3ab^3a$ $3$-compresses $\A$.
\end{proposition}

\begin{proof}
Let $\A$ be a $(\bf{4},\bf{p})$-automaton that does not satisfy one of the conditions $1.-4.$, we prove that it is not $3$-compressible.

\begin{enumerate}
  \item Let condition $1.$ be false, i.e., $\{1,3\}b=\{1,3\}$, then the 3MSA in Fig. \ref{fig:prop:4-p:d:1:1} proves that $\A$ is not 3-compressible.
  \item Let condition $2.$ be false, i.e., $b=(12)(3)\pi$, then the 3MSA in Fig. \ref{fig:prop:4-p:d:1:2} proves that $\A$ is not 3-compressible.
      \begin{figure}[ht]%
        \centering
        \subfigure[][3MSA for the case $\{1,3\}b=\{1,3\}$.]{%
        \label{fig:prop:4-p:d:1:1}%
        {\footnotesize
        \begin{tikzpicture}[shorten >=1pt,node distance=5em,auto]

            \node[state] (q_0) {};
            \node[state] (q_1) [right of= q_0] {$3$};
            \node[state] (q_2) [right of= q_1] {$1,3$};
            \node[state] (q_3) [right of= q_2,draw=none] {};
            \node[state] (q_4) [right of= q_3] {$1,3$};
            \node[state] (q_5) [right of= q_4] {$1$};
            \node[state] (q_6) [below of= q_1] {$3$};
            \node[state] (q_7) [below of= q_2] {$1,3$};
            \node[state] (q_8) [below of= q_5] {$3$};

            \path[->] (q_0) edge node {$a$} (q_1)
                      (q_0) edge [loop above] node {$b$} (q_0)
                      (q_1) edge node {$a$} (q_2)
                      (q_1) edge [bend left=20] node {$b|3b=1$} (q_5)
                      (q_1) edge node [swap] {$b|3b=3$} (q_6)
                      (q_2) edge [max distance=5mm,in=215,out=245,loop] node {$a$} (q_2)
                      (q_2) edge node {$3b=3$} (q_7)
                      (q_2) edge node {$3b=1$} (q_4)
                      (q_4) edge [max distance=5mm,in=215,out=245,loop] node {$a,b$} (q_4)
                      (q_5) edge [bend left=10] node {$b$} (q_8)
                      (q_5) edge [loop above] node {$a$} (q_5)
                      (q_6) edge [loop left] node {$b$} (q_6)
                      (q_6) edge node {$a$} (q_7)
                      (q_7) edge [loop right] node {$a,b$} (q_7)
                      (q_8) edge node {$a$} (q_4)
                      (q_8) edge [bend left=10] node {$b$} (q_5);

            \draw [->] (-.5,.5) -- (-.27,.27);

        \end{tikzpicture}}%
        }
        \hspace{20pt}%
        \subfigure[][3MSA for the case $b=(12)(3)\pi$.]{%
        \label{fig:prop:4-p:d:1:2}%
        {\footnotesize
        \begin{tikzpicture}[shorten >=1pt,node distance=5em,auto]

            \node[state] (q_0) {};
            \node[state] (q_1) [right of= q_0] {$3$};
            \node[state] (q_2) [right of= q_1] {$1,3$};
            \node[state] (q_3) [right of= q_2] {$2,3$};
            \node[state,draw=none] (q_4) [below of= q_0] {};

            \path[->] (q_0) edge node {$a$} (q_1)
                      (q_0) edge [loop above] node {$b$} (q_0)
                      (q_1) edge [loop above] node {$b$} (q_1)
                      (q_1) edge node {$a$} (q_2)
                      (q_2) edge [bend left=10] node {$b$} (q_3)
                      (q_2) edge [loop above] node {$a$} (q_2)
                      (q_3) edge [bend left=10] node {$a,b$} (q_2);

            \draw [->] (-.5,.5) -- (-.27,.27);

        \end{tikzpicture}}}%
        \caption[]{3MSA for automata that do not satisfy conditions 1. or 2. of Proposition \ref{proposition:4-p}.}%
        \label{fig:prop:4-p:d:1}%
      \end{figure}
  \item Let condition $3.$ be false, i.e., $b=(1)(23)\pi$ or $b=(123)\pi$ or $b=(132)\pi$, and  $2a= 2$.
      Then the 3MSA in figures \ref{fig:prop:4-p:d:3:1}, \ref{fig:prop:4-p:d:3:2} and \ref{fig:prop:4-p:d:3:3} prove that $\A$ is not $3$-compressible.
      \begin{figure}[ht]%
        \centering
        \subfigure[][3MSA for the case $b=(1)(23)\pi$ and $2a=2$.]{%
        \label{fig:prop:4-p:d:3:1}%
        {\footnotesize
        \begin{tikzpicture}[shorten >=1pt,node distance=4.6em,auto]

            \node[state] (q_0) {};
            \node[state] (q_1) [right of= q_0] {$3$};
            \node[state] (q_2) [right of= q_1] {$1,3$};
            \node[state] (q_3) [right of= q_2] {$2,3$};
            \node[state] (q_4) [below of= q_1] {$2$};
            \node[state] (q_5) [right of= q_4] {$1,2$};
            \node[state,draw=none] (q_6) [right of= q_5] {};

            \path[->] (q_0) edge node {$a$} (q_1)
                      (q_0) edge [loop above] node {$b$} (q_0)
                      (q_1) edge node {$a$} (q_2)
                      (q_1) edge [bend left=10] node {$b$} (q_4)
                      (q_2) edge [loop above] node {$a$} (q_2)
                      (q_2) edge [bend left=10] node {$b$} (q_5)
                      (q_3) edge node [swap] {$a,b$} (q_2)
                      (q_4) edge [bend left=10] node {$a,b$} (q_1)
                      (q_5) edge [bend left=10] node {$b$} (q_2)
                      (q_5) edge node {$a$} (q_3)
                      (q_6) edge [bend left=30,draw=none] node {\phantom{$a,b$}} (q_4);

            \draw [->] (-.5,.5) -- (-.27,.27);

        \end{tikzpicture}}%
        }
        \hspace{5pt}%
        \subfigure[][3MSA for the case $b=(123)\pi$ and $2a=2$.]{%
        \label{fig:prop:4-p:d:3:2}%
        {\footnotesize
        \begin{tikzpicture}[shorten >=1pt,node distance=4.6em,auto]

            \node[state] (q_0) {};
            \node[state] (q_1) [right of= q_0] {$3$};
            \node[state] (q_2) [right of= q_1] {$1$};
            \node[state] (q_3) [right of= q_2] {$2$};
            \node[state] (q_4) [below of= q_1] {$1,3$};
            \node[state] (q_5) [right of= q_4] {$1,2$};
            \node[state] (q_6) [right of= q_5] {$2,3$};

            \path[->] (q_0) edge node {$a$} (q_1)
                      (q_0) edge [loop above] node {$b$} (q_0)
                      (q_1) edge node {$a$} (q_4)
                      (q_1) edge [bend left=10] node {$b$} (q_2)
                      (q_2) edge node {$b$} (q_3)
                      (q_2) edge [bend left=10] node {$a$} (q_1)
                      (q_3) edge [bend left=30] node {$a,b$} (q_1)
                      (q_4) edge [loop left] node {$a$} (q_4)
                      (q_4) edge node {$b$} (q_5)
                      (q_5) edge node {$a,b$} (q_6)
                      (q_6) edge [bend left=30] node {$a,b$} (q_4);

            \draw [->] (-.5,.5) -- (-.27,.27);

        \end{tikzpicture}}}%
        \hspace{5pt}%
        \subfigure[][3MSA for the case $b=(132)\pi$ and $2a=2$.]{%
        \label{fig:prop:4-p:d:3:3}%
        {\footnotesize
        \begin{tikzpicture}[shorten >=1pt,node distance=4.6em,auto]

            \node[state] (q_0) {};
            \node[state] (q_1) [right of= q_0] {$3$};
            \node[state] (q_2) [right of= q_1] {$2$};
            \node[state] (q_3) [right of= q_2] {$1$};
            \node[state] (q_4) [below of= q_1] {$1,3$};
            \node[state] (q_5) [right of= q_4] {$2,3$};
            \node[state] (q_6) [right of= q_5] {$1,2$};

            \path[->] (q_0) edge node {$a$} (q_1)
                      (q_0) edge [loop above] node {$b$} (q_0)
                      (q_1) edge node {$a$} (q_4)
                      (q_1) edge [bend left=10] node {$b$} (q_2)
                      (q_2) edge node {$b$} (q_3)
                      (q_2) edge [bend left=10] node {$a$} (q_1)
                      (q_3) edge [bend left=30] node {$a,b$} (q_1)
                      (q_4) edge [loop left] node {$a$} (q_4)
                      (q_4) edge [bend left=10] node {$b$} (q_5)
                      (q_5) edge [bend left=10] node {$b$} (q_6)
                      (q_5) edge [bend left=10] node {$a$} (q_4)
                      (q_6) edge [bend left=10] node {$a$} (q_5)
                      (q_6) edge [bend left=30] node {$b$} (q_4)
                      (q_6) edge [bend left=30,draw=none] node {\phantom{$a,b$}} (q_4);

            \draw [->] (-.5,.5) -- (-.27,.27);

        \end{tikzpicture}}}%
        \caption[]{3MSA for automata that do not satisfy condition 3. of Proposition \ref{proposition:4-p}.}%
        \label{fig:prop:4-p:d:3}%
        \end{figure}
  \item Let condition $4.$ be false, i.e., $b=(1)(2)(34)\pi$ or $b=(12)(34)\pi$ or $b=(14)(23)\pi$ or $b=(1423)\pi$ or $b=(1324)\pi$, and $4a= 2$.
      Then the 3MSA in figures \ref{fig:prop:4-p:d:4:1}, \ref{fig:prop:4-p:d:4:3}, \ref{fig:prop:4-p:d:4:4} and \ref{fig:prop:4-p:d:4:5} prove that $\A$ is not $3$-compressible.
      \begin{figure}[ht]%
        \centering
        \subfigure[][3MSA for the case $b=(1)(2)(34)\pi$ or $b=(12)(34)\pi$, and  $4a=2$.]{%
        \label{fig:prop:4-p:d:4:1}%
        {\footnotesize
        \begin{tikzpicture}[shorten >=1pt,node distance=5em,auto]

            \node[state] (q_0) {};
            \node[state] (q_1) [right of= q_0] {$3$};
            \node[state] (q_2) [right of= q_1] {$4$};
            \node[state] (q_3) [right of= q_2] {$2,3$};
            \node[state] (q_4) [below of= q_1] {$1,3$};
            \node[state] (q_5) [right of= q_4] {$1b,4$};
            \node[state] (q_6) [right of= q_5] {$2b,4$};

            \path[->] (q_0) edge node {$a$} (q_1)
                      (q_0) edge [loop above] node {$b$} (q_0)
                      (q_1) edge node {$a$} (q_4)
                      (q_1) edge [bend left=10] node {$b$} (q_2)
                      (q_2) edge node {$a$} (q_3)
                      (q_2) edge [bend left=10] node {$b$} (q_1)
                      (q_3) edge [bend left=10] node {$b$} (q_6)
                      (q_3) edge [bend right=5] node [swap] {$a$} (q_4)
                      (q_4) edge [loop left] node {$a$} (q_4)
                      (q_4) edge [bend left=10] node {$b$} (q_5)
                      (q_5) edge [bend left=10] node {$b$} (q_4)
                      (q_5) edge node {$a$} (q_3)
                      (q_6) edge [bend left=10] node {$a,b$} (q_3);

            \draw [->] (-.5,.5) -- (-.27,.27);

        \end{tikzpicture}}%
        }
        \hspace{20pt}%
        \subfigure[][3MSA for the case $b=(14)(23)\pi$ and $4a=2$.]{%
        \label{fig:prop:4-p:d:4:3}%
        {\footnotesize
        \begin{tikzpicture}[shorten >=1pt,node distance=5em,auto]

            \node[state] (q_0) {};
            \node[state] (q_1) [right of= q_0] {$3$};
            \node[state] (q_2) [right of= q_1] {$1,3$};
            \node[state] (q_3) [right of= q_2] {$2,3$};
            \node[state] (q_4) [below of= q_1] {$2$};
            \node[state] (q_5) [right of= q_4] {$2,4$};

            \path[->] (q_0) edge node {$a$} (q_1)
                      (q_0) edge [loop above] node {$b$} (q_0)
                      (q_1) edge node {$a$} (q_2)
                      (q_1) edge [bend left=10] node {$b$} (q_4)
                      (q_2) edge [loop above] node {$a$} (q_2)
                      (q_2) edge [bend left=10] node {$b$} (q_5)
                      (q_3) edge node [swap] {$a$} (q_2)
                      (q_3) edge [loop above] node {$b$} (q_3)
                      (q_4) edge [bend left=10] node {$a,b$} (q_1)
                      (q_5) edge [bend left=10] node {$b$} (q_2)
                      (q_5) edge node {$a$} (q_3);

            \draw [->] (-.5,.5) -- (-.27,.27);

        \end{tikzpicture}}}%
        \\
        \subfigure[][3MSA for the case $b=(1423)\pi$ and $4a=2$.]{%
        \label{fig:prop:4-p:d:4:4}%
        {\footnotesize
        \begin{tikzpicture}[shorten >=1pt,node distance=5em,auto]

            \node[state] (q_0) {};
            \node[state] (q_1) [right of= q_0] {$3$};
            \node[state] (q_2) [right of= q_1] {$1$};
            \node[state] (q_3) [right of= q_2] {$4$};
            \node[state] (q_4) [right of= q_3] {$2$};
            \node[state] (q_5) [below of= q_1] {$1,3$};
            \node[state] (q_6) [right of= q_5] {$1,4$};
            \node[state] (q_7) [right of= q_6] {$2,4$};
            \node[state] (q_8) [right of= q_7] {$2,3$};

            \path[->] (q_0) edge node {$a$} (q_1)
                      (q_0) edge [loop above] node {$b$} (q_0)
                      (q_1) edge node {$a$} (q_5)
                      (q_1) edge [bend left=10] node {$b$} (q_2)
                      (q_2) edge node {$b$} (q_3)
                      (q_2) edge [bend left=10] node {$a$} (q_1)
                      (q_3) edge node {$b$} (q_4)
                      (q_3) edge node {$a$} (q_8)
                      (q_4) edge [bend right=30] node [swap] {$a,b$} (q_1)
                      (q_5) edge [loop left] node {$a$} (q_5)
                      (q_5) edge node {$b$} (q_6)
                      (q_6) edge node {$b$} (q_7)
                      (q_6) edge [bend right=30] node [swap] {$a$} (q_8)
                      (q_7) edge node {$a,b$} (q_8)
                      (q_8) edge [bend right=30] node [swap] {$a,b$} (q_5)
                      (q_8) edge [bend left=30,draw=none] node {\phantom{$b$}} (q_5);

            \draw [->] (-.5,.5) -- (-.27,.27);

        \end{tikzpicture}}%
        }
        \hspace{20pt}%
        \subfigure[][3MSA for the case $b=(1324)\pi$ and $4a=2$.]{%
        \label{fig:prop:4-p:d:4:5}%
        {\footnotesize
        \begin{tikzpicture}[shorten >=1pt,node distance=5em,auto]

            \node[state] (q_0) {};
            \node[state] (q_1) [right of= q_0] {$3$};
            \node[state] (q_2) [right of= q_1] {$2$};
            \node[state] (q_3) [right of= q_2] {$4$};
            \node[state] (q_4) [right of= q_3] {$1$};
            \node[state] (q_5) [below of= q_1] {$1,3$};
            \node[state] (q_6) [right of= q_5] {$2,3$};
            \node[state] (q_7) [right of= q_6] {$2,4$};
            \node[state] (q_8) [right of= q_7] {$1,4$};

            \path[->] (q_0) edge node {$a$} (q_1)
                      (q_0) edge [loop above] node {$b$} (q_0)
                      (q_1) edge node {$a$} (q_5)
                      (q_1) edge [bend left=10] node {$b$} (q_2)
                      (q_2) edge node {$b$} (q_3)
                      (q_2) edge [bend left=10] node {$a$} (q_1)
                      (q_3) edge node {$b$} (q_4)
                      (q_3) edge node [swap] {$a$} (q_6)
                      (q_4) edge [bend right=30] node [swap] {$a,b$} (q_1)
                      (q_5) edge [loop left] node {$a$} (q_5)
                      (q_5) edge [bend left=10] node {$b$} (q_6)
                      (q_6) edge [bend left=10] node {$a$} (q_5)
                      (q_6) edge [bend left=10] node {$b$} (q_7)
                      (q_7) edge node {$b$} (q_8)
                      (q_7) edge [bend left=10] node {$a$} (q_6)
                      (q_8) edge [bend right=30] node [swap] {$a$} (q_6)
                      (q_8) edge [bend left=30] node {$b$} (q_5);

            \draw [->] (-.5,.5) -- (-.27,.27);

        \end{tikzpicture}}}%
        \caption[]{3MSA for automata that do not satisfy condition $4.$ of Proposition \ref{proposition:4-p}.}%
        \label{fig:prop:4-p:d:4-1-2-3}%
      \end{figure}
\end{enumerate}

\noindent Conversely, let $\A$ be an automaton satisfying conditions $1.-4.$.

If $Orb_b(1,3)\subseteq\{1,2,3\}$, then $Orb_b(1,3)=\{1,2,3\}$, as $\{1,3\}b\neq\{1,3\}$.
There are two subcases:
    \begin{enumerate}
        \item if $b=(1)(23)\pi$ or $b=(123)\pi$, and $2a\neq 2$, then $\M(a^2b)=\{1,2\}$, $\M(a^2ba)=\{2a,3\}$, and $\M(a^2ba^2)=\{1,3,2a^2\}$, and the word $a^2ba^2$ $3$-compresses $\A$;
        \item if $b=(132)\pi$ and $2a\neq 2$, then $\M(a^2b)=\{2,3\}$, $\M(a^2b^2)=\{1,2\}$, $\M(a^2b^2a)=\{2a,3\}$, and $\M(a^2b^2a^2)=\{1,3,2a^2\}$, and the word $a^2b^2a^2$ $3$-compresses $\A$.
    \end{enumerate}

Let now $Orb_b(1,3)\nsubseteq\{1,2,3\}$, we distinguish four subcases by considering the cardinality of $Orb_b(3)$.
    \begin{enumerate}
        \item Let $Orb_b(3)=\{3\}$, then there are two further subcases:
            \begin{enumerate}
                \item if $b=(14\ldots)(3)\pi$, then $\M(a^2b)=\{3,4\}$, $\M(a^2ba)=\{1,3,4a\}$, and the word $a^2ba$ $3$-compresses $\A$;
                \item if $b=(124\ldots)(3)\pi$, then $\M(a^2b)=\{2,3\}$, $\M(a^2b^2)=\{3,4\}$, $\M(a^2b^2a)=\{1,3,4a\}$, and the word $a^2b^2a$ $3$-compresses $\A$.
            \end{enumerate}
        \item Let $Orb_b(3)=\{2,3\}$, then there are two further subcases:
            \begin{enumerate}
                \item if $b=(14)(23)\pi$, and $4a\neq 2$, then $\M(a^2b)=\{2,4\}$, $\M(a^2ba)=\{3,4a\}$, $\M(a^2ba^2)=\{1,3,4a^2\}$, and the word $a^2ba^2$ $3$-compresses $\A$;
                \item if $b=(145\ldots)(23)\pi$, then $\M(a^2b)=\{2,4\}$, $\M(a^2b^2)=\{3,5\}$, $\M(a^2b^2a)=\{1,3,5a\}$, and the word $a^2b^2a$ $3$-compresses $\A$.
            \end{enumerate}
        \item Let $Orb_b(3)=\{3,4\}$, i.e., $b=(34)\pi$, then in Fig. \ref{fig:caso:4-p:r:1} we draw a P3MSA for this case, proving that either the word $a^2b^2a$ or $a^2ba$ or $aba^2$ or $ababa$ $3$-compresses $\A$. Observe that if $1b=1$, then $2b\neq 2$.
            \begin{figure}[ht]%
                \centering
                {\footnotesize
                \begin{tikzpicture}[shorten >=1pt,node distance=5em,auto]

                    \node[state] (q_0) {};
                    \node[state] (q_1) [right of= q_0] {$3$};
                    \node[state,draw=none] (q_2) [right of= q_1] {};
                    \node[state] (q_3) [right of= q_2] {$1,3$};
                    \node[state] (q_4) [right of= q_3] {$1b,4$};
                    \node[state,draw=none] (q_5) [right of= q_4] {};
                    \node[state] (q_55) [right of= q_5] {$2b,3$};
                    \node[state] (q_6) [below of= q_0] {$4$};
                    \node[state] (q_7) [right of= q_6] {$3,4a$};
                    \node[state] (q_8) [right of= q_7] {$4$};
                    \node[state] (q_9) [right of= q_8] {$2,3$};
                    \node[state] (q_10) [right of= q_9] {$2b,4$};
                    \node[state,accepting] (q_11) [right of= q_10] {};

                    \path[->] (q_0) edge node {$a$} (q_1)
                              (q_1) edge node {$a|4a=2,1b\neq 1$} (q_3)
                              (q_3) edge node {$b$} (q_4)
                              (q_4) edge node {$b|1b=2$} (q_55)
                              (q_4) edge node {$a|1b\neq 2$} (q_11)
                              (q_1) edge node [swap] {$b|4a\neq 2$} (q_6)
                              (q_6) edge node {$a$} (q_7)
                              (q_7) edge [bend right=20] node [swap] {$a$} (q_11)
                              (q_1) edge node {$b|4a=2,1b=1$} (q_8)
                              (q_8) edge node {$a$} (q_9)
                              (q_9) edge node {$b$} (q_10)
                              (q_10) edge node {$a$} (q_11)
                              (q_55) edge node {$a$} (q_11);

                    \draw [->] (-.5,.5) -- (-.27,.27);

                \end{tikzpicture}}\\
                \caption[]{P3MSA for the case $b=(34)\pi$.}%
                \label{fig:caso:4-p:r:1}%
            \end{figure}
        \item Let $|Orb_b(3)|\geq 3$, then there are three main subcases:
            \begin{enumerate}
                \item $3b=1$.
                    \begin{enumerate}
                      \item If $b=(3124\ldots)\pi$, then $\M(a^2b)=\{1,2\}$, $\M(a^2b^2)=\{2,4\}$, $\M(a^2b^3)=\{4,4b\}$, $\M(a^2b^3a)=\{3,4a,4ba\}$ and the word $a^2b^3a$ $3$-compresses $\A$.
                      \item If $b=(314\ldots)\pi$, then the P3MSA in Fig. \ref{fig:caso:4-p:r:2} proves that either the word $a^2baba^2$ or $a^2b^2a$ or $ab^2a^2$ $3$-compresses $\A$.
                          Observe that if $4a=2$ and $4b=2$, then $b$ is of the form $(31425\ldots)\pi$, otherwise condition $4.$ it is not satisfied.
        \begin{figure}[ht]%
            \centering
            {\footnotesize
            \begin{tikzpicture}[shorten >=1pt,node distance=5em,auto]

                \node[state] (q_0) {};
                \node[state] (q_1) [right of= q_0] {$3$};
                \node[state,draw=none] (q_2) [right of= q_1] {};
                \node[state] (q_3) [right of= q_2] {$1,3$};
                \node[state] (q_4) [right of= q_3] {$1,4$};
                \node[state,draw=none] (q_5) [right of= q_4] {};
                \node[state] (q_6) [right of= q_5] {$4,4b$};
                \node[state] (q_7) [below of= q_0] {$1$};
                \node[state] (q_8) [right of= q_7] {$4$};
                \node[state] (q_9) [right of= q_8] {$3,4a$};
                \node[state] (q_10) [right of= q_9] {$2,3$};
                \node[state] (q_11) [right of= q_10] {$1,5$};
                \node[state] (q_12) [right of= q_11] {$3,5a$};
                \node[state,accepting] (q_13) [right of= q_12] {};

                \path[->] (q_0) edge node {$a$} (q_1)
                          (q_1) edge node {$a|4a=2$} (q_3)
                          (q_1) edge node [swap] {$b|4a\neq 2$} (q_7)
                          (q_3) edge node {$b$} (q_4)
                          (q_4) edge node {$b|4b\neq2$} (q_6)
                          (q_4) edge node [swap] {$a|4b=2$} (q_10)
                          (q_6) edge node {$a$} (q_13)
                          (q_7) edge node {$b$} (q_8)
                          (q_8) edge node {$a$} (q_9)
                          (q_9) edge [bend right=20] node [swap] {$a$} (q_13)
                          (q_10) edge node {$b$} (q_11)
                          (q_11) edge node {$a$} (q_12)
                          (q_12) edge node {$a$} (q_13);

                \draw [->] (-.5,.5) -- (-.27,.27);

            \end{tikzpicture}}\\
            \caption[]{P3MSA for the case $b=(314\ldots)\pi$.}%
            \label{fig:caso:4-p:r:2}%
        \end{figure}
                    \end{enumerate}
                \item $3b=2$.
                    \begin{enumerate}
                      \item If $b=(3214\ldots)\pi$, then the P3MSA in Fig. \ref{fig:caso:4-p:r:3} proves that either the word $a^2ba^2$ or $ab^3ab^3a$ $3$-compresses $\A$.
        \begin{figure}[ht]%
            \centering
            {\footnotesize
            \begin{tikzpicture}[shorten >=1pt,node distance=5em,auto]

                \node[state] (q_0) {};
                \node[state] (q_1) [right of= q_0] {$3$};
                \node[state,draw=none] (q_2) [right of= q_1] {};
                \node[state] (q_3) [right of= q_2] {$1,3$};
                \node[state] (q_4) [right of= q_3] {$2,4$};
                \node[state] (q_5) [right of= q_4] {$3,4a$};
                \node[state,accepting] (q_6) [right of= q_5] {};
                \node[state] (q_7) [below of= q_0] {$2$};
                \node[state] (q_8) [right of= q_7] {$1$};
                \node[state] (q_9) [right of= q_8] {$4$};
                \node[state] (q_10) [right of= q_9] {$2,3$};
                \node[state] (q_11) [right of= q_10] {$1,2$};
                \node[state] (q_12) [right of= q_11] {$1,4$};
                \node[state] (q_13) [right of= q_12] {$4,4b$};

                \path[->] (q_0) edge node {$a$} (q_1)
                          (q_1) edge node {$a|4a\neq 2$} (q_3)
                          (q_1) edge node [swap] {$b|4a=2$} (q_7)
                          (q_3) edge node {$b$} (q_4)
                          (q_4) edge node {$a$} (q_5)
                          (q_5) edge node {$a$} (q_6)
                          (q_7) edge node {$b$} (q_8)
                          (q_8) edge node {$b$} (q_9)
                          (q_9) edge node {$a$} (q_10)
                          (q_10) edge node {$b$} (q_11)
                          (q_11) edge node {$b$} (q_12)
                          (q_12) edge node {$b$} (q_13)
                          (q_13) edge node {$a$} (q_6);

                \draw [->] (-.5,.5) -- (-.27,.27);

            \end{tikzpicture}}\\
            \caption[]{P3MSA for the case $b=(3214\ldots)\pi$.}%
            \label{fig:caso:4-p:r:3}%
        \end{figure}
                      \item If $b=(324\ldots)\pi$, then in Fig. \ref{fig:caso:4-p:r:4} we draw a P3MSA for this case, proving that either the word $ab^2a^2$ or $a^2b^3a$ or $a^2ba^2$ $3$-compresses $\A$.
                          Observe that $1b\neq 3$ implies $1ba\neq 1$.
                          Moreover if $4a=2$, then $1ba\neq 2$, otherwise the contradiction $1b=4$ arises.
                          So when $4a=2$ and $1b\neq 3$ it is $1ba\not\in\{1,2\}$.
        \begin{figure}[ht]%
            \centering
            {\footnotesize
            \begin{tikzpicture}[shorten >=1pt,node distance=5em,auto]

                \node[state] (q_0) {};
                \node[state] (q_1) [right of= q_0] {$3$};
                \node[state,draw=none] (q_2) [right of= q_1] {};
                \node[state] (q_3) [right of= q_2] {$2$};
                \node[state] (q_4) [right of= q_3] {$4$};
                \node[state] (q_5) [right of= q_4] {$3,4a$};
                \node[state,accepting] (q_6) [right of= q_5] {};
                \node[state] (q_7) [below of= q_0] {$1,3$};
                \node[state,draw=none] (q_8) [right of= q_7] {};
                \node[state] (q_9) [right of= q_8] {$2,3$};
                \node[state] (q_10) [right of= q_9] {$2,4$};
                \node[state] (q_11) [right of= q_10] {$4,4b$};
                \node[state] (q_12) [right of= q_11] {$1b,2$};
                \node[state] (q_13) [right of= q_12] {$1ba,3$};

                \path[->] (q_0) edge node {$a$} (q_1)
                          (q_1) edge node {$b|4a\neq 2$} (q_3)
                          (q_1) edge node [swap] {$a|4a=2$} (q_7)
                          (q_3) edge node {$b$} (q_4)
                          (q_4) edge node {$a$} (q_5)
                          (q_5) edge node {$a$} (q_6)
                          (q_7) edge node {$b|1b=3$} (q_9)
                          (q_7) edge [bend right=20] node [swap] {$b|1b\neq 3$} (q_12)
                          (q_9) edge node {$b$} (q_10)
                          (q_10) edge node {$b$} (q_11)
                          (q_12) edge node {$a$} (q_13)
                          (q_13) edge node {$a$} (q_6)
                          (q_11) edge node {$a$} (q_6);

                \draw [->] (-.5,.5) -- (-.27,.27);

            \end{tikzpicture}}\\
            \caption[]{P3MSA for the case $b=(324\ldots)\pi$.}%
            \label{fig:caso:4-p:r:4}%
        \end{figure}
                    \end{enumerate}
                \item $3b=4$.
                    \begin{enumerate}
                      \item If $1b\not\in\{1,2\}$, then $\M(a^2b)=\{1b,4\}$, $\M(a^2ba)=\{1ba,3,4a\}$, and the word $a^2ba$ $3$-compresses $\A$.
                      \item If $1b\in\{1,2\}$ and $4a\neq 2$, then $\M(aba)=\{4a,3\}$ and $\M(aba^2)=\{4a^2,1,3\}$, and the word $aba^2$ $3$-compresses $\A$.
                      \item If $1b\in\{1,2\}$, $4a\neq 2$ and $2b\in\{1,2\}$, then $4b\not\in\{1,2,3\}$, so $\M(aba)=\{4ba,3\}$ (with $4ba\neq 2$) and $\M(aba^2)=\{4ba^2,1,3\}$, and the word $aba^2$ $3$-compresses $\A$.
                      \item If $1b\in\{1,2\}$, $4a\neq 2$ and $2b\not\in\{1,2\}$, then $\M(aba)=\{2,3\}$, $\M(abab)=\{2b,4\}$ and $\M(ababa)=\{2ba,2,3\}$, and the word $ababa$ $3$-compresses $\A$.
                    \end{enumerate}
            \end{enumerate}
\end{enumerate}
\end{proof}

\FloatBarrier 
\section{\texorpdfstring{$3$-compressible automata without permutations}{3-compressible automata without permutations}}\label{nopermutation}

In this section we characterize proper $3$-compressible automata on $2$-letter alphabet where no letter acts as permutation.


\begin{proposition}\label{proposition:1-2-4}
Let $\A$ a $(\bf{i},\bf{j})$-automaton with $i\in\{1,2\}$ and $j\in\{1,2,4\}$, then $\A$ is either not $3$-compressible or not proper.
\end{proposition}

\begin{proof}
We have to consider five different cases.
\begin{enumerate}
  \item Let $\A$ be a $(\bf{1},\bf{1})$-automaton with $a=[1,2,3]\backslash 1,2$ and $b=[x,y,z]\backslash x,y$:
      \begin{enumerate}
        \item if $\{1,2\}\subseteq\{x,y,z\}$ and $\{x,y\}\subseteq\{1,2,3\}$, then for any $w\in\{a,b\}^*$, we have $\M(wa)=\{1,2\}$ and $\M(wb)=\{x,y\}$, so $\A$ is not 3-compressible;
        \item if $\{1,2\}\nsubseteq\{x,y,z\}$, then $|\M(ab)|\geq 3$, so $\A$ is not proper; similarly if $\{x,y\}\nsubseteq\{1,2,3\}$, then $|\M(ba)|\geq 3$, and again $\A$ is not proper.
      \end{enumerate}
  \item Let $\A$ be a $(\bf{1},\bf{2})$-automaton with $a=[1,2,3]\backslash 1,2$ and $b=[x,y][z,v]\backslash x,z$:
      \begin{enumerate}
        \item if $\{1,2\}\in\{\{x,z\},\{x,v\},\{y,z\},\{y,v\}\}$ and $\{x,z\}\subseteq\{1,2,3\}$, then for any $w\in\{a,b\}^*$, we have $\M(wa)=\{1,2\}$ and $\M(wb)=\{x,z\}$, so $\A$ is not $3$-compressible;
        \item if $\{1,2\}\not\in\{\{x,z\},\{x,v\},\{y,z\},\{y,v\}\}$, then $|\M(ab)|\geq 3$, so $\A$ is not proper; similarly if $\{x,z\}\nsubseteq\{1,2,3\}$, then $|\M(ba)|\geq 3$, and again $\A$ is not proper.
      \end{enumerate}
  \item Let $\A$ be a $(\bf{2},\bf{2})$-automaton with $a=[1,2][3,4]\backslash 1,3$ and $b=[x,y][z,v]\backslash x,z$:
      \begin{enumerate}
        \item if $\{1,3\}\in\{\{x,z\},\{x,v\},\{y,z\},\{y,v\}\}$ and $\{x,z\}\in\{\{1,3\},\{1,4\},\{2,3\},\{2,4\}\}$, then for all $w\in\{a,b\}^*$, we have $\M(wa)=\{1,3\}$ and $\M(wb)=\{x,z\}$, so $\A$ is not $3$-compressible;
        \item if $\{1,3\}\not\in\{\{x,z\},\{x,v\},\{y,z\},\{y,v\}\}$, then $|\M(ab)|\geq 3$, so $\A$ is not proper; similarly if $\{x,z\}\not\in\{\{1,3\},\{1,4\},\{2,3\},\{2,4\}\}$, then $|\M(ba)|\geq 3$, and again $\A$ is not proper.
      \end{enumerate}
  \item Let $\A$ be a $(\bf{1},\bf{4})$-automaton with $a=[1,2,3]\backslash 1,2$ and $b=[x,y]\backslash z$, $zb=x$:
      \begin{enumerate}
        \item if $\{x,z\}\nsubseteq\{1,2,3\}$, then $\M(b^2)=\{x,z\}$ and $\M(b^2a)=\{1,2,xa,za\}$, so $\A$ is not proper;
        \item if $\{x,z\}=\{1,2\}$, then for all $w\in\{a,b\}^+\setminus\{b\}$, we have $\M(w)=\{1,2\}$, so $\A$ is not $3$-compressible;
        \item if $\{x,z\}\in\{\{1,3\},\{2,3\}\}$, let $q\in\{1,2,3\}\setminus\{x,z\}$. Then $\M(ab)=\{z,qb\}$ and $\M(ab^2)=\{x,z,qb^2\}$, so $\A$ is not proper.
      \end{enumerate}
  \item Let $\A$ be a $(\bf{2},\bf{4})$-automaton with $a=[1,2][3,4]\backslash 1,3$ and $b=[x,y]\backslash z$, $zb=x$:
      \begin{enumerate}
        \item if $\{x,z\}\not\in\{\{1,3\},\{1,4\},\{2,3\},\{2,4\}\}$, then $\M(b^2)=\{x,z\}$ and $|\M(b^2a)|\geq 3$, so $\A$ is not proper;
        \item if $\{x,y\}\cap\{1,3\}=\emptyset$, then $\M(ab)=\{1b,3b,z\}$, so $\A$ is not proper;
        \item if $\{x,z\}=\{1,3\}$, then for all $w\in\{a,b\}^+\setminus\{b\}$, we have $\M(w)=\{1,3\}$, so $\A$ is not $3$-compressible;
        \item if $\{x,z\}=\{1,4\}$ and $\{x,y\}\cap\{1,3\}\neq\emptyset$, we consider two subcases:
          \begin{enumerate}
            \item if $x=4$ and $z=1$, then $y$ must be equal to 3. So $b=[4,3]\backslash 1$, then for all $w\in\{a,b\}^+$, we have $\M(wa)=\{1,3\}$ and $\M(wb)=\{1,4\}$, so $\A$ is not $3$-compressible;
            \item if $x=1$ and $z=4$, then if $y\neq 2$ the P3MSA in Fig. \ref{fig:caso:2-4:1-4:2} proves that $\A$ is not proper. Else, if $y=2$ and $3b\neq 2$, then the word $aba$ $3$-compresses $\A$ which is not proper. Otherwise, if $y=2$ and $3b=2$, then the 3MSA in Fig. \ref{fig:caso:2-4:1-4} proves that $\A$ is not $3$-compressible;
          \end{enumerate}
        \item if $\{x,z\}=\{2,3\}$, this case reduces to the previous exchanging the state $1$ with $3$ and $2$ with $4$;
        \item if $\{x,z\}=\{2,4\}$ and $\{x,y\}\cap\{1,3\}\neq\emptyset$ then $y\in\{1,3\}$. Let $q\in\{1,3\}\setminus\{y\}$.
            \begin{enumerate}
              \item If $b=[2,1]\backslash 4$ and $q=3$, or $b=[4,3]\backslash 2$ and $q=1$ then, if $qb\neq y$ we have $\M(ab)=\{z,qb\}$ and $\M(aba)=\{1,3,qba\}$, so the automaton is not proper, else if $qb=y$ $\A$ is not $3$-compressible, as shown in Fig. \ref{fig:caso:2-4:2-4:1}.
              \item Let $b=[2,3]\backslash 4$ and $q=1$, or $b=[4,1]\backslash 2$ and $q=3$ then either the word $ab^2$ or $aba$ $3$-compresses $\A$, so it is not proper.
                  \begin{figure}[ht]%
                    \centering
                    \subfigure[][P3MSA for the case $b={[1,y]}\backslash 4$, $4b=1$, $y\neq 2$.]{%
                    \label{fig:caso:2-4:1-4:2}%
                    {\footnotesize
                    \begin{tikzpicture}[shorten >=1pt,node distance=5em,auto]

                        \node[state] (q_0) {};
                        \node[state] (q_1) [right of= q_0] {$1,3$};
                        \node[state] (q_3) [right=4.5em of q_1] {$4,3b$};
                        \node[state] (q_4) [below of= q_1] {$2,4$};
                        \node[state,accepting] (q_5) [below of= q_3] {};

                        \path[->] (q_0) edge node {$a$} (q_1)
                                  (q_1) edge node {$b|3b\neq 2$} (q_3)
                                  (q_1) edge node {$b|3b=2$} (q_4)
                                  (q_3) edge node {$a$} (q_5)
                                  (q_4) edge node {$b$} (q_5);

                        \draw [->] (-.5,.5) -- (-.27,.27);

                    \end{tikzpicture}}}%
                    \hspace{30pt}%
                    \subfigure[][3MSA for the case $b={[1,2]}\backslash 4$, $4b=1$, $3b=2$, or $b={[2,1]}\backslash 4$, $4b=1$ and $3b= 1$, or $b={[4,3]}\backslash 2$, $2b=4$ and $1b= 3$.]{%
                    \label{fig:caso:2-4:2-4:1}%
                    {\footnotesize
                    \begin{tikzpicture}[shorten >=1pt,node distance=5em,auto]

                        \node[state] (q_0) {};
                        \node[state] (q_1) [right of= q_0] {$1,3$};
                        \node[state] (q_2) [right of= q_1] {$y,z$};
                        \node[state] (qq) [draw=none,right=8.1em of q_1] {};
                        \node[state] (q_3) [below of= q_0] {$z$};
                        \node[state] (q_4) [below of= q_1] {$x,z$};
                        \node[state,draw=none] (q) [right of= q_3] {};

                        \path[->] (q_0) edge node {$a$} (q_1)
                                  (q_0) edge node {$b$} (q_3)
                                  (q_1) edge [loop above] node {$a$} (q_1)
                                  (q_1) edge [bend left=10] node {$b$} (q_2)
                                  (q_2) edge node {$b$} (q_4)
                                  (q_2) edge [bend left=10] node {$a$} (q_1)
                                  (q_3) edge node {$b$} (q_4)
                                  (q_3) edge node {$a$} (q_1)
                                  (q_4) edge [loop right] node {$b$} (q_4)
                                  (q_4) edge node {$a$} (q_1);

                        \draw [->] (-.5,.5) -- (-.27,.27);

                    \end{tikzpicture}}}%
                    \caption[]{3MSA for a $(\bf{2},\bf{4})$-automaton with $a=[1,2][3,4]\backslash 1,3$ and $b$ satisfying either condition (d).ii or (f).i.}%
                    \label{fig:caso:2-4:1-4}%
                \end{figure}
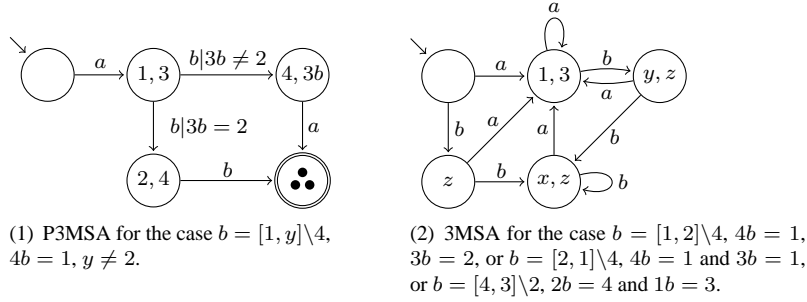
            \end{enumerate}
      \end{enumerate}
\end{enumerate}
\end{proof}

\FloatBarrier


\begin{proposition}\label{proposition:1-3}
Let $\A$ be a $(\bf{1},\bf{3})$-automaton with $a=[1,2,3]\backslash 1,2$ and $b=[x,y]\backslash x$.
Then $\A$ is $3$-compressible and proper if, and only if, the following conditions hold:
\begin{enumerate}
  \item $x\in\{1,2,3\}$, and
  \item $\{1,2\}\cap\{x,y\}\neq\emptyset$, and
  \item for all $q\in\{1,2\}\setminus\{x\}$, $qb\in\{1,2,3\}$, and $Orb_b(q)\not\subseteq\{1,2,3\}$.
\end{enumerate}
Moreover, if $\mathcal{A}$ is proper and $3$-compressible, then the word $ab^2a$ $3$-compresses $\A$.
\end{proposition}

\begin{proof}
Let $\A$ be a $(\bf{1},\bf{3})$-automaton that does not satisfy one of the conditions $1.-3.$
\begin{enumerate}
  \item Let $x\not\in\{1,2,3\}$, then the word $ba$ $3$-compresses $\A$, and so it is not proper.
  \item Let $\{1,2\}\cap\{x,y\}=\emptyset$, then the word $ab$ $3$-compresses $\A$, and so it is not proper.
  \item Let for all $q\in\{1,2\}\setminus\{x\}$ be $qb\not\in\{1,2,3\}$ or $Orb_b(q)\subseteq\{1,2,3\}$, and $x\in\{1,2,3\}$, and $\{1,2\}\cap\{x,y\}\neq\emptyset$.
      \begin{enumerate}
        \item Let $qb\not\in\{1,2,3\}$, then if $x\in\{1,2\}$, we have $\M(ab)=\{qb,x\}$ and $\M(aba)=\{1,2,qba\}$, else if $x=3$, we choose $q\in\{1,2\}\setminus\{y\}$, and again $\M(ab)=\{qb,x\}$ and $\M(aba)=\{1,2,qba\}$.
            So if $qb\not\in\{1,2,3\}$ the word $aba$ $3$-compresses $\A$ which is not proper.
        \item So let $qb\in\{1,2,3\}$ and $Orb_b(q)\subseteq\{1,2,3\}$.
            \begin{enumerate}
                \item If $x\in\{1,2\}$, then $qb\in\{q,3\}$ and $3b=q$, so $\A$ is not $3$-compressible, as shown in Fig. \ref{fig:prop:1-3:1}.
                \item If $x=3$, then $y\in\{1,2\}$ and $\{1b,2b\}=\{1,2\}$, and $\A$ is not $3$-compressible, as shown in Fig. \ref{fig:prop:1-3:2}, where in this case $q\in\{1,2\}\setminus\{y\}$.
            \begin{figure}[ht]%
                \centering
                \subfigure[][3MSA for the case $x\in\{1,2\}$.]{%
                \label{fig:prop:1-3:1}%
                {\footnotesize
                \begin{tikzpicture}[shorten >=1pt,node distance=5em,auto]

                    \node[state] (q_0) {};
                    \node[state] (q_1) [right of= q_0] {$1,2$};
                    \node[state] (q_2) [right of= q_1,draw=none] {};
                    \node[state] (q_3) [right of= q_2] {$x,3$};
                    \node[state] (q_4) [below of= q_0] {$x$};
                    \node[state] (q_5) [right of= q_4] {$1,2$};
                    \node[state] (q_6) [right of= q_5] {$x,q$};
                    \node[state] (q_7) [right of= q_6] {$1,2$};

                    \path[->] (q_0) edge node {$a$} (q_1)
                              (q_0) edge node {$b$} (q_4)
                              (q_1) edge [loop above] node {$a$} (q_1)
                              (q_1) edge node {$b|qb=3$} (q_3)
                              (q_1) edge node {$b|qb=q$} (q_5)
                              (q_3) edge [bend left=10] node {$a$} (q_7)
                              (q_3) edge [bend left=10] node {$b$} (q_6)
                              (q_4) edge [loop below] node {$b$} (q_4)
                              (q_4) edge node {$a$} (q_1)
                              (q_5) edge [loop below] node {$a,b$} (q_5)
                              (q_6) edge node {$a$} (q_7)
                              (q_6) edge [bend left=10] node {$b$} (q_3)
                              (q_7) edge [bend left=10] node {$b$} (q_3)
                              (q_7) edge [loop below] node {$a$} (q_7);

                    \draw [->] (-.5,.5) -- (-.27,.27);

                \end{tikzpicture}}%
                }
                \hspace{20pt}%
                \subfigure[][3MSA for the case $x=3$.]{%
                \label{fig:prop:1-3:2}%
                {\footnotesize
                \begin{tikzpicture}[shorten >=1pt,node distance=5em,auto]

                    \node[state] (q_0) {};
                    \node[state] (q_1) [right of= q_0] {$1,2$};
                    \node[state] (q_2) [right of= q_1,draw=none] {};
                    \node[state] (q_3) [right of= q_2] {$3,q$};
                    \node[state] (q_4) [right of= q_3] {$1,2$};
                    \node[state] (q_5) [below of= q_0] {$1$};
                    \node[state] (q_6) [right of= q_5] {$3,q$};
                    \node[state] (q_7) [right of= q_6] {$3,y$};
                    \node[state] (q_8) [right of= q_7] {$1,2$};

                    \path[->] (q_0) edge node {$a$} (q_1)
                              (q_0) edge node {$b$} (q_5)
                              (q_1) edge [loop above] node {$a$} (q_1)
                              (q_1) edge node {$b|qb=q$} (q_3)
                              (q_1) edge node {$b|qb=y$} (q_7)
                              (q_3) edge [loop above] node {$b$} (q_3)
                              (q_3) edge [bend left=10] node {$a$} (q_4)
                              (q_4) edge [loop above] node {$a$} (q_4)
                              (q_4) edge [bend left=10] node {$b$} (q_3)
                              (q_5) edge [loop below] node {$b$} (q_5)
                              (q_5) edge node {$a$} (q_1)
                              (q_6) edge [bend right=35] node [swap] {$a$} (q_8)
                              (q_6) edge [bend left=10] node {$b$} (q_7)
                              (q_7) edge [bend left=10] node {$b$} (q_6)
                              (q_7) edge [bend left=10] node {$a$} (q_8)
                          (q_8) edge [loop below] node {$a$} (q_8)
                          (q_8) edge [bend left=10] node {$b$} (q_7)

                          (q_6) edge [loop below,draw=none] node {\phantom{$a,b$}} (q_6);

                \draw [->] (-.5,.5) -- (-.27,.27);

                \end{tikzpicture}}}%
                \caption[]{3MSA for automata that do not satisfy condition 3. of Proposition \ref{proposition:1-3} with $zb\in\{1,2,3\}$ and $Orb_b(z)\subseteq\{1,2,3\}$.}%
                \label{fig:prop:1-3}%
            \end{figure}
            \end{enumerate}
        \end{enumerate}
\end{enumerate}
Conversely, suppose $x\in\{1,2,3\}$, $\{1,2\}\cap\{x,y\}\neq\emptyset$, and for all $q\in\{1,2\}\setminus\{x\}$, $qb\in\{1,2,3\}$ and $Orb_b(q)\not\subseteq\{1,2,3\}$.
If $x\in\{1,2\}$, let $q\in\{1,2\}\setminus\{x\}$, else let $q\in\{1,2\}\setminus\{y\}$; then $\M(ab)=\{x,qb\}$, $\M(ab^2)=\{x,qb^2\}$, and, as $qb^2\not\in\{1,2,3\}$, $\M(ab^2a)=\{1,2,qb^2a\}$, and the word $ab^2a$ $3$-compresses $\A$.
\end{proof}

\FloatBarrier


\begin{proposition}\label{proposition:2-3}
Let $\A$ be a $(\bf{2},\bf{3})$-automaton with $a=[1,2][3,4]\backslash 1,3$ and $b=[x,y]\backslash x$. Then $\A$ is $3$-compressible and proper if, and only if, the following conditions hold:
\begin{enumerate}
  \item $x\in\{1,2,3,4\}$, and
  \item if $x=1$, then $3b=4$ and $Orb_b(3)\not\subseteq\{3,4\}$, and
  \item if $x=2$, then $y\in\{1,3\}$ and
  \begin{enumerate}
    \item if $y=1$, then $3b=4$ and $4b\neq 3$, and
    \item if $y=3$, then $1b\in\{3,4\}$, and
  \end{enumerate}
  \item if $x=3$, then $1b=2$ and $Orb_b(1)\not\subseteq\{1,2\}$, and
  \item if $x=4$, then $y\in\{1,3\}$ and
  \begin{enumerate}
    \item if $y=1$, then $3b\in\{1,2\}$, and
    \item if $y=3$, then $1b=2$ and $2b\neq 1$.
  \end{enumerate}
\end{enumerate}
Moreover, if $\mathcal{A}$ is proper and $3$-compressible, then either the word $ab^2a$ or $ab^3a$ $3$-compresses $\A$.
\end{proposition}

\begin{proof}
Let $\A$ be a $(\bf{2},\bf{3})$-automaton that does not satisfy one of the conditions $1.-5.$
\begin{enumerate}
  \item Let $x\not\in\{1,2,3,4\}$, then the word $ba$ $3$-compresses $\A$, and so it is not proper.
  \item Let $x=1$,
  \begin{enumerate}
    \item if $3b\neq 4$, then if $3b=3$ $\A$ is not $3$-compressible, while if $3b\neq 3$ the word $aba$ $3$-compresses $\A$ that it is not proper, as shown in Fig. \ref{fig:prop:2-3:2:1};
    \item if $3b=4$ and $Orb_b(3)\subseteq\{3,4\}$, then $\A$ is not $3$-compressible, as shown in Fig. \ref{fig:prop:2-3:2:2}.
        \begin{figure}[ht]%
            \centering
            \subfigure[][Case $3b\neq 4$.]{%
            \label{fig:prop:2-3:2:1}%
            {\footnotesize
            \begin{tikzpicture}[shorten >=1pt,node distance=5em,auto]

                \node[state] (q_0) {};
                \node[state] (q_1) [right of= q_0] {$1,3$};
                \node[state] (q_2) [right of= q_1,draw=none] {};
                \node[state] (q_3) [right of= q_2] {$1,3$};
                \node[state] (q_4) [below of= q_0] {$1$};
                \node[state] (q_5) [below of= q_2] {$1,3b$};
                \node[state,accepting] (q_6) [right of= q_5] {};
                \node[state,draw=none] (q_7) [right of= q_6] {};

                \path[->] (q_0) edge node {$a$} (q_1)
                          (q_0) edge node {$b$} (q_4)
                          (q_1) edge [loop above] node {$a$} (q_1)
                          (q_1) edge node {$b|3b=3$} (q_3)
                          (q_1) edge node {$b|3b\neq 3$} (q_5)
                          (q_3) edge [loop above] node {$a,b$} (q_3)
                          (q_4) edge [loop left] node {$b$} (q_4)
                          (q_4) edge node {$a$} (q_1)
                          (q_5) edge node {$a$} (q_6);

                \draw [->] (-.5,.5) -- (-.27,.27);

            \end{tikzpicture}}%
            }
            \hspace{20pt}%
            \subfigure[][Case $Orb_b(3)\subseteq\{3,4\}$ and $3b=4$.]{%
            \label{fig:prop:2-3:2:2}%
            {\footnotesize
            \begin{tikzpicture}[shorten >=1pt,node distance=5em,auto]

                \node[state] (q_0) {};
                \node[state] (q_1) [right of= q_0] {$1,3$};
                \node[state] (q_2) [right of= q_1] {$1,4$};
                \node[state] (q_3) [right of= q_2,draw=none] {};
                \node[state] (q_4) [below of= q_0] {$1$};

                \path[->] (q_0) edge node {$a$} (q_1)
                          (q_0) edge node {$b$} (q_4)
                          (q_1) edge [loop above] node {$a$} (q_1)
                          (q_1) edge [bend left=10] node {$b$} (q_2)
                          (q_2) edge [bend left=10] node {$a,b$} (q_1)
                          (q_4) edge [loop left] node {$b$} (q_4)
                          (q_4) edge node {$a$} (q_1);

                \draw [->] (-.5,.5) -- (-.27,.27);

            \end{tikzpicture}}}%
            \caption[]{3MSA for automata that do not satisfy condition $2.$ of Proposition \ref{proposition:2-3}.}%
            \label{fig:prop:2-3:2}%
        \end{figure}
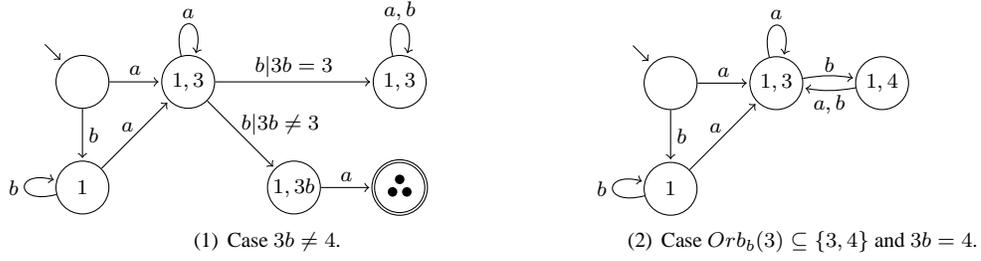
  \end{enumerate}
  \item Let $x=2$,
      \begin{enumerate}
        \item if $y\not\in\{1,3\}$, then the word $ab$ $3$-compresses $\A$, and so it is not proper;
        \item if $y=1$, then if $3b\neq 4$, then $\A$ is not $3$-compressible, as shown in Fig. \ref{fig:prop:2-3:3:1}; while if $3b=4$ and $4b=3$, then $\A$ is not $3$-compressible, as shown in Fig. \ref{fig:prop:2-3:3:2};
            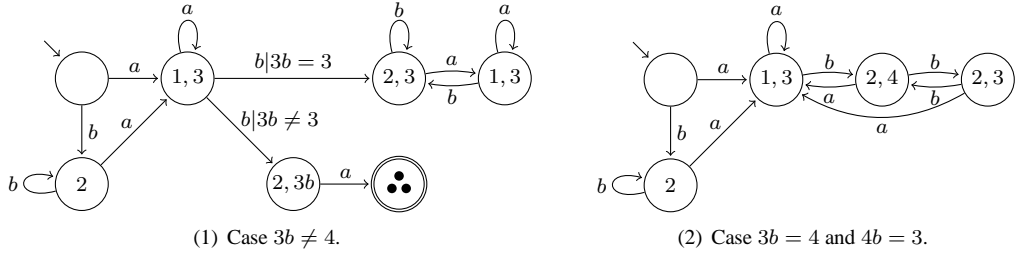
\begin{figure}[ht]%
                \centering
                \subfigure[][Case $3b\neq 4$.]{%
                \label{fig:prop:2-3:3:1}%
                {\footnotesize
                \begin{tikzpicture}[shorten >=1pt,node distance=5em,auto]

                    \node[state] (q_0) {};
                    \node[state] (q_1) [right of= q_0] {$1,3$};
                    \node[state] (q_2) [right of= q_1,draw=none] {};
                    \node[state] (q_3) [right of= q_2] {$2,3$};
                    \node[state] (q_4) [below of= q_0] {$2$};
                    \node[state] (q_5) [below of= q_2] {$2,3b$};
                    \node[state,accepting] (q_6) [right of= q_5] {};
                    \node[state] (q_7) [right of= q_3] {$1,3$};

                    \path[->] (q_0) edge node {$a$} (q_1)
                              (q_0) edge node {$b$} (q_4)
                              (q_1) edge [loop above] node {$a$} (q_1)
                              (q_1) edge node {$b|3b=3$} (q_3)
                              (q_1) edge node {$b|3b\neq 3$} (q_5)
                              (q_3) edge [loop above] node {$b$} (q_3)
                              (q_3) edge [bend left=10] node {$a$} (q_7)
                              (q_7) edge [bend left=10] node {$b$} (q_3)
                              (q_7) edge [loop above] node {$a$} (q_7)
                              (q_4) edge [loop left] node {$b$} (q_4)
                              (q_4) edge node {$a$} (q_1)
                              (q_5) edge node {$a$} (q_6);

                    \draw [->] (-.5,.5) -- (-.27,.27);

                \end{tikzpicture}}%
                }
                \hspace{20pt}%
                \subfigure[][Case $3b=4$ and $4b=3$.]{%
                \label{fig:prop:2-3:3:2}%
                {\footnotesize
                \begin{tikzpicture}[shorten >=1pt,node distance=5em,auto]

                    \node[state] (q_0) {};
                    \node[state] (q_1) [right of= q_0] {$1,3$};
                    \node[state] (q_2) [right of= q_1] {$2,4$};
                    \node[state] (q_3) [right of= q_2] {$2,3$};
                    \node[state] (q_4) [below of= q_0] {$2$};

                    \path[->] (q_0) edge node {$a$} (q_1)
                              (q_0) edge node {$b$} (q_4)
                              (q_1) edge [loop above] node {$a$} (q_1)
                              (q_1) edge [bend left=10] node {$b$} (q_2)
                              (q_2) edge [bend left=10] node {$b$} (q_3)
                              (q_2) edge [bend left=10] node {$a$} (q_1)
                              (q_3) edge [bend left=30] node {$a$} (q_1)
                              (q_3) edge [bend left=10] node {$b$} (q_2)
                              (q_4) edge [loop left] node {$b$} (q_4)
                              (q_4) edge node {$a$} (q_1);

                    \draw [->] (-.5,.5) -- (-.27,.27);

                \end{tikzpicture}}}%
                \caption[]{3MSA for automata that do not satisfy condition $3.$ of Proposition \ref{proposition:2-3} with $y=1$.}%
                \label{fig:prop:2-3:3}%
            \end{figure}
        \item if $y=3$ and $1b\not\in\{3,4\}$, then $\M(ab)=\{2,1b\}$, and $\M(aba)=\{1,3,1ba\}$, so the word $aba$ $3$-compresses $\A$ that it is not proper.
      \end{enumerate}
\end{enumerate}
The cases $x=3$ and $x=4$ reduce to case $2.$ and $3.$, respectively.

Conversely, suppose $x\in\{1,2,3,4\}$.
\begin{enumerate}
  \item Let $x=1$, $3b=4$ and $Orb_b(3)\not\subseteq\{3,4\}$, this implies $4b=3b^2\not\in\{3,4\}$. Then $\M(ab)=\{1,4\}$, $\M(ab^2)=\{1,4b\}$ and $\M(ab^2a)=\{1,3,4ba\}$, so $ab^2a$ $3$-compresses $\A$.
  \item Let $x=2$.
    \begin{enumerate}
      \item If $y=1$, $3b=4$ and $4b\neq 3$, then $\M(ab)=\{2,4\}$, $\M(ab^2)=\{2,4b\}$ and $\M(ab^2a)=\{1,3,4ba\}$, so $ab^2a$ $3$-compresses $\A$.
      \item If $y=3$ and $1b\in\{3,4\}$, then either the word $ab^2a$ or $ab^3a$ $3$-compresses $\A$, as shown in Fig. \ref{fig:prop:2-3:r}.
          \begin{figure}[ht]%
            \centering
            {\footnotesize
            \begin{tikzpicture}[shorten >=1pt,node distance=5em,auto]

                \node[state] (q_0) {};
                \node[state] (q_1) [right of= q_0] {$1,3$};
                \node[state] (q_2) [right of= q_1] {$2,1b$};
                \node[state] (q_3) [right of= q_2] {$2,1b^2$};
                \node[state,draw=none] (q_4) [right of= q_3] {};
                \node[state] (q_5) [right of= q_4] {$2,1b^3$};
                \node[state,accepting] (q_6) [right of= q_5] {};

                \path[->] (q_0) edge node {$a$} (q_1)
                          (q_1) edge node {$b$} (q_2)
                          (q_2) edge node {$b$} (q_3)
                          (q_3) edge node {$b|1b^2\in\{3,4\}$} (q_5)
                          (q_3) edge [bend right=25] node [swap] {$b|1b^2\not\in\{3,4\}$} (q_6)
                          (q_5) edge node {$a$} (q_6);

                \draw [->] (-.5,.5) -- (-.27,.27);

            \end{tikzpicture}}\\
            \caption[]{P3MSA for the case $b=[2,3]\backslash 2$ and $1b\in\{3,4\}$.}%
            \label{fig:prop:2-3:r}%
            \end{figure}
    \end{enumerate}
\end{enumerate}
The cases $x=3$ and $x=4$ reduce to cases $1.$ and $2.$, respectively.
\end{proof}

\FloatBarrier


The following lemma is straightforward.
\begin{lemma}\label{lemma:3-3:0}
Let $\A$ be a $(\bf{3},\bf{3})$ $3$-compressible automaton with $a=[1,2]\backslash 1$ and $b=[x,y]\backslash x$.
Then $x\neq 1$ and $\{x,y\}\neq \{1,2\}$.
\end{lemma}

\begin{lemma}\label{lemma:3-3:1}
Let $\A$ be a $(\bf{3},\bf{3})$ $3$-compressible automaton with $a=[1,2]\backslash 1$ and $b=[2,y]\backslash 2$ with $y\neq 1$.
Then:
\begin{enumerate}
  \item $\{1ba,1b^2a,1b^3a\}\not\subseteq\{2,y\}$, or
  \item $Orb_a(1ba)\not\subseteq\{2,y\}$.
\end{enumerate}
\end{lemma}

\begin{proof}
Let $\A$ be a $(\bf{3},\bf{3})$-automaton that does not satisfy the above conditions, we prove that it is not $3$-compressible or not proper.
\begin{enumerate}
  \item Let $1b=1$ and $Orb_a(1ba)\subseteq\{2,y\}$.
      \begin{enumerate}
        \item If $1ba=1a=2a=2$, then the 3MSA in Fig. \ref{fig:lemma:3-3:1:1:1} proves that $\A$ is not 3-compressible;
        \item If $1ba=1a=2a=y$, then $ya=1ba^2=2$ and the 3MSA in Fig. \ref{fig:lemma:3-3:1:1:2} proves that $\A$ is not 3-compressible.
            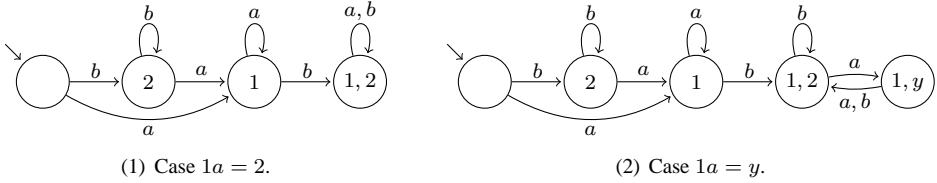
\begin{figure}[ht]%
                \centering
                \subfigure[][Case $1a=2$.]{%
                \label{fig:lemma:3-3:1:1:1}%
                {\footnotesize
                \begin{tikzpicture}[shorten >=1pt,node distance=5em,auto]

                    \node[state] (q_0) {};
                    \node[state] (q_1) [right of= q_0] {$2$};
                    \node[state] (q_2) [right of= q_1] {$1$};
                    \node[state] (q_3) [right of= q_2] {$1,2$};

                    \path[->] (q_0) edge node {$b$} (q_1)
                              (q_0) edge [bend right=30] node [swap] {$a$} (q_2)
                              (q_1) edge node {$a$} (q_2)
                              (q_1) edge [loop above] node {$b$} (q_1)
                              (q_2) edge node {$b$} (q_3)
                              (q_2) edge [loop above] node {$a$} (q_2)
                              (q_3) edge [loop above] node {$a,b$} (q_3);

                    \draw [->] (-.5,.5) -- (-.27,.27);

                \end{tikzpicture}}%
                }
                \hspace{20pt}%
                \subfigure[][Case $1a=y$.]{%
                \label{fig:lemma:3-3:1:1:2}%
                {\footnotesize
                \begin{tikzpicture}[shorten >=1pt,node distance=5em,auto]

                    \node[state] (q_0) {};
                    \node[state] (q_1) [right of= q_0] {$2$};
                    \node[state] (q_2) [right of= q_1] {$1$};
                    \node[state] (q_3) [right of= q_2] {$1,2$};
                    \node[state] (q_4) [right of= q_3] {$1,y$};

                    \path[->] (q_0) edge node {$b$} (q_1)
                              (q_0) edge [bend right=30] node [swap] {$a$} (q_2)
                              (q_1) edge node {$a$} (q_2)
                              (q_1) edge [loop above] node {$b$} (q_1)
                              (q_2) edge node {$b$} (q_3)
                              (q_2) edge [loop above] node {$a$} (q_2)
                              (q_3) edge [loop above] node {$b$} (q_3)
                              (q_3) edge [bend left=10] node {$a$} (q_4)
                              (q_4) edge [bend left=10] node {$a,b$} (q_3);

                    \draw [->] (-.5,.5) -- (-.27,.27);

                \end{tikzpicture}}}
                \caption[]{3MSA for automata that do not satisfy conditions 1. and 2. of Lemma \ref{lemma:3-3:1} with $1b=1$.}%
                \label{fig:lemma:3-3:1:1}%
            \end{figure}
      \end{enumerate}
  \item Let $1b\neq1$, $\{1ba,1b^2a,1b^3a\}\subseteq\{2,y\}$ and $Orb_a(1ba)\subseteq\{2,y\}$.
      If $1ba=1b^2a$ or $1b^2a=1b^3a$, then for all $h>0$ $1b^h\neq 2$, then $1=1b$, against the hypothesis.
      If $1b^2a=1b^3a$, as $1b^2\neq 2$, then $1b^2=1b^3$, and then $1=1b$, against the hypothesis.
      So, as $|\{1ba,1b^2a,1b^3a\}|\leq 2$, $1ba=1b^3a$, then $1b=1b^3$ and $1=1b^2$.
      \begin{enumerate}
        \item If $1ba=2$, then $1b^2a=y$, hence $1a=2a=y$, $ya=2$ and $1b=y$.
            The 3MSA in Fig. \ref{fig:lemma:3-3:1:2:1} proves that $\A$ is not 3-compressible;
        \item If $1ba=y$, then $1b^2a=1a=2a=2$.
            Moreover, $1ba^2=y$, otherwise $1ba^2=2$ gives the contradiction $1ba=1$, hence $ya=y$, $1b=y$ and $yb=1b^2=1$.
            The 3MSA in Fig. \ref{fig:lemma:3-3:1:2:2} proves that $\A$ is not 3-compressible.
            \begin{figure}[ht]%
                \centering
                \subfigure[][Case $1ba=2$.]{%
                \label{fig:lemma:3-3:1:2:1}%
                {\footnotesize
                \begin{tikzpicture}[shorten >=1pt,node distance=5em,auto]

                    \node[state] (q_0) {};
                    \node[state] (q_1) [below of= q_0] {$2$};
                    \node[state] (q_2) [right of= q_0] {$1$};
                    \node[state] (q_3) [right of= q_2] {$y,2$};
                    \node[state] (q_4) [right of= q_3] {$1,2$};
                    \node[state] (q_5) [below of= q_3] {$1,y$};

                    \path[->] (q_0) edge node {$b$} (q_1)
                              (q_0) edge node {$a$} (q_2)
                              (q_1) edge node [swap] {$a$} (q_2)
                              (q_1) edge [loop right] node {$b$} (q_1)
                              (q_2) edge node {$b$} (q_3)
                              (q_2) edge [loop above] node {$a$} (q_2)
                              (q_3) edge [bend left=10] node {$a,b$} (q_4)
                              (q_4) edge [bend left=10] node {$a$} (q_5)
                              (q_4) edge [bend left=10] node {$b$} (q_3)
                              (q_5) edge [bend left=10] node {$a$} (q_4)
                              (q_5) edge node {$b$} (q_3);

                    \draw [->] (-.5,.5) -- (-.27,.27);

                \end{tikzpicture}}%
                }
                \hspace{20pt}%
                \subfigure[][Case $1ba=y$.]{%
                \label{fig:lemma:3-3:1:2:2}%
                {\footnotesize
                \begin{tikzpicture}[shorten >=1pt,node distance=5em,auto]

                    \node[state] (q_0) {};
                    \node[state] (q_1) [below of= q_0] {$2$};
                    \node[state] (q_2) [right of= q_0] {$1$};
                    \node[state] (q_4) [right of= q_2] {$y,2$};
                    \node[state] (q_3) [below of= q_4] {$1,2$};
                    \node[state] (q_5) [right of= q_4] {$1,y$};

                    \path[->] (q_0) edge node {$b$} (q_1)
                              (q_0) edge node {$a$} (q_2)
                              (q_1) edge node [swap] {$a$} (q_2)
                              (q_1) edge [loop right] node {$b$} (q_1)
                              (q_2) edge node {$b$} (q_4)
                              (q_2) edge [loop above] node {$a$} (q_2)
                              (q_3) edge [loop right] node {$a$} (q_3)
                              (q_3) edge [bend left=10] node {$b$} (q_4)
                              (q_4) edge [bend left=10] node {$a$} (q_5)
                              (q_4) edge [bend left=10] node {$b$} (q_3)
                              (q_5) edge [loop above] node {$a$} (q_5)
                              (q_5) edge [bend left=10] node {$b$} (q_4);

                    \draw [->] (-.5,.5) -- (-.27,.27);

                \end{tikzpicture}}}
                \caption[]{3MSA for automata that do not satisfy conditions 1. and 2. of Lemma \ref{lemma:3-3:1} with $1b\neq1$.}%
                \label{fig:lemma:3-3:1:2}%
            \end{figure}
      \end{enumerate}
\end{enumerate}
\end{proof}

\begin{corollary}\label{corollary:3-3:1}
Let $\A$ be a $(\bf{3},\bf{3})$ $3$-compressible automaton with $a=[1,2]\backslash 1$ and $b=[x,1]\backslash x$ with $x\neq 2$.
Then:
\begin{enumerate}
  \item $\{xab,xa^2b,xa^3b\}\not\subseteq\{1,2\}$, or
  \item $Orb_b(xab)\not\subseteq\{1,2\}$.
\end{enumerate}
\end{corollary}

\begin{proof}
It is a straightforward consequence of the previous lemma, simply replacing $a$ with $b$, 1 with 2, 2 with $y$ and $x$ with 1.
\end{proof}

\begin{lemma}\label{lemma:3-3:2}
Let $\A$ be a proper $(\bf{3},\bf{3})$ $3$-compressible automaton with $a=[1,2]\backslash 1$ and $b=[x,y]\backslash x$ with $x\not\in\{1,2\}$ and $y\neq 1$.
Then all the following conditions hold:
\begin{enumerate}
  \item $1b\in\{1,2\}$, and
  \item $xa\in\{x,y\}$, and
  \item $Orb_b(1)\not\subseteq\{1,2\}$ or $Orb_a(x)\not\subseteq\{x,y\}$.
\end{enumerate}
\end{lemma}

\begin{proof}
Let $\A$ be a $(\bf{3},\bf{3})$-automaton that does not satisfy the above conditions, we prove that it is not $3$-compressible or not proper.
\begin{enumerate}
  \item If $1b\not\in\{1,2\}$, then $\A$ is not proper, in fact $\M(ab)=\{x,1b\}$ and $|\M(aba)|=3$; similarly
  \item if $1b\in\{1,2\}$ and $xa\not\in\{x,y\}$, then $\A$ is not proper, in fact $\M(ba)=\{1,xa\}$ and $|\M(bab)|=3$.
  \item If $1b\in\{1,2\}$, $xa\in\{x,y\}$, $Orb_b(1)\subseteq\{1,2\}$ and $Orb_a(x)\subseteq\{x,y\}$, then
      \begin{enumerate}
        \item if $1b=1$ and $xa=x$, the 3MSA of $\A$ is in Fig. \ref{fig:lemma:3-3:2:2:1};
        \item $1b=1$ and $xa=y$ (and then $ya=x$), the 3MSA of $\A$ is in Fig. \ref{fig:lemma:3-3:2:2:2};
        \item $1b=2$ (and then $2b=1$) and $xa=x$, the 3MSA of $\A$ is in Fig. \ref{fig:lemma:3-3:2:2:3};
        \item $1b=2$ (and then $2b=1$) and $xa=y$ (and then $ya=x$) the 3MSA of $\A$ is in Fig. \ref{fig:lemma:3-3:2:2:4}.
      \end{enumerate}
      \begin{figure}[ht]%
        \centering
        \subfigure[][$1b=1$ and $xa=x$.]{%
        \label{fig:lemma:3-3:2:2:1}%
        {\footnotesize
        \begin{tikzpicture}[shorten >=1pt,node distance=5em,auto]

            \node[state] (q_0) {};
            \node[state] (q_1) [right of= q_0] {$1$};
            \node[state] (q_2) [below of= q_0] {$x$};
            \node[state] (q_3) [right of= q_2] {$1,x$};
            \node[state] (q_4) [below of= q_3,draw=none] {};

            \path[->] (q_0) edge node {$a$} (q_1)
                      (q_0) edge node {$b$} (q_2)
                      (q_1) edge [loop above] node {$a$} (q_1)
                      (q_1) edge node {$b$} (q_3)
                      (q_2) edge [loop below] node {$b$} (q_2)
                      (q_2) edge node {$a$} (q_3)
                      (q_3) edge [loop below] node {$a,b$} (q_3)(q_1) edge [loop right,draw=none] node {} (q_1);

            \draw [->] (-.5,.5) -- (-.27,.27);

        \end{tikzpicture}}%
        }
        \hspace{10pt}%
        \subfigure[][$1b=1$, $xa=y$ and $ya=x$.]{%
        \label{fig:lemma:3-3:2:2:2}%
        {\footnotesize
        \begin{tikzpicture}[shorten >=1pt,node distance=5em,auto]

            \node[state] (q_0) {};
            \node[state] (q_1) [right of= q_0] {$1$};
            \node[state] (q_2) [below of= q_0] {$x$};
            \node[state] (q_3) [right of= q_2] {$1,x$};
            \node[state] (q_4) [below of=q_3] {$1,y$};

            \path[->] (q_0) edge node {$a$} (q_1)
                      (q_0) edge node {$b$} (q_2)
                      (q_1) edge [loop above] node {$a$} (q_1)
                      (q_1) edge node {$b$} (q_3)
                      (q_2) edge [loop below] node {$b$} (q_2)
                      (q_2) edge node [swap] {$a$} (q_4)
                      (q_3) edge [loop left] node {$b$} (q_3)
                      (q_3) edge [bend left=10] node {$a$} (q_4)
                      (q_4) edge [bend left=10] node {$a,b$} (q_3)
                      (q_1) edge [loop right,draw=none] node {} (q_1);

            \draw [->] (-.5,.5) -- (-.27,.27);

        \end{tikzpicture}}}
        \hspace{10pt}%
        \subfigure[][$1b=2$, $2b=1$ and $xa=x$.]{%
        \label{fig:lemma:3-3:2:2:3}%
        {\footnotesize
        \begin{tikzpicture}[shorten >=1pt,node distance=5em,auto]

            \node[state] (q_0) {};
            \node[state] (q_1) [right of= q_0] {$1$};
            \node[state] (q_2) [below of= q_0] {$x$};
            \node[state] (q_3) [right of= q_2] {$2,x$};
            \node[state] (q_4) [below of=q_3] {$1,x$};

            \path[->] (q_0) edge node {$a$} (q_1)
                      (q_0) edge node {$b$} (q_2)
                      (q_1) edge [loop above] node {$a$} (q_1)
                      (q_1) edge node {$b$} (q_3)
                      (q_2) edge [loop below] node {$b$} (q_2)
                      (q_2) edge node [swap] {$a$} (q_4)
                      (q_3) edge [bend right=10] node [swap] {$a,b$} (q_4)
                      (q_4) edge [bend right=10] node [swap] {$b$} (q_3)
                      (q_4) edge [loop left] node {$a$} (q_4)(q_1) edge [loop right,draw=none] node {} (q_1);

            \draw [->] (-.5,.5) -- (-.27,.27);

        \end{tikzpicture}}%
        }
        \hspace{10pt}%
        \subfigure[][$1b=2$, $2b=1$, $xa=y$ and $ya=x$.]{%
        \label{fig:lemma:3-3:2:2:4}%
        {\footnotesize
        \begin{tikzpicture}[shorten >=1pt,node distance=5em,auto]

            \node[state] (q_0) {};
            \node[state] (q_1) [right of= q_0] {$1$};
            \node[state] (q_2) [below of= q_0] {$x$};
            \node[state] (q_3) [right of= q_2] {$2,x$};
            \node[state] (q_4) [below of=q_2] {$1,y$};
            \node[state] (q_5) [right of=q_4] {$1,x$};

            \path[->] (q_0) edge node {$a$} (q_1)
                      (q_0) edge node [swap] {$b$} (q_2)
                      (q_1) edge [loop above] node {$a$} (q_1)
                      (q_1) edge node {$b$} (q_3)
                      (q_2) edge [loop right] node {$b$} (q_2)
                      (q_2) edge node [swap] {$a$} (q_4)
                      (q_3) edge [bend left=10] node {$a$} (q_4)
                      (q_3) edge [bend left=10] node {$b$} (q_5)
                      (q_4) edge [bend left=10] node {$b$} (q_3)
                      (q_4) edge [bend left=10] node {$a$} (q_5)
                      (q_5) edge [bend left=10] node {$a$} (q_4)
                      (q_5) edge [bend left=10] node {$b$} (q_3)(q_1) edge [loop right,draw=none] node {} (q_1);

            \draw [->] (-.5,.5) -- (-.27,.27);

        \end{tikzpicture}}}\\
        \caption[]{3MSA for automata that do not satisfy conditions 3. of Lemma \ref{lemma:3-3:2}}%
        \label{fig:lemma:3-3:2:2}%
      \end{figure}
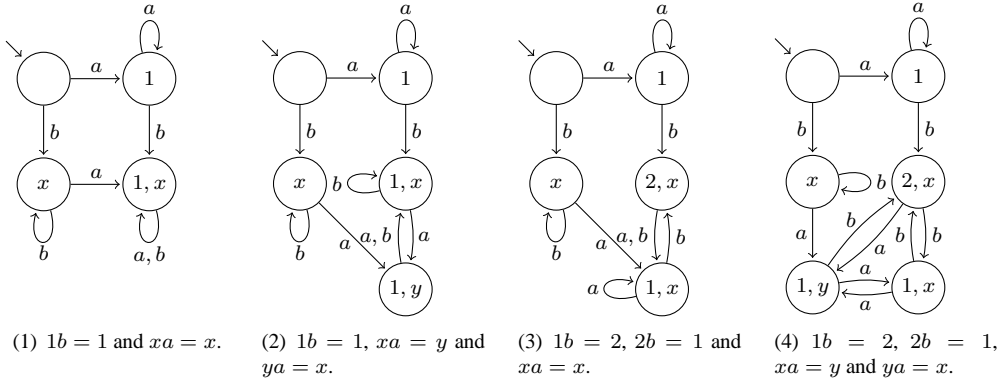
      In all the subcases, $\A$ is not $3$-compressible.
\end{enumerate}
\end{proof}

\begin{proposition}\label{proposition:3-3}
Let $\A$ be a $(\bf{3},\bf{3})$-automaton with $a=[1,2]\backslash 1$ and $b=[x,y]\backslash x$.
Then $\A$ is proper and $3$-compressible if, and only if, one of the following conditions holds:
\begin{enumerate}
  \item $b=[2,y]\backslash 2$, $y\neq 1$, and either $\{1ba,1b^2a,1b^3a\}\not\subseteq\{2,y\}$ or $Orb_a(1ba)\not\subseteq\{2,y\}$;
  \item $b=[x,1]\backslash x$, $x\neq 2$, and either $\{xba,xb^2a,xb^3a\}\not\subseteq\{1,2\}$ or $Orb_b(xab)\not\subseteq\{1,2\}$;
  \item $b=[x,y]\backslash x$, $x\not\in\{1,2\}$, $y\neq 1$, $1b\in\{1,2\}$, $xa\in\{x,y\}$, and either $Orb_b(1)\not\subseteq\{1,2\}$ or $Orb_a(x)\not\subseteq\{x,y\}$.
\end{enumerate}
Moreover, if $\mathcal{A}$ is proper and $3$-compressible, then one of the words $abab$ or $ab^2ab$ or $ab^3ab$ or $aba^2b$ or $aba^3b$ or $baba$ or $ba^2ba$ or $ba^3ba$ or $bab^2a$ or $bab^3a$ $3$-compresses $\A$.
\end{proposition}

\begin{proof}
If $\A$ is a $3$-compressible $(\bf{3},\bf{3})$-automaton with $a=[1,2]\backslash 1$ and $b=[x,y]\backslash x$, then, from the previous lemmata and corollary, one of the above conditions must hold.

Conversely, we find for any automaton satisfying conditions $1.-3.$ a (short) $3$-compressing word.
\begin{enumerate}
  \item Let $b=[2,y]\backslash 2$, $y\neq 1$,
      \begin{enumerate}
        \item if $\{1ba,1b^2a,1b^3a\}\not\subseteq\{2,y\}$, then either the word $abab$ or $ab^2ab$ or $ab^3ab$ $3$-compresses $\A$, since $\M(aba)=\{1,1ba\}$, $\M(ab^2a)=\{1,1b^2a\}$ and $\M(ab^3a)=\{1,1b^3a\}$;
        \item if $Orb_a(1ba)\not\subseteq\{2,y\}$, then $\{1ba,1ba^2,1ba^3\}\not\subseteq\{2,y\}$ and either the word $abab$ or $aba^2b$ or $aba^3b$ $3$-compresses $\A$, since $\M(aba)=\{1,1ba\}$, $\M(aba^2)=\{1,1ba^2\}$ and $\M(ab^3a)=\{1,1ba^3\}$.
      \end{enumerate}
  \item Let $b=[x,1]\backslash x$, $x\neq 2$, and either $\{xba,xb^2a,xb^3a\}\not\subseteq\{1,2\}$ or $Orb_b(xab)\not\subseteq\{1,2\}$.
      This case reduces to the previous one replacing $a$ with $b$, 1 with 2, 2 with $y$ and $x$ with 1, then either the word $baba$ or $ba^2ba$ or $ba^3ba$ or $bab^2a$ or $bab^3a$ $3$-compresses $\A$.
  \item Let $b=[x,y]\backslash x$, $x\not\in\{1,2\}$, $y\neq 1$, $1b\in\{1,2\}$, $xa\in\{x,y\}$.
      If $Orb_b(1)\not\subseteq\{1,2\}$, then $1b=2$ and $2b\not\in\{1,2\}$, and so the word $ab^2a$ $3$-compresses $\A$, as $\M(ab^2)=\{x,2b\}$; similarly, if $Orb_a(x)\not\subseteq\{x,y\}$,  exchanging $b$ and $a$ and so $x$ with 1 and $y$ with 2, the word $ba^2b$ $3$-compresses $\A$.
\end{enumerate}
\end{proof}

\FloatBarrier


\begin{lemma}\label{lemma:3-4}
Let $\A$ be a $(\bf{3},\bf{4})$ $3$-compressible automaton with $a=[1,2]\backslash 1$ and $b=[x,y]\backslash z,\ zb=x$.
If $\A$ is proper then all the following conditions hold:
\begin{enumerate}
  \item $\{1,2\}\cap\{x,z\}\neq \emptyset$,
  \item $z\in\{1,2\}$ or $\{1,za\}\cap\{x,y\}\neq \emptyset$,
  \item $\{1,1b\}\cap\{x,y\}\neq \emptyset$,
  \item $1\in \{x,y\}$ or $\{z,1b\}\cap\{1,2\}\neq \emptyset$.
\end{enumerate}
\end{lemma}

\begin{proof}
If $\{1,2\}\cap\{x,z\}=\emptyset$, then $|\M(b^2a)|=3$; if $z\not\in\{1,2\}$ and $\{1,za\}\cap\{x,y\}=\emptyset$, then $|\M(bab)|=3$; if $\{1,1b\}\cap\{x,y\}=\emptyset$, then $|\M(ab^2)|=3$; if $1\not\in \{x,y\}$ and $\{z,1b\}\cap\{1,2\}=\emptyset$, then $|\M(aba)|=3$.
So each automaton that does not satisfy one of the conditions of the lemma is $3$-compressed by a word of length $3$, and then it is not proper.
\end{proof}

\begin{corollary}\label{corollary:3-4}
Let $\A$ be a \emph{proper} $(\bf{3},\bf{4})$ $3$-compressible automaton with $a=[1,2]\backslash 1$ and $b=[x,y]\backslash z$, $zb=x$.
The following conditions hold:
\begin{enumerate}
  \item $b=[x,y]\backslash 1$, or
  \item $b=[x,y]\backslash 2$ and $\{1,1b\}\cap\{x,y\}\neq\emptyset$, or
  \item $b=[1,y]\backslash z$ or $b=[2,1]\backslash z$ and in both cases $z\neq 2$.
\end{enumerate}
\end{corollary}

\begin{proof}
It is a straightforward consequence of the previous lemma.
\end{proof}

\begin{proposition}\label{proposition:3-4}
Let $\A$ be a $(\bf{3},\bf{4})$-automaton with $a=[1,2]\backslash 1$ and $b=[x,y]\backslash z,\ zb=x$. Then $\A$ is proper and $3$-compressible if, and only if, the following conditions hold:
\begin{enumerate}
  \item if $b=[x,y]\backslash 1$, then $Orb_a(x)\not\subseteq\{x,y\}$,
  \item if $z=2$ and
    \begin{enumerate}
      \item $b=[1,y]\backslash 2$, then $Orb_a(2)\not\subseteq\{2,y\}$ or $2ab\not\in\{1,y\}$,
      \item $b=[x,1]\backslash 2$, then $\{2a,xa\}\neq \{2,x\}$ or $\{2ab,xab\}\neq \{1,x\}$
      \item $b=[x,y]\backslash 2$, $1b=y$, and $1\not\in\{x,y\}$, then $\{xa,ya\}\neq \{x,y\}$,
    \end{enumerate}
  \item if $z\not\in\{1,2\}$ and
    \begin{enumerate}
      \item $b=[1,y]\backslash z$ and $y\neq 2$, then $za\neq z$,
      \item $b=[x,y]\backslash z$ and $\{x,y\}=\{1,2\}$, then $|Orb_a(z)|>2$ or $zab\not\in\{1,2\}$.
   \end{enumerate}
\end{enumerate}
Moreover, if $\mathcal{A}$ is proper and $3$-compressible, then either the word $b^2ab^2$ or $b^2a^2b^2$ or $b^2a^3b^2$ or $b^2abab^2$ $3$-compresses $\A$.
\end{proposition}

\begin{proof}
First observe that from the previous corollary a proper $3$-compressible $(\bf{3},\bf{4})$-automaton always satisfies the antecedent of one of the above conditions, so all the possible cases are taken into account.

\noindent We start proving that a $(\bf{3},\bf{4})$-automaton that does not satisfy the conditions $1.-3.$ is not $3$-compressible.
\begin{enumerate}
  \item Let condition $1.$ be false, \textit{i.e.}, $b=[x,y]\backslash 1$ and $Orb_a(x)\subseteq\{x,y\}$.
      The 3MSA in Fig. \ref{fig:prop:3-4:d:1} proves that $\A$ is not $3$-compressible.
      \begin{figure}[ht]%
        \centering
        {\footnotesize
        \begin{tikzpicture}[shorten >=1pt,node distance=5em,auto]

            \node[state] (q_0) {};
            \node[state] (q_1) [right of= q_0] {$1$};
            \node[state] (q_2) [right of= q_1] {$1,x$};
            \node[state] (q) [right of= q_2,draw=none] {};
            \node[state] (q_3) [right of= q] {$1,x$};
            \node[state] (q_4) [right of= q_3] {$1,y$};
            \node[state] (q_5) [right of= q_4] {$1,x$};

            \path[->] (q_0) edge node {$a,b$} (q_1)
                      (q_1) edge [loop above] node {$a$} (q_1)
                      (q_1) edge node {$b$} (q_2)
                      (q_2) edge node {$a|xa=x$} (q_3)
                      (q_2) edge [bend right=20] node [swap] {$a|xa=y$} (q_4)
                      (q_2) edge [loop above] node {$b$} (q_2)
                      (q_3) edge [loop above] node {$a,b$} (q_3)
                      (q_4) edge [bend right=20] node [swap] {$a,b$} (q_5)
                      (q_5) edge [bend right=20] node [swap] {$a$} (q_4)
                      (q_5) edge [loop right] node {$b$} (q_5);

            \draw [->] (-.5,.5) -- (-.27,.27);

        \end{tikzpicture}}\\
        \caption[]{3MSA for the case in which condition 1. of Proposition \ref{proposition:3-4} is false.}%
        \label{fig:prop:3-4:d:1}%
      \end{figure}
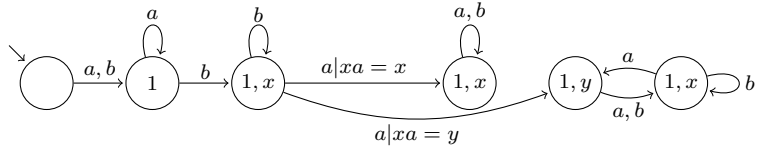
  \item Let $z=2$.
      \begin{enumerate}
        \item Let condition $2.(a)$ be false, \textit{i.e.}, $b=[1,y]\backslash 2$, $Orb_a(2)\subseteq\{2,y\}$ and $2ab\in\{1,y\}$.
            Observe that if $2a\neq 2$, then $2a=y$ and $ya=2$.
            Since $y\neq 2$, then $yb=2ab\neq 1$, and so  $2ab=y$.
            The 3MSA in Fig. \ref{fig:prop:3-4:d:2:1} proves that $\A$ is not $3$-compressible.
        \item Let condition $2.(b)$ be false, \textit{i.e.}, $b=[x,1]\backslash 2$, $\{2a,xa\}=\{2,x\}$ and $\{2ab,xab\}=\{1,x\}$.
            Observe that if $2a=2$, then $2ab=2b=x$ and so $xab=1$.
            If $xa=2$, then $xab=2b=x$ and so $2ab=1$.
            The 3MSA in Fig. \ref{fig:prop:3-4:d:2:2} proves that $\A$ is not $3$-compressible.
        \item Let condition $2.(c)$ be false, \textit{i.e.}, $b=[x,y]\backslash 2$, $1b=y$, $1\not\in\{x,y\}$ and $\{xa,ya\}= \{x,y\}$.
            The 3MSA in Fig. \ref{fig:prop:3-4:d:2:3} proves that $\A$ is not $3$-compressible.
            \begin{figure}[ht]%
                \centering
                \subfigure[][3MSA for the case $b={[1,y]}\backslash 2$, $2b=1$, $Orb_a(2)\subseteq\{2,y\}$ and $2ab\in\{1,y\}$.]{%
                \label{fig:prop:3-4:d:2:1}%
                {\footnotesize
                \begin{tikzpicture}[shorten >=1pt,node distance=5em,auto]

                    \node[state] (q_0) {};
                    \node[state] (q_2) [right of= q_0] {$2$};
                    \node[state] (q_1) [right of= q_2] {$1$};
                    \node[state] (q_3) [below of= q_0] {$1,2$};
                    \node[state] (q_9) [right of= q_3] {$2,y$};
                    \node[state] (q_7) [right of= q_9] {$1,y$};
                    \node[state] (q_8) [right of= q_7] {$1,2$};
                    \node[state] (q_4) [below of= q_3] {$1,2$};

                    \path[->] (q_0) edge [bend left=40] node {$a$} (q_1)
                              (q_0) edge node {$b$} (q_2)
                              (q_1) edge [loop above] node {$a$} (q_1)
                              (q_1) edge [bend left=20] node {$b$} (q_2)
                              (q_2) edge [bend left=20] node {$a$} (q_1)
                              (q_2) edge node [swap] {$b$} (q_3)
                              (q_3) edge [loop left] node {$b$} (q_3)
                              (q_3) edge node [swap] {$a|2a=2$} (q_4)
                              (q_3) edge [bend left=30] node {$a|2a=y$} (q_7)
                              (q_4) edge [loop below] node {$a,b$} (q_4)
                              (q_7) edge [bend left=10] node {$a$} (q_8)
                              (q_7) edge [bend left=10] node {$b$} (q_9)
                              (q_8) edge [bend left=10] node {$a$} (q_7)
                              (q_8) edge [loop above] node {$b$} (q_8)
                              (q_9) edge [bend right=30] node [swap] {$a,b$} (q_8);

                    \draw [->] (-.5,.5) -- (-.27,.27);

                \end{tikzpicture}}%
                }
                \hspace{20pt}%
                \subfigure[][3MSA for the case $b={[x,1]}\backslash 2$, $2b=x$ $\{2a,xa\}=\{2,x\}$ and $\{2ab,xab\}=\{1,x\}$.]{%
                \label{fig:prop:3-4:d:2:2}%
                {\footnotesize
                \begin{tikzpicture}[shorten >=1pt,node distance=5em,auto]

                    \node[state] (q_0) {};
                    \node[state] (q_2) [right of= q_0] {$2$};
                    \node[state] (q_1) [right of= q_2] {$1$};
                    \node[state] (q_3) [below of= q_0] {$2,x$};
                    \node[state] (q_9) [right of= q_3] {$2,x$};
                    \node[state] (q_7) [right of= q_9] {$1,2$};
                    \node[state] (q_8) [right of= q_7] {$1,x$};
                    \node[state] (q_4) [below of= q_3] {$1,x$};
                    \node[state] (q_5) [right of= q_4] {$1,2$};
                    \node[state] (q_6) [right of= q_5] {$2,x$};

                    \path[->] (q_0) edge [bend left=40] node {$a$} (q_1)
                              (q_0) edge node {$b$} (q_2)
                              (q_1) edge [loop above] node {$a$} (q_1)
                              (q_1) edge [bend left=20] node {$b$} (q_2)
                              (q_2) edge [bend left=20] node {$a$} (q_1)
                              (q_2) edge node [swap] {$b$} (q_3)
                              (q_3) edge [loop left] node {$b$} (q_3)
                              (q_3) edge node [swap] {$a|2a=2$} (q_4)
                              (q_3) edge [bend left=30] node {$a|2a=x$} (q_7)
                              (q_4) edge [loop below] node {$a$} (q_4)
                              (q_4) edge node {$b$} (q_5)
                              (q_5) edge [loop below] node {$a$} (q_5)
                              (q_5) edge node {$b$} (q_6)
                              (q_6) edge [loop below] node {$b$} (q_6)
                              (q_6) edge [bend right=30] node [swap] {$a$} (q_4)
                              (q_7) edge [bend left=10] node {$a$} (q_8)
                              (q_7) edge [bend left=10] node {$b$} (q_9)
                              (q_8) edge [bend left=10] node {$a,b$} (q_7)
                              (q_9) edge [bend left=10] node {$a$} (q_7)
                              (q_9) edge [loop left] node {$b$} (q_9);

                    \draw [->] (-.5,.5) -- (-.27,.27);

                \end{tikzpicture}}}
                \subfigure[][3MSA for the case $b={[x,y]}\backslash 2$, $2y=x$, $1b=y$, $1\not\in\{x,y\}$ and $\{xa,ya\}=\{x,y\}$.]{%
                \label{fig:prop:3-4:d:2:3}%
                {\footnotesize
                \begin{tikzpicture}[shorten >=1pt,node distance=5em,auto]

                    \node[state] (q_0) {};
                    \node[state] (q_7) [right of= q_0,draw=none] {};
                    \node[state] (q_1) [above of= q_7] {$1$};
                    \node[state] (q_2) [right of= q_1] {$2,y$};
                    \node[state] (q_3) [right of= q_2,draw=none] {};
                    \node[state] (q_4) [right of= q_3,draw=none] {};
                    \node[state] (q_5) [right of= q_4,draw=none] {};
                    \node[state] (q_6) [right of= q_5] {$1,x$};
                    \node[state] (q_8) [right of= q_7] {$1,y$};
                    \node[state] (q_9) [right of= q_8] {$2,y$};
                    \node[state] (q_10) [right of= q_9] {$1,x$};
                    \node[state] (q_11) [right of= q_10] {$2,x$};
                    \node[state] (q_12) [right of= q_6] {$2,y$};
                    \node[state] (q_14) [below of= q_7] {$2$};
                    \node[state] (q_15) [right of= q_14] {$2,x$};
                    \node[state] (q_16) [right of= q_15,draw=none] {};
                    \node[state] (q_17) [right of= q_16,draw=none] {};
                    \node[state] (q_18) [right of= q_17,draw=none] {};
                    \node[state] (q_19) [right of= q_18] {$1,y$};
                    \node[state] (q_13) [right of= q_19] {$2,x$};

                    \path[->] (q_0) edge node {$a$} (q_1)
                              (q_0) edge node {$b$} (q_14)
                              (q_1) edge [loop above] node {$a$} (q_1)
                              (q_1) edge node {$b$} (q_2)
                              (q_2) edge node {$a|xa=x$} (q_8)
                              (q_2) edge node {$a|xa=y$} (q_6)
                              (q_2) edge [bend right=30] node [swap] {$b$} (q_15)
                              (q_6) edge [bend right=10] node [swap] {$a$} (q_19)
                              (q_6) edge [bend right=20] node [swap] {$b$} (q_12)
                              (q_8) edge [loop below] node {$a$} (q_8)
                              (q_8) edge [bend right=20] node [swap] {$b$} (q_9)
                              (q_9) edge [bend right=20] node [swap] {$a$} (q_8)
                              (q_9) edge [bend left=30] node {$b$} (q_11)
                              (q_10) edge node {$b$} (q_9)
                              (q_10) edge [loop below] node {$a$} (q_10)
                              (q_11) edge node {$a$} (q_10)
                              (q_11) edge [loop below] node {$b$} (q_11)
                              (q_12) edge node {$b$} (q_13)
                              (q_12) edge [bend right=20] node [swap] {$a$} (q_6)
                              (q_13) edge node {$a$} (q_19)
                              (q_13) edge [loop below] node {$b$} (q_13)
                              (q_14) edge node {$a$} (q_1)
                              (q_14) edge node {$b$} (q_15)
                              (q_15) edge node [swap] {$a|xa=y$} (q_19)
                              (q_15) edge node [swap] {$a|xa=x$} (q_10)
                          (q_15) edge [loop below] node {$b$} (q_15)
                          (q_19) edge node [swap] {$b$} (q_12)
                          (q_19) edge [bend right=10] node [swap] {$a$} (q_6);

                \draw [->] (-.5,.5) -- (-.27,.27);

                \end{tikzpicture}}}\\
                \caption[]{3MSA for the case in which condition 2. of Proposition \ref{proposition:3-4} is false.}%
                \label{fig:prop:3-4:d:2}%
            \end{figure}
            \FloatBarrier
    \end{enumerate}
  \item Let $z\not\in\{1,2\}$.
    \begin{enumerate}
      \item Let condition $3.(a)$ be false, \textit{i.e.}, $b=[1,y]\backslash z$, $y\neq 2$ but $za=z$.
          Then for all $w\in a^*b$ and $u\in\{a,b\}^+$, $\M(w)=\{z\}$ and $\M(wu)=\{1,z\}$, so $\A$ is not $3$-compressible.
      \item Let condition $3.(b)$ be false, \textit{i.e.}, $b=[x,y]\backslash z$ and $\{x,y\}=\{1,2\}$, but $|Orb_a(z)|\leq2$ and $zab\in\{1,2\}$.
          If $x=1$ and $y=2$, then the 3MSA in Fig. \ref{fig:prop:3-4:d:3:2} proves that $\A$ is not $3$-compressible. Else, if $x=2$ and $y=1$, then the 3MSA in Fig. \ref{fig:prop:3-4:d:3:3} proves that $\A$ is not $3$-compressible.
          \begin{figure}[ht]%
              \centering
              \subfigure[][3MSA for the case $b={[1,2]}\backslash z$, $|Orb_a(z)|\leq2$ and $zab\in\{1,2\}$.]{%
              \label{fig:prop:3-4:d:3:2}%
              {\footnotesize
              \begin{tikzpicture}[shorten >=1pt,node distance=5em,auto]

                  \node[state] (q_0) {};
                  \node[state] (q_1) [right of= q_0] {$z$};
                  \node[state] (q_2) [right of= q_1] {$1$};
                  \node[state] (q_3) [right of= q_2] {$1,z$};
                  \node[state] (qq) [draw=none,right of= q_3] {};
                  \node[state] (q_4) [below of= q_0] {$1,za$};
                  \node[state] (q_5) [right of=q_4,draw=none] {};
                  \node[state] (q_6) [right of= q_5] {$1,z$};
                  \node[state] (q_7) [right of= q_6] {$1,za$};
                  \node[state] (q_8) [below of= q_4] {$2,z$};
                  \node[state] (q_9) [right of= q_8] {$1,za$};
                  \node[state] (q_10) [right of= q_9] {$1,z$};

                  \path[->] (q_0) edge [bend left=40] node {$a$} (q_2)
                            (q_0) edge node {$b$} (q_1)
                            (q_1) edge [bend right=30] node [swap] {$a|za=2,b$} (q_3)
                            (q_1) edge node [swap] {$a|za\neq z$} (q_4)
                            (q_2) edge [loop above] node {$a$} (q_2)
                            (q_2) edge node [swap] {$b$} (q_1)
                            (q_3) edge [loop above] node {$a,b$} (q_3)
                            (q_4) edge node {$b|zab=1,a$} (q_6)
                            (q_4) edge node {$b|zab=2$} (q_8)
                            (q_6) edge [loop below] node {$b$} (q_6)
                            (q_6) edge [bend left=10] node {$a$} (q_7)
                            (q_7) edge [bend left=10] node {$a,b$} (q_6)
                            (q_8) edge [bend left=10] node {$a$} (q_9)
                            (q_8) edge [bend right=40] node [swap] {$b$} (q_10)
                            (q_9) edge [bend left=10] node {$b$} (q_8)
                            (q_9) edge [bend left=10] node {$a$} (q_10)
                            (q_10) edge [bend left=10] node {$a$} (q_9)
                            (q_10) edge [loop right] node {$b$} (q_10);

                  \draw [->] (-.5,.5) -- (-.27,.27);

              \end{tikzpicture}}%
              }
              \hspace{20pt}%
              \subfigure[][3MSA for the case $b={[2,1]}\backslash z$, $|Orb_a(z)|\leq 2$ and $zab\in\{1,2\}$.]{%
              \label{fig:prop:3-4:d:3:3}%
              {\footnotesize
              \begin{tikzpicture}[shorten >=1pt,node distance=5em,auto]

                  \node[state] (q_1) [draw=none] {};
                  \node[state] (q_0) [right of= q_1] {};
                  \node[state] (q_2) [right of= q_0] {$z$};
                  \node[state] (q_11) [right of=q_2] {$1$};
                  \node[state] (q_3) [below of= q_0] {$1,za$};
                  \node[state] (q_4) [right of= q_3,draw=none] {};
                  \node[state] (q_5) [right of= q_4] {$1,z$};
                  \node[state] (q_6) [left of= q_3] {$1,z$};
                  \node[state] (q_7) [below of= q_6] {$2,z$};
                  \node[state] (q_8) [below of= q_4] {$2,z$};
                  \node[state] (q_9) [right of= q_5] {$2,z$};

                  \path[->] (q_0) edge node {$b$} (q_2)
                            (q_0) edge [bend left=40] node {$a$} (q_11)
                            (q_11) edge [loop above] node {$a$} (q_11)
                            (q_11) edge node [swap] {$b$} (q_2)
                            (q_2) edge node {$b$} (q_8)
                            (q_2) edge node {$a|za=z$} (q_5)
                            (q_2) edge node [swap] {$a|za\neq z$} (q_3)
                        (q_3) edge [bend left=10] node {$a,b$} (q_6)
                        (q_5) edge [loop left] node {$a$} (q_5)
                    (q_5) edge [bend left=10] node {$b$} (q_9)
                    (q_6) edge [bend left=10] node {$a$} (q_3)
                    (q_6) edge node [swap] {$b$} (q_7)
                    (q_7) edge node {$a$} (q_3)
                    (q_7) edge [loop below] node {$b$} (q_7)
                    (q_8) edge [loop below] node {$b$} (q_8)
                    (q_8) edge node [swap] {$a|za=z$} (q_5)
                    (q_8) edge node {$a|za\neq z$} (q_3)
                    (q_9) edge [bend left=10] node {$a$} (q_5)
                    (q_9) edge [loop below] node {$b$} (q_9);

              \draw [->] (.9,.5) -- (1.13,.27);

              \end{tikzpicture}}}\\
              \caption[]{3MSA for the case in which condition 3.(b) of Proposition \ref{proposition:3-4} is false.}%
          \end{figure}
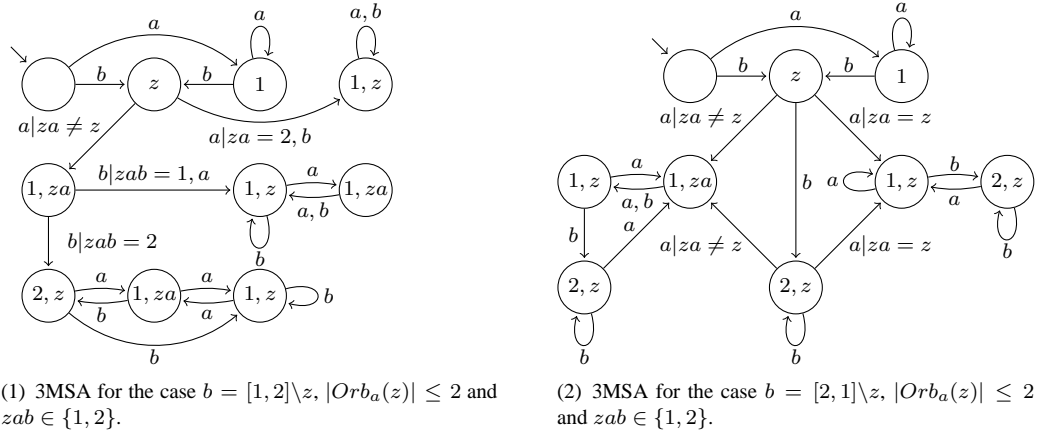
    \end{enumerate}
\end{enumerate}

\FloatBarrier

Conversely, we find for any automaton satisfying conditions $1.-3.$ a (short) $3$-compressing word.
\begin{enumerate}
  \item Let $b=[x,y]\backslash 1$ and $Orb_a(x)\not\subseteq \{x,y\}$, whence $xa\neq x$.
      Then either $xa$ or $xa^2$ are different from $y$, and so either the word $b^2ab$ or $b^2a^2b$ $3$-compresses $\A$.
  \item Let $z=2$, we consider various subcases.
      \begin{enumerate}
        \item Let $b=[1,y]\backslash 2$.
            Then $2a\neq 2$, otherwise $Orb_a(2)=\{2\}$ and $2ab=1$, against the hypothesis.
            Hence $2ab=y$ and $2a^2b\neq y$.
            The P3MSA in Fig. \ref{fig:prop:3-4:r:2:1} proves that either the word $b^2ab^2$ or $b^2a^2b^2$ $3$-compresses $\A$;
            \begin{figure}[ht]%
                \centering
                {\footnotesize
                \begin{tikzpicture}[shorten >=1pt,node distance=5em,auto]

                    \node[state] (q_0) {};
                    \node[state] (q_1) [right of= q_0] {$2$};
                    \node[state] (q_2) [right of= q_1] {$1,2$};
                    \node[state] (q_3) [right of= q_2] {$1,2a$};
                    \node[state] (q_4) [right of= q_3,draw=none] {};
                    \node[state] (q_5) [right of= q_4] {$2,2ab$};
                    \node[state,accepting] (q_8) [right of= q_5] {};
                    \node[state] (q_7) [right of= q_8] {$2,2a^2b$};
                    \node[state] (q_6) [right of= q_7] {$1,2a^2$};

                    \path[->] (q_0) edge node {$b$} (q_1)
                              (q_1) edge node {$b$} (q_2)
                              (q_2) edge node {$a$} (q_3)
                              (q_3) edge node {$b|2ab\neq y$} (q_5)
                              (q_5) edge node {$b$} (q_8)
                              (q_3) edge [bend right=25] node [swap] {$a|2ab=y$} (q_6)
                              (q_6) edge node [swap] {$b$} (q_7)
                              (q_7) edge node [swap] {$b$} (q_8);

                    \draw [->] (-.5,.5) -- (-.27,.27);

                \end{tikzpicture}}
                \caption[]{P3MSA for the case 2.(a) of Proposition \ref{proposition:3-4}.}%
                \label{fig:prop:3-4:r:2:1}%
            \end{figure}
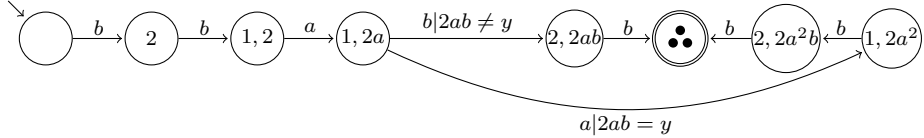
        \item Let $b=[x,1]\backslash 2$:
            \begin{enumerate}
              \item if $\{2ab,xab\}\neq\{1,x\}$, then the P3MSA in Fig. \ref{fig:prop:3-4:r:2:2:1} proves that the word $b^2ab^2$ or $b^2a^2b^2$ or $b^2abab^2$ $3$-compresses $\A$.
                  Observe that if $xab=x$, then $xa=2$ and $2ab\not\in\{1,x\}$.
                  \begin{figure}[ht]%
                    \centering
                    {\footnotesize
                    \begin{tikzpicture}[shorten >=1pt,node distance=4.8em,auto]

                        \node[state] (q_0) {};
                        \node[state] (q_1) [right of= q_0] {$2$};
                        \node[state] (q_2) [right of= q_1] {$2,x$};
                        \node[state] (q_3) [right of= q_2] {$1,xa$};
                        \node[state] (q_4) [right of= q_3,draw=none] {};
                        \node[state] (q_5) [right of= q_4] {$1,2$};
                        \node[state] (q_6) [right of= q_5] {$1,2a$};
                        \node[state] (q_7) [right of= q_6] {$2,2ab$};
                        \node[state] (q_8) [below of= q_4] {$1,2a$};
                        \node[state] (q_9) [right of= q_8] {$2,2ab$};
                        \node[state] (q_10) [right of= q_9] {$2,xab$};
                        \node[state,accepting] (q_11) [right of= q_10] {};

                        \path[->] (q_0) edge node {$b$} (q_1)
                                  (q_1) edge node {$b$} (q_2)
                                  (q_2) edge node {$a$} (q_3)
                                  (q_3) edge node {$b|xab=1$} (q_5)
                                  (q_3) edge node {$b|xab\not\in\{1,x\}$} (q_10)
                                  (q_3) edge node [swap] {$a|xab=x$} (q_8)
                                  (q_5) edge node {$a$} (q_6)
                                  (q_6) edge node {$b$} (q_7)
                                  (q_7) edge node {$b$} (q_11)
                                  (q_10) edge node {$b$} (q_11)
                                  (q_8) edge node {$b$} (q_9)
                                  (q_9) edge [bend right=28] node [swap] {$b$} (q_11);

                        \draw [->] (-.5,.5) -- (-.27,.27);

                    \end{tikzpicture}}\\
                    \caption[]{P3MSA for the case $b={[x,1]}\backslash 2$ and $\{2ab,xab\}\neq\{1,x\}$.}%
                    \label{fig:prop:3-4:r:2:2:1}%
                \end{figure}
              \item if $2ab=x$, $xab=1$, and $\{2a,xa\}\neq\{2,x\}$, then $2ab=2b$, $2a=2$ and then $xa\not\in\{2,x\}$. So $\M(b^2a)=\{1,xa\}$, $\M(b^2a^2)=\{1,xa^2\}$ and $\M(b^2a^2b)=\{2,xa^2b\}$. If $xa^2b=x=2b$, then $xa^2=2a$, and so $xa=2$, against the hypothesis. Else, if $xa^2b=1=xab$, then $xa^2=xa$, and so $xa=x$, against the hypothesis. Then it follows that $|\M(b^2a^2b^2)|=3$ and the word $b^2a^2b^2$ $3$-compresses $\A$.
              \item if $2ab=1$, $xab=x$, and $\{2a,xa\}\neq\{2,x\}$, then $xab=2b$, $xa=2$ and then $2a\not\in\{2,x\}$.
                  Then the P3MSA in Fig. \ref{fig:prop:3-4:r:2:2} proves that any word belonging to the language $\mathcal{L}=b(a^+b)^*(ba)^+a(ba)^*ab(b+ab^2)$ (and in particular $b^2a^3b^2$) $3$-compresses $\A$.
                  Indeed, if $2a^2b=x=2b$, then $2a^2=2$, against the hypothesis, else, if $2a^2b=1=2ab$, then $2a^2=2a$, and so $2a=2$, against the hypothesis, and so $2a^2b\not\in\{1,x\}$. Moreover, if $2a^2bab=1=2ab$, then $2a^2ba=2a$, and so $2a^2b=2$, against the hypothesis else, if $2a^2bab=x=xab$, then $xa^2ba=xa$, and so $2a^2b=x$, against the hypothesis, and so $2a^2bab\not\in\{1,x\}$.
                  \begin{figure}[ht]%
                    \centering
                    {\footnotesize
                    \begin{tikzpicture}[shorten >=1pt,node distance=4.8em,auto]

                        \node[state] (q_0) {};
                        \node[state] (q_1) [right of= q_0] {$2$};
                        \node[state] (q_2) [right of= q_1] {$2,x$};
                        \node[state] (q_3) [right of= q_2] {$1,2$};
                        \node[state] (q_4) [right of= q_3] {$1,2a$};
                        \node[state] (q_5) [right of= q_4] {$1,2a^2$};
                        \node[state,inner sep=-4pt] (q_6) [right of= q_5] {{\scriptsize $\begin{array}{c}2, \\2,2a^2b\end{array}$}};
                        \node[state,accepting] (q_7) [below of= q_6] {};
                        \node[state,inner sep=-4pt] (q_8) [right of= q_6] {{\scriptsize $\begin{array}{c}1, \\2a^2ba\end{array}$}};
                        \node[state,inner sep=-4pt] (q_9) [below of= q_8] {{\scriptsize $\begin{array}{c}2, \\2a^2bab\end{array}$}};
                        \node[state] (q_10) [below of= q_1] {$1$};

                        \path[->] (q_0) edge node {$b$} (q_1)
                                  (q_1) edge node {$b$} (q_2)
                                  (q_2) edge [loop below] node {$b$} (q_2)
                                  (q_2) edge [bend left=10] node {$a$} (q_3)
                                  (q_3) edge [bend left=10] node {$b$} (q_2)
                                  (q_3) edge [bend left=10] node {$a$} (q_4)
                                  (q_4) edge [bend left=10] node {$b$} (q_3)
                                  (q_4) edge node {$a$} (q_5)
                                  (q_5) edge node {$b$} (q_6)
                                  (q_6) edge node {$b$} (q_7)
                                  (q_6) edge node {$a$} (q_8)
                                  (q_8) edge node {$b$} (q_9)
                                  (q_9) edge node [swap] {$b$} (q_7)
                                  (q_1) edge [bend left=10] node {$a$} (q_10)
                                  (q_10) edge [bend left=10] node {$b$} (q_1)
                                  (q_10) edge [loop left] node {$a$} (q_10);

                        \draw [->] (-.5,.5) -- (-.27,.27);

                    \end{tikzpicture}}\\
                    \caption[]{P3MSA for the case $b={[x,1]}\backslash 2$, $2ab=1$, $xab=x$ and $\{2a,xa\}\neq\{2,x\}$.}%
                    \label{fig:prop:3-4:r:2:2}%
                \end{figure}
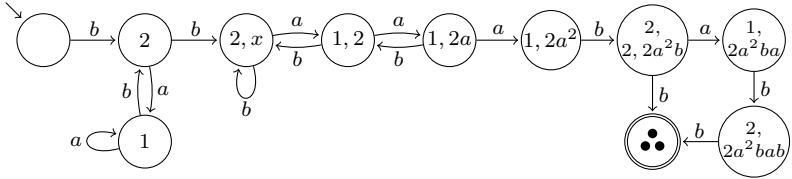
            \end{enumerate}
        \item Let $b=[x,y]\backslash 2$, $1b=y$, $1\not\in\{x,y\}$ and $\{xa,ya\}\neq \{x,y\}$.
            Observe that if $xa\in\{x,y\}$, then $ya\not\in \{x,y\}$.
            The P3MSA in Fig. \ref{fig:prop:3-4:r:2:3} proves that in this case either the word $b^2ab$ or $b^2abab$ $3$-compresses $\A$.
            \begin{figure}[ht]%
            \centering
            {\footnotesize
            \begin{tikzpicture}[shorten >=1pt,node distance=5em,auto]

                \node[state] (q_0) {};
                \node[state] (q_1) [right of= q_0] {$2$};
                \node[state] (q_2) [right of= q_1] {$2,x$};
                \node[state] (q_3) [right of= q_2] {$1,xa$};
                \node[state] (q_4) [right of= q_3,draw=none] {};
                \node[state,accepting] (q_5) [right of= q_4] {};
                \node[state] (q_6) [right of= q_5] {$1,ya$};
                \node[state] (q_7) [right of= q_6] {$2,y$};

                \path[->] (q_0) edge node {$b$} (q_1)
                          (q_1) edge node {$b$} (q_2)
                          (q_2) edge node {$a$} (q_3)
                          (q_3) edge node {$b|xa\not\in\{x,y\}$} (q_5)
                          (q_3) edge [bend right=18] node [swap] {$b|xa\in\{x,y\}$} (q_7)
                          (q_7) edge node [swap] {$a$} (q_6)
                          (q_6) edge node [swap] {$b$} (q_5);

                \draw [->] (-.5,.5) -- (-.27,.27);

            \end{tikzpicture}}\\
            \caption[]{P3MSA for the case $b=[x,y]\backslash 2$, $1b=y$, $1\not\in\{x,y\}$ and $\{xa,ya\}\neq \{x,y\}$.}%
            \label{fig:prop:3-4:r:2:3}%
            \end{figure}
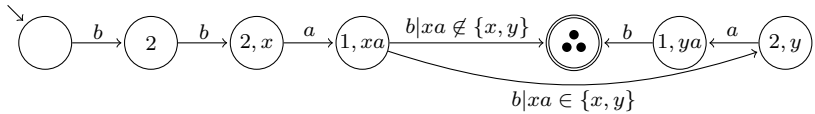
    \end{enumerate}
  \item Let $z\not\in\{1,2\}$, we consider two main subcases.
    \begin{enumerate}
      \item Let $b=[1,y]\backslash z$, $y\neq 2$ and $za\neq z$. Then $\M(ba)=\{1,za\}$, $\M(bab)=\{z,zab\}$.
          If $zab\neq y$, then $|\M(bab^2)|=3$, else if $zab=y$ then $|\M(baba)|=3$, hence either the word $bab^2$ or $baba$ $3$-compresses $\A$.
      \item Let $\{x,y\}=\{1,2\}$. We consider two subcases.
        \begin{enumerate}
          \item Let $b=[x,y]\backslash z$ and $|Orb_a(z)|>2$.
            In particular $za\neq z$, $zab\neq za^2b$ and $z\neq za^2$.
            If $zab=y$, then $za^2b\not\in\{zb,zab\}=\{1,2\}$. The P3MSA in Fig. \ref{fig:prop:3-4:r:3:2} proves that either the word $ba^2b^2$ or $bab^2$ $3$-compresses $\A$.
            \begin{figure}[ht]%
            \centering
            {\footnotesize
            \begin{tikzpicture}[shorten >=1pt,node distance=5em,auto]

                \node[state] (q_0) {};
                \node[state] (q_1) [right of= q_0] {$z$};
                \node[state] (q_2) [right of= q_1] {$1,za$};
                \node[state] (q_3) [right of= q_2,draw=none] {};
                \node[state] (q_4) [right of= q_3] {$z,zab$};
                \node[state,accepting] (q_7) [right of= q_4] {};
                \node[state] (q_6) [right of= q_7] {$z,za^2b$};
                \node[state] (q_5) [right of= q_6] {$1,za^2$};

                \path[->] (q_0) edge node {$b$} (q_1)
                          (q_1) edge node {$a$} (q_2)
                          (q_2) edge node {$b|zab\neq y$} (q_4)
                          (q_2) edge [bend right=22] node [swap] {$a|zab=y$} (q_5)
                          (q_5) edge node [swap] {$b$} (q_6)
                          (q_6) edge node [swap] {$b$} (q_7)
                          (q_4) edge node {$b$} (q_7);

                \draw [->] (-.5,.5) -- (-.27,.27);

            \end{tikzpicture}}\\
            \caption[]{P3MSA for the case $b={[x,y]}\backslash z$, $\{x,y\}=\{1,2\}$ and $|Orb_a(z)|>2$.}%
            \label{fig:prop:3-4:r:3:2}%
            \end{figure}
          \item Let $b=[x,y]\backslash z$ and $zab\not\in\{1,2\}$.
            Then $\M(bab)=\{z,zab\}$ and $|\M(bab^2)|=|\M(baba)|=3$, both the words $bab^2$ and $baba$ $3$-compress $\A$.
        \end{enumerate}
    \end{enumerate}
\end{enumerate}
\end{proof}

\FloatBarrier


\begin{lemma}\label{lemma:4-4}
Let $\A$ be a $(\bf{4},\bf{4})$ $3$-compressible automaton with $a=[1,2]\backslash 3$, $3a=1$ and $b=[x,y]\backslash z$, $zb=x$.
If $\A$ is proper then all the following conditions hold:
\begin{enumerate}
  \item $\{1,2\}\cap\{x,z\}\neq \emptyset$,
  \item $z\in\{1,2\}$ or $\{3,za\}\cap\{x,y\}\neq \emptyset$,
  \item $\{3,3b\}\cap\{x,y\}\neq \emptyset$,
  \item $\{1,3\}\cap\{x,y\}\neq \emptyset$,
  \item $3\in \{x,y\}$ or $\{z,3b\}\cap\{1,2\}\neq \emptyset$,
  \item $\{z,za\}\cap\{1,2\}\neq \emptyset$.
\end{enumerate}
\end{lemma}

\begin{proof}
If $\{1,2\}\cap\{x,z\}=\emptyset$, then $\M(b^2)=\{x,z\}$ and $|\M(b^2a)|=3$;
if $z\not\in\{1,2\}$ and $\{3,za\}\cap\{x,y\}=\emptyset$, then $\M(ba)=\{3,za\}$ and $|\M(bab)|=3$;
if $\{3,3b\}\cap\{x,y\}=\emptyset$, then $\M(ab)=\{3,zb\}$ and $|\M(ab^2)|=3$;
if $\{1,3\}\cap\{x,y\}=\emptyset$, then $\M(a^2)=\{1,3\}$ and $|\M(a^2b)|=3$;
if $3\not\in \{x,y\}$ and $\{z,3b\}\cap\{1,2\}=\emptyset$, then $\M(ab)=\{z,3b\}$ and $|\M(a^2b)|=3$;
if $\{z,za\}\cap\{1,2\}=\emptyset$, then $\M(ba)=\{z,3b\}$ and $|\M(bab)|=3$.
Then each automaton that does not satisfies one of the above is not proper.
\end{proof}

\begin{corollary}\label{corollary:4-4}
Let $\A$ be a \emph{proper} $(\bf{4},\bf{4})$ $3$-compressible automaton with $a=[1,2]\backslash 3,\ 3a=1$ and $b=[x,y]\backslash z,\ zb=x$.
The following conditions hold:
\begin{enumerate}
  \item if $z=1$, then $3\in \{x,y\}$,
  \item if $z=2$, then $3\in \{x,y\}$ or it is $1\in \{x,y\}$ and $3b\in\{x,y\}$,
  \item if $z\not\in\{1,2\}$, then $x\in \{1,2\}$ and $za\in\{1,2\}$.
\end{enumerate}
\end{corollary}

\begin{proof}
It is a straightforward consequence of the previous lemma.
Observe that not all the conditions of Lemma \ref{lemma:4-4} are applied, so some automaton satisfying the conditions of the corollary could possibly be not proper or not $3$-compressible.
\end{proof}

\begin{proposition}\label{proposition:4-4}
Let $\A$ be a $(\bf{4},\bf{4})$-automaton with $a=[1,2]\backslash 3,\ 3a=1$ and $b=[x,y]\backslash z,\ zb=x$.
Then $\A$ is proper and $3$-compressible if, and only if, the following conditions hold:
\begin{enumerate}
  \item if $z=1$, then $y=3$ and it is $xa\neq 2$ or $2b \neq 3$,
  \item if $z=2$ and $3\not\in\{x,y\}$, then $1\in\{x,y\}$ and $3b \in\{x,y\}$,
  \item if $z=2$ and $3\in\{x,y\}$, if $q\in\{x,y\}\setminus\{3\}$ then $qa\neq 2$ or $1b \neq y$,
  \item if $b=[x,y]\backslash z$, $z\not\in\{1,2\}$, $x\in\{1,2\}$ and $za\in\{1,2\}$, then $y=3$.
\end{enumerate}
Moreover, if $\mathcal{A}$ is proper and $3$-compressible, then either the word $b^2a^2$ or $b^2ab^2$ or $a^2b^2$ or $a^2ba^2$ $3$-compresses $\A$.
\end{proposition}

\begin{proof}
First observe that from Corollary \ref{corollary:4-4}, a proper $3$-compressible $(\bf{4},\bf{4})$-automaton always satisfies the antecedent of one of the above conditions, so all the possible cases are taken into account.

\noindent We start proving that a $(\bf{4},\bf{4})$-automaton that does not satisfy conditions 1.-4. is not proper or it is not $3$-compressible.
\begin{enumerate}
  \item Let condition 1. be false, \textit{i.e.}, $b=[x,y]\backslash 1$ but either $y\neq 3$ or $xa=2$ and $2b=3$.
      Observe that if $y\neq 3$ then from the previous corollary we have $x=3$.
      The 3MSA in Fig. \ref{fig:prop:4-4:d:1} proves that if condition 1. is false, then $\A$ is not $3$-compressible.
      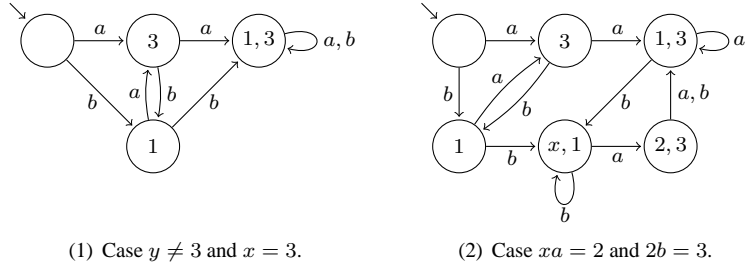
\begin{figure}[ht]%
        \centering
        \subfigure[][Case $y\neq 3$ and $x=3$.]{%
        \label{fig:prop:4-4:d:1:1}%
        {\footnotesize
        \begin{tikzpicture}[shorten >=1pt,node distance=5em,auto]

            \node[state] (q_0) {};
            \node[state] (q_1) [right of= q_0] {$3$};
            \node[state] (q_2) [right of= q_1] {$1,3$};
            \node[state] (q_3) [below of= q_1] {$1$};

            \path[->] (q_0) edge node {$a$} (q_1)
                      (q_0) edge node [swap] {$b$} (q_3)
                      (q_1) edge [bend left=10] node {$b$} (q_3)
                      (q_1) edge node {$a$} (q_2)
                      (q_2) edge [loop right] node {$a,b$} (q_2)
                      (q_3) edge [bend left=10] node {$a$} (q_1)
                      (q_3) edge [draw=none,loop below] node {\phantom{b}} (q_3)
                      (q_3) edge node [swap] {$b$} (q_2);

            \draw [->] (-.5,.5) -- (-.27,.27);

        \end{tikzpicture}}%
        }
        \hspace{20pt}%
        \subfigure[][Case $xa=2$ and $2b=3$.]{%
        \label{fig:prop:4-4:d:1:2}%
        {\footnotesize
        \begin{tikzpicture}[shorten >=1pt,node distance=5em,auto]

            \node[state] (q_0) {};
            \node[state] (q_1) [right of= q_0] {$3$};
            \node[state] (q_2) [right of= q_1] {$1,3$};
            \node[state] (q_3) [below of= q_0] {$1$};
            \node[state] (q_4) [right of= q_3] {$x,1$};
            \node[state] (q_5) [right of= q_4] {$2,3$};

            \path[->] (q_0) edge node {$a$} (q_1)
                      (q_0) edge node [swap] {$b$} (q_3)
                      (q_1) edge [bend left=10] node {$b$} (q_3)
                      (q_1) edge node {$a$} (q_2)
                      (q_2) edge [loop right] node {$a$} (q_2)
                      (q_2) edge node {$b$} (q_4)
                      (q_3) edge [bend left=10] node {$a$} (q_1)
                      (q_3) edge node [swap] {$b$} (q_4)
                      (q_4) edge [loop below] node {$b$} (q_4)
                      (q_4) edge node [swap] {$a$} (q_5)
                      (q_5) edge node [swap] {$a,b$} (q_2);

            \draw [->] (-.5,.5) -- (-.27,.27);

        \end{tikzpicture}}}\\
        \caption[]{3MSA for the case in which condition 1. of Proposition \ref{proposition:4-4} is false.}%
        \label{fig:prop:4-4:d:1}%
      \end{figure}
  \item Let condition 2. be false, \textit{i.e.}, $b=[x,y]\backslash 2$ and $3\not\in\{x,y\}$ but either $1\not\in\{x,y\}$ or $3b\not\in\{x,y\}$.
      From conditions 3. and 4. of Lemma \ref{lemma:4-4} we have that in this cases $\A$ is not proper, as it is $3$-compressed either by $ab^2$ or by $a^2b$.
  \item Let condition 3. be false. If $b=[x,3]\backslash 2$ but $xa=2$ and $1b=3$, then the 3MSA in Fig. \ref{fig:prop:4-4:d:2:1} proves that  $\A$ is not $3$-compressible.
      Else, if $b=[3,y]\backslash 2$ but $ya=2$ and $1b=y$, then the 3MSA in Fig. \ref{fig:prop:4-4:d:2:2} proves $\A$ is not $3$-compressible.
      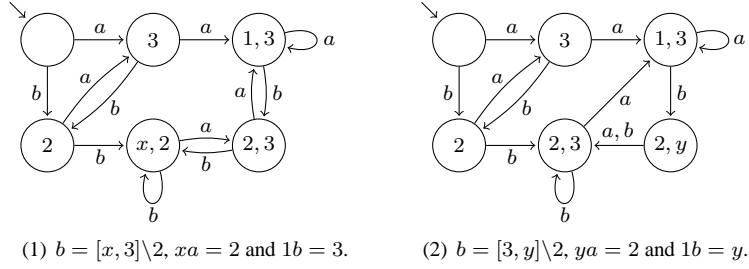
\begin{figure}[ht]%
        \centering
        \subfigure[][$b={[x,3]}\backslash 2$, $xa=2$ and $1b=3$.]{%
        \label{fig:prop:4-4:d:2:1}%
        {\footnotesize
        \begin{tikzpicture}[shorten >=1pt,node distance=5em,auto]

            \node[state] (q_0) {};
            \node[state] (q_1) [right of= q_0] {$3$};
            \node[state] (q_2) [right of= q_1] {$1,3$};
            \node[state] (q_3) [below of= q_0] {$2$};
            \node[state] (q_4) [right of= q_3] {$x,2$};
            \node[state] (q_5) [right of= q_4] {$2,3$};

            \path[->] (q_0) edge node {$a$} (q_1)
                      (q_0) edge node [swap] {$b$} (q_3)
                      (q_1) edge [bend left=10] node {$b$} (q_3)
                      (q_1) edge node {$a$} (q_2)
                      (q_2) edge [loop right] node {$a\phantom{,b}$} (q_2)
                      (q_2) edge [bend left=10]node {$b$} (q_5)
                      (q_3) edge [bend left=10] node {$a$} (q_1)
                      (q_3) edge node [swap] {$b$} (q_4)
                      (q_4) edge [loop below] node {$b$} (q_4)
                      (q_4) edge [bend left=10] node {$a$} (q_5)
                      (q_5) edge [bend left=10] node {$b$} (q_4)
                      (q_5) edge [bend left=10] node {$a$} (q_2);

            \draw [->] (-.5,.5) -- (-.27,.27);

        \end{tikzpicture}}%
        }
        \hspace{20pt}%
        \subfigure[][$b={[3,y]}\backslash 2$, $ya=2$ and $1b=y$.]{%
        \label{fig:prop:4-4:d:2:2}%
        {\footnotesize
        \begin{tikzpicture}[shorten >=1pt,node distance=5em,auto]

            \node[state] (q_0) {};
            \node[state] (q_1) [right of= q_0] {$3$};
            \node[state] (q_2) [right of= q_1] {$1,3$};
            \node[state] (q_3) [below of= q_0] {$2$};
            \node[state] (q_4) [right of= q_3] {$2,3$};
            \node[state] (q_5) [right of= q_4] {$2,y$};

            \path[->] (q_0) edge node {$a$} (q_1)
                      (q_0) edge node [swap] {$b$} (q_3)
                      (q_1) edge [bend left=10] node {$b$} (q_3)
                      (q_1) edge node {$a$} (q_2)
                      (q_2) edge [loop right] node {$a$} (q_2)
                      (q_2) edge node {$b$} (q_5)
                      (q_3) edge [bend left=10] node {$a$} (q_1)
                      (q_3) edge node [swap] {$b$} (q_4)
                      (q_4) edge [loop below] node {$b$} (q_4)
                      (q_4) edge node [swap] {$a$} (q_2)
                      (q_5) edge node [swap] {$a,b$} (q_4);

            \draw [->] (-.5,.5) -- (-.27,.27);

        \end{tikzpicture}}}\\
        \caption[]{3MSA for the case in which condition 3. of Proposition \ref{proposition:4-4} is false.}%
        \label{fig:prop:4-4:d:2}%
      \end{figure}
  \item Let condition 4. be false, \textit{i.e.}, $b=[x,y]\backslash z$, $z\not\in\{1,2\}$, $x\in\{1,2\}$ and $za\in\{1,2\}$, but $y\neq 3$.
      We have to consider the following subcases.
      \begin{enumerate}
        \item Let $za=1$, hence $z=3$ and $zb=x$.
            If $x=1$, then for all $w\in\{a,b\}^+$ with $|w|\geq 2$, we have $\M(w)=\{1,3\}$, and then $\A$ is not $3$-compressible.
            If $x=2$, then by condition 4 of Lemma \ref{lemma:4-4} we have $y=1$, and for all $w\in\{a,b\}^+$ we have $M(wa)=\{1,3\}$ and $\M(wb)=\{2,3\}$ and again $\A$ is not $3$-compressible.
        \item Let $za=2$, hence $z\neq 3$.
            We consider two further subcases:
            \begin{enumerate}
              \item if $x=1$, then $3b\neq 1$ and by condition 2. of Lemma \ref{lemma:4-4} $za=y=2$, and by condition 3. of Lemma \ref{lemma:4-4} we have $3b=2$.
                  The 3MSA in Fig. \ref{fig:prop:4-4:d:5:1} proves that $\A$ is not $3$-compressible;
              \item if $x=2$, then by condition 4. of Lemma \ref{lemma:4-4} we have $y=1$ and by condition 3. of Lemma \ref{lemma:4-4} we have $3b=1$.
                  The 3MSA in Fig. \ref{fig:prop:4-4:d:5:2} proves that $\A$ is not $3$-compressible.
                  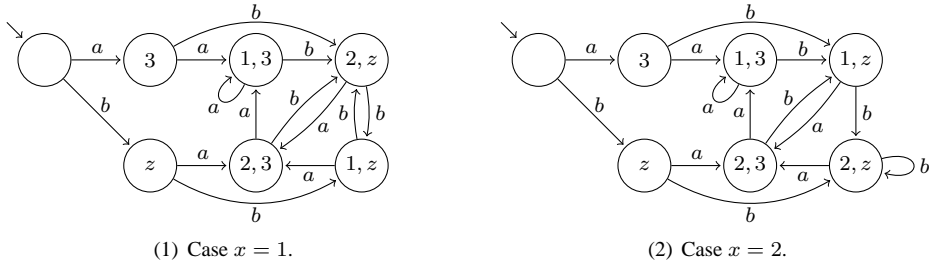
\begin{figure}[ht]%
                    \centering
                    \subfigure[][Case $x=1$.]{%
                    \label{fig:prop:4-4:d:5:1}%
                    {\footnotesize
                    \begin{tikzpicture}[shorten >=1pt,node distance=5em,auto]

                        \node[state] (q_0) {};
                        \node[state] (q_1) [right of= q_0] {$3$};
                        \node[state] (q_2) [right of= q_1] {$1,3$};
                        \node[state] (q_3) [right of= q_2] {$2,z$};
                        \node[state] (q_4) [below of= q_1] {$z$};
                        \node[state] (q_5) [right of= q_4] {$2,3$};
                        \node[state] (q_6) [right of= q_5] {$1,z$};

                        \path[->] (q_0) edge node {$a$} (q_1)
                                  (q_0) edge node {$b$} (q_4)
                                  (q_1) edge node {$a$} (q_2)
                                  (q_1) edge [bend left=30] node {$b$} (q_3)
                                  (q_2) edge [max distance=5mm,in=215,out=245,loop] node {$a$} (q_2)
                                  (q_2) edge node {$b$} (q_3)
                                  (q_3) edge [bend left=10] node {$a$} (q_5)
                                  (q_3) edge [bend left=10] node {$b$} (q_6)
                                  (q_4) edge node {$a$} (q_5)
                                  (q_4) edge [bend right=30] node [swap] {$b$} (q_6)
                                  (q_5) edge node {$a$} (q_2)
                                  (q_5) edge [bend left=10] node {$b$} (q_3)
                                  (q_6) edge node {$a$} (q_5)
                                  (q_6) edge [bend left=10] node {$b$} (q_3)
                                  (q_6) edge [loop right,draw=none] node {} (q_6);

                        \draw [->] (-.5,.5) -- (-.27,.27);

                    \end{tikzpicture}}%
                    }
                    \hspace{20pt}%
                    \subfigure[][Case $x=2$.]{%
                    \label{fig:prop:4-4:d:5:2}%
                    {\footnotesize
                    \begin{tikzpicture}[shorten >=1pt,node distance=5em,auto]

                        \node[state] (q_0) {};
                        \node[state] (q_1) [right of= q_0] {$3$};
                        \node[state] (q_2) [right of= q_1] {$1,3$};
                        \node[state] (q_3) [right of= q_2] {$1,z$};
                        \node[state] (q_4) [below of= q_1] {$z$};
                        \node[state] (q_5) [right of= q_4] {$2,3$};
                        \node[state] (q_6) [right of= q_5] {$2,z$};

                        \path[->] (q_0) edge node {$a$} (q_1)
                                  (q_0) edge node {$b$} (q_4)
                                  (q_1) edge node {$a$} (q_2)
                                  (q_1) edge [bend left=30] node {$b$} (q_3)
                                  (q_2) edge [max distance=5mm,in=215,out=245,loop] node {$a$} (q_2)
                                  (q_2) edge node {$b$} (q_3)
                                  (q_3) edge [bend left=10] node {$a$} (q_5)
                                  (q_3) edge node {$b$} (q_6)
                                  (q_4) edge node {$a$} (q_5)
                          (q_4) edge [bend right=30] node [swap] {$b$} (q_6)
                          (q_5) edge node {$a$} (q_2)
                          (q_5) edge [bend left=10] node {$b$} (q_3)
                          (q_6) edge node {$a$} (q_5)
                          (q_6) edge [loop right] node {$b$} (q_3);

                \draw [->] (-.5,.5) -- (-.27,.27);

                \end{tikzpicture}}}
                \caption[]{3MSA for the case in which condition 4. of Proposition \ref{proposition:4-4} is false and $za=2$.}%
                \label{fig:prop:4-4:d:5}%
                \end{figure}
            \end{enumerate}
      \end{enumerate}
\end{enumerate}

Conversely, we find for each automaton satisfying conditions 1.-4. a (short) $3$-compressing word.
\begin{enumerate}
  \item Let $b=[x,3]\backslash 1$ and $xa\neq 2$ or $2b\neq 3$.
      Observe that $xa\neq 1$, as $3a=1$ and $x\neq 3$.
      The P3MSA in Fig. \ref{fig:prop:4-4:r:1} proves that either the word $b^2a^2$ or $b^2ab^2$ $3$-compresses $\A$.
      \begin{figure}[ht]%
        \centering
        {\footnotesize
        \begin{tikzpicture}[shorten >=1pt,node distance=5em,auto]

            \node[state] (q_0) {};
            \node[state] (q_1) [right of= q_0] {$1$};
            \node[state] (q_2) [right of= q_1] {$1,x$};
            \node[state] (q_3) [right of= q_2] {$3,xa$};
            \node[state] (q) [right of= q_3,draw=none] {};
            \node[state,accepting] (q_4) [right of= q] {};
            \node[state] (q_5) [right of= q_4] {$1,2b$};

            \path[->] (q_0) edge node {$b$} (q_1)
                      (q_1) edge node {$b$} (q_2)
                      (q_2) edge node {$a$} (q_3)
                      (q_3) edge node {$a|xa\neq 2$} (q_4)
                      (q_3) edge [bend right=20] node [swap] {$b|xa=2,2b\neq 3$} (q_5)
                      (q_5) edge node [swap] {$b$} (q_4);

            \draw [->] (-.5,.5) -- (-.27,.27);

        \end{tikzpicture}}\\
        \caption[]{P3MSA for the case $b={[x,3]}\backslash 1$ and $xa\neq 2$ or $2b\neq 3$.}%
        \label{fig:prop:4-4:r:1}%
      \end{figure}
  \item Let $b=[x,y]\backslash 2$, $3\not\in\{x,y\}$, $1\in\{x,y\}$ and $3b\in\{x,y\}$.
      There are two subcases.
      \begin{enumerate}
        \item Let $x=1$, then $b=[1,y]\backslash 2$, $y\neq 3$, $3b=y$ (as $2b=1$) and $2a\neq 1$ (as $3a=1$).
            So $\M(b^2a)=\{2a,3\}$ and if $2a\neq 2$, then $|\M(b^2a^2)|=3$, else if $2a= 2$, then $|\M(b^2ab)|=3$, so either the word $b^2a^2$ or $b^2ab$ $3$-compresses $\A$.
        \item Let $y=1$, then $b=[x,1]\backslash 2$, $x\neq 3$, $3b=1$ (as $2b=x$) and $2a\neq 1$ (as $3a=1$).
            If $2a=2$, then the P3MSA in Fig. \ref{fig:prop:4-4:r:4:1} proves that the word $a^2b^2$ $3$-compresses $\A$.
            If $2a\neq 2$, then the P3MSA in Fig. \ref{fig:prop:4-4:r:4:2} proves that the word $a^2ba^2$ $3$-compresses $\A$.
            \begin{figure}[ht]%
                \centering
                \subfigure[][Case $2a=2$.]{%
                \label{fig:prop:4-4:r:4:1}%
                {\footnotesize
                \begin{tikzpicture}[shorten >=1pt,node distance=4.3em,auto]

                    \node[state] (q_0) {};
                    \node[state] (q_1) [right of= q_0] {$2$};
                    \node[state] (q_2) [right of= q_1] {$2,x$};
                    \node[state] (q_3) [right of= q_2] {$3,xa$};
                    \node[state,accepting] (q_4) [right of= q_3] {};

                    \path[->] (q_0) edge node {$b$} (q_1)
                              (q_1) edge node {$b$} (q_2)
                              (q_2) edge node {$a$} (q_3)
                              (q_3) edge node {$a$} (q_4)
                              (q_3) edge [bend left=10,draw=none] node {\phantom{$a|2a\neq 2$}} (q_4)
                              (q_3) edge [bend right=10,draw=none] node [swap] {\phantom{$b|2a= 2$}} (q_4);

                    \draw [->] (-.5,.5) -- (-.27,.27);

                \end{tikzpicture}}}
                \hspace{15pt}%
                \subfigure[][Case $2a\neq 2$.]{%
                \label{fig:prop:4-4:r:4:2}%
                {\footnotesize
                \begin{tikzpicture}[shorten >=1pt,node distance=4.3em,auto]

                    \node[state] (q_0) {};
                    \node[state] (q_1) [right of= q_0] {$1$};
                    \node[state] (q_2) [right of= q_1] {$1,3$};
                    \node[state] (q_3) [right of= q_2] {$1,2$};
                    \node[state] (q_4) [right of= q_3] {$3,2a$};
                    \node[state,accepting] (q_5) [right of= q_4] {};
                    \node[state] (q) [draw=none,right of= q_5] {};

                    \path[->] (q_0) edge node {$a$} (q_1)
                              (q_1) edge node {$a$} (q_2)
                              (q_2) edge node {$b$} (q_3)
                              (q_3) edge node {$a$} (q_4)
                              (q_4) edge node {$a$} (q_5);

                    \draw [->] (-.5,.5) -- (-.27,.27);

                \end{tikzpicture}}}\\
                \caption[]{P3MSA for the case $b={[x,1]}\backslash 2$, $x\neq 3$ and $3b\in\{x,1\}$.}%
                \label{fig:prop:4-4:r:4}%
            \end{figure}
      \end{enumerate}

  \item If $b=[x,3]\backslash 2$ and $xa\neq 2$ or $1b\neq 3$, then
      observe that $xa\neq 1$, as $3a=1$ and $x\neq 3$.
      If $xa\neq 2$, then $\M(b^2)=\{2,x\}$, $\M(b^2a)=\{3,xa\}$ and $|\M(b^2a^2)|=3$.
      If $xa=2$ and $1b\neq 3$, then $\M(a^2)=\{1,3\}$, $\M(a^2b)=\{1b,2\}$ and $|\M(a^2b^2)|=3$.
      Then either the word $b^2a^2$ or $a^2b^2$ $3$-compresses $\A$.
      If $b=[3,y]\backslash 2$ and $ya\neq 2$ or $1b\neq y$, then observe that $ya\neq 1$, as $3a=1$ and $y\neq 3$.
      The P3MSA in Fig. \ref{fig:prop:4-4:r:3} proves that either the word $a^2b^2$ or $a^2ba^2$ $3$-compresses $\A$.
      \begin{figure}[ht]%
        \centering
        {\footnotesize
        \begin{tikzpicture}[shorten >=1pt,node distance=5em,auto]

            \node[state] (q_0) {};
            \node[state] (q_1) [right of= q_0] {$3$};
            \node[state] (q_2) [right of= q_1] {$1,3$};
            \node[state] (q) [right of= q_2,draw=none] {};
            \node[state] (q_3) [right of= q] {$2,1b$};
            \node[state,accepting] (q_4) [right of= q_3] {};
            \node[state] (q_6) [right of= q_4] {$3,ya$};
            \node[state] (q_5) [right of= q_6] {$2,y$};

            \path[->] (q_0) edge node {$a$} (q_1)
                      (q_1) edge node {$a$} (q_2)
                      (q_2) edge node {$b|1b\neq y$} (q_3)
                      (q_3) edge node {$b$} (q_4)
                      (q_2) edge [bend right=18] node [swap] {$b|1b=y,ya\neq 2$} (q_5)
                      (q_5) edge node [swap] {$a$} (q_6)
                      (q_6) edge node [swap] {$a$} (q_4);

            \draw [->] (-.5,.5) -- (-.27,.27);

        \end{tikzpicture}}\\
        \caption[]{P3MSA for the case $b={[3,y]}\backslash 2$ and $ya\neq 2$ or $1b\neq y$.}%
        \label{fig:prop:4-4:r:3}%
      \end{figure}
  \item Let $b=[x,3]\backslash z$, $z\not\in\{1,2\}$, $x\in\{1,2\}$ and $za\in\{1,2\}$.
      Observe that $z\neq 3$ and $3a=1$, so $za\neq 1$, and then it is always $za=2$.
      We have to consider two subcases.
      \begin{enumerate}
        \item Let $x=1$, then $b=[1,3]\backslash z$, $zb=1$, $3b\neq 1$ and $\M(a^2b)=\{z,3b\}$.
            If $3b\neq 2$, then $\M(a^2ba)=\{za,3ba,3\}$ and $a^2ba$ $3$-compresses $\A$.
            Else, if $3b=2$, then $\M(a^2b^2)=\{1,z,2b\}$ and $a^2b^2$ $3$-compresses $\A$.
        \item Let $x=2$, then $b=[2,3]\backslash z$, $3b\neq 2$ (as $zb=2$) and $\M(a^2b)=\{z,1b\}$.
            If $1b\neq 3$, then $\M(a^2b^2)=\{2,1b^2,z\}$ and $a^2b^2$ $3$-compresses $\A$.
            Else, if $1b=3$, then $\M(a^2ba)=\{za,1ba,3\}$ and $a^2ba$ $3$-compresses $\A$.\end{enumerate}\end{enumerate}\end{proof}

\FloatBarrier 
\section{\texorpdfstring{Finding lower and upper bounds for $c(3,2)$}{Finding lower and upper bounds for c(3,2)}}\label{zot}

Collecting the words arising from the previous propositions, and taking into account that the roles of letters $a$ and $b$ are interchangeable, then each $3$-full word, containing as factors the words in the set $W$:\\

\begin{tabular}{ll}
      $W=$ & $\{ab^2ab^2a,\ ab^2a^2b^2a,\ abab^2aba,\ ab^3aba,\ abab^3a,\ ab^3ab^3a,\ ba^2baba^2b,\ a^2b^3a,\ ba^2ba^2b$\\
      & $\phantom{\{}ba^2b^2a^2b,\ baba^2bab,\ ba^3bab,\ baba^3b,\ ba^3ba^3b,\ ab^2abab^2a, b^2a^3b, a^2b^3a^2, b^2a^3b^2\}$
\end{tabular}
\ \\
is a $3$-collapsing word on a two letter alphabet. Remark that $a^2b^3a$ and $b^2a^3b$ are factors of $a^2b^3a^2$ and $b^2a^3b^2$ respectively, the reason for which they occur in $W$ will be clear in the sequel.

Then, in order to construct a short $3$-collasing word we need to find a word having as factor all the words above, \textit{i.e.}, to solve the Shortest Common Supersequence problem (SCS) for $W$. It is well-known that SCS is NP-complete (\cite{Raiha}) also over a two letter alphabet, thus approximation algorithm are often used. Nevertheless, the cost of finding a good approximation is comparable to the cost of finding an optimal solution, as in \cite{K} the authors prove that the problem to approximate within any constant approximation factor better than 333/332 is NP-hard. On the other hand, efficient algorithms give poor approximation, as the best one yields an approximation factor of $\approx$ 2, \textit{i.e.}, it produces a solution whose length is about twice that of the optimum (\cite{K,Turner}).

So, as no near-optimal solutions can be found in reasonable time, we decide to code the problem in the bounded satisfiability problem for a set of liner-time temporal logic (LTL) formulae, for which we developed a tool (\cite{io}).
More precisely, let $S$ be a propositional letter, a word $w$ of length $n$ is coded in the LTL formula $\mathbf{w}$ such that it is satisfied if and only if for all $1\leq i\leq n$ the $i$-th letter of $w$ is \lq\lq a\rq\rq\ if and only if at the $i$-th time instant $S$ is true. E.g., the word $aba$ is encoded in the formula $S\wedge (\mathbf{X}(\neg S \wedge \mathbf{X} (S)))$, where $\mathbf{X}$ is the \lq\lq next\rq\rq\ operator. Then we look for the shortest model that satisfy the formula $\bigwedge_{w\in W}\mathbf{w}$, obtaining that the shortest word having as factor all the words in $W$ has length 55.

However, we were well aware that such \lq\lq greedy approach\rq\rq, \emph{i.e.}, to find an optimal solution for each subcase and combining them to obtain a global solution, is not suitable in order to achieve a global optimum.

Then we observed that the words $a^2b^3a^2$ and $b^2a^3b^2$ are needed only to solve a special subcase of $(\bf{3},\bf{4})$-automata, so we try to replace them with longest factor in order to obtain a shorter $3-$collapsing word.

Actually, the shortest word having as factor the words in $W\setminus\{a^2b^3a^2,b^2a^3b^2\}$ is
\[
w_3=\ b^2a^3b\underbrace{a^3b^3aba^2}_{v}baba^2ba^2b^2a^2b^2ab^2aba\underbrace{b^2aba^3bab^3}_{u}ab^3a
\]
has as factor the word $u=b^2aba^3bab^2$ and $v=a^2b^3aba^2$. As $u$ belongs to the language $\mathcal{L}$ defined in Proposition \ref{proposition:3-4}, case $2.b.iii$, and $v$ belongs to the dual of $\mathcal{L}$, this proves that $w_3$ is a $3-$collapsing word of length 53.

It is known that in general the language $\mathcal{C}_{k,t}$
of $k$-collapsing words on an alphabet of $t$ letters differs
from the language $\mathcal{S}_{k,t}$ of $k$-synchronizing words on the same alphabet. However, this not excludes that in some cases $c(k,t)=s(k,t)$. Up to now it was only know that $c(2,2)=s(2,2)$ (\cite{SS}) and that $c(2,3)\neq s(2,3)$ (\cite{AP}).
We find a counterexample proving that $c(3,2)\neq s(3,2)$ (and so $c(3,2)\geq 34$). In fact, the automaton in Fig. \ref{automa} is 3-compressible (and also 3-synchronizing), but the word $s_{3,2}$ do not compresses it. On the other hand, its dual $\bar{s}_{3,2}$ synchronizes it. We check with a computer calculation that any 3-compressible 5-states automaton over a two letter alphabet, is 3-compressed either by $s_{3,2}$ or by $\bar{s}_{3,2}$,
so we believe interesting to understand if it is always the case for every 3-compressible automata.
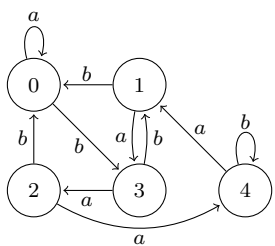
\begin{figure}[ht]%
\centering
{\footnotesize
\begin{tikzpicture}[shorten >=1pt,node distance=5em,auto]

    \node[state] (q_0) {$0$};
    \node[state] (q_1) [right of= q_0] {$1$};
    \node[state] (q_2) [below of= q_0] {$2$};
    \node[state] (q_3) [right of= q_2] {$3$};
    \node[state] (q_4) [right of= q_3] {$4$};

    \path[->] (q_0) edge node [swap] {$b$} (q_3)
              (q_0) edge [loop above] node {$a$} (q_0)
              (q_1) edge node [swap] {$b$} (q_0)
              (q_1) edge [bend right=10] node [swap] {$a$} (q_3)
              (q_2) edge node {$b$} (q_0)
              (q_2) edge [bend right] node [swap] {$a$} (q_4)
              (q_3) edge [bend right=10] node [swap] {$b$} (q_1)
              (q_3) edge node {$a$} (q_2)
              (q_4) edge node [swap] {$a$} (q_1)
              (q_4) edge [loop above] node {$b$} (q_4);

\end{tikzpicture}}
\caption{A semiautomaton which is not 3-compressed by $s_{3,2}$: $Q\bar{s}_{3,2}=\{3\}$, $Qs_{3,2}=\{0,1,3\}$.}
\label{automa}
\end{figure}

\section{Conclusion}\label{conclusion}

Although very technical, our analysis can be effectively exploited in order to obtain more general results and to investigate some conjectures. In \cite{kisi2, kisi}, the authors exploit the characterization of $(\bf{3},\bf{p})$-automata (Proposition \ref{proposition:3-p}) to prove that the problem of recognizing whether a binary word is $3$-collapsing is co-\textsc{NP}-complete.

Moreover, the word $w_3$ can be used to improve the procedure arising from \cite{MVP} (Theorem 3.5) to obtain shorter $k$-collapsing words for $k\geq4$.
In particular it follows that $c(4,2)\leq 1741$ and $c(5,2)\leq 109941$.
Though very lengthy, they can be effectively used in testing the compressibility of an automaton. In particular this can accelerate the algorithm presented in \cite{AP, Pet} to find short (possibly shortest) 4- and 5-synchronizing words. 

\bibliographystyle{abbrvnat}
\bibliography{bibliografia}

\end{document}